\documentclass{amsart}
\usepackage{amsmath, amscd}
\usepackage{stmaryrd}
\usepackage{mathrsfs}
\usepackage{amssymb}
\usepackage{amsfonts}
\usepackage{tensor}
\usepackage[all]{xy}
\usepackage{enumerate}
\usepackage{tikz}
\usepackage{adjustbox}
\usepackage{graphicx}
\usetikzlibrary{matrix, arrows}
\usetikzlibrary{positioning}
\usetikzlibrary{decorations.pathreplacing}
\numberwithin{equation}{subsection}
\theoremstyle{plain}
\newtheorem{thm}[subsection]{Theorem}
\newtheorem{prop}[subsection]{Proposition}
\newtheorem{lemma}[subsection]{Lemma}
\newtheorem{cor}[subsection]{Corollary}
\theoremstyle{definition}
\newtheorem{defn}[subsection]{Definition}

\theoremstyle{remark}
\newtheorem{rem}[subsection]{Remark}

\newtheorem{final remark}[subsection]{Final Remark}
\makeatletter
\newcommand*\bigcdot{\mathpalette\bigcdot@{.65}}
\newcommand*\bigcdot@[2]{\mathbin{\vcenter{\hbox{\scalebox{#2}{$\m@th#1\bullet$}}}}}
\makeatother
\begin{document}
\title[A log-motivic cohomology for semistable varieties]{A log-motivic cohomology for semistable varieties and its $p$-adic deformation theory}
\author{Oliver Gregory and Andreas Langer}
\begin{abstract}
We construct log-motivic cohomology groups for semistable varieties and study the $p$-adic deformation theory of log-motivic cohomology classes. Our main result is the deformational part of a $p$-adic variational Hodge conjecture for varieties with semistable reduction: a rational log-motivic cohomology class in bidegree $(2n,n)$ lifts to a continuous pro-class if and only if its Hyodo-Kato class lies in the $n$-th step of the Hodge filtration. This generalises \cite[Theorem 8.5]{BEK14} which treats the good reduction case. In the case $n=1$ the lifting criterion is the one obtained by Yamashita for the logarithmic Picard group \cite[Theorem 3.1]{Yam11}. Along the way, we relate log-motivic cohomology to logarithmic Milnor $K$-theory and the logarithmic Hyodo-Kato Hodge-Witt sheaves.  
\end{abstract}
\address{25 Gordon Street, University College London, London, UK, WC1H 0AY}
\email {o.gregory@ucl.ac.uk}
\address{Harrison Building, University of Exeter, Exeter, EX4 4QF, Devon, UK}
\email{a.langer@exeter.ac.uk}
\date{December 10, 2025 \\ This research was supported by EPSRC grant EP/T005351/1}
\maketitle
\pagestyle{myheadings}

\section{Introduction}
In the present work we construct a variant $\mathbb{Z}_{\log}(n)$ of the motivic complexes  of Suslin-Voevodsky \cite{SV00a} suitable for semistable varieties. Our approach relies on a definition of finite correspondences due to Suslin-Voevodsky \cite[\S3]{SV00b} which also includes singular varieties (see also \cite[Appendix 1A]{MVW06} and \cite[\S8 and \S9]{CD19}). Then the complexes $\mathbb{Z}_{\log}(n)$ are defined analogously as simplicial sheaves associated to a certain sheaf with transfers and coincides with the usual motivic complexes on the smooth locus.

In the case $n=1$, in order to get a geometric interpretation as $\mathbb{Z}_{\log}(1)$, we will modify the given logarithmic structure $M$ on the semistable variety and define a log-structure $N$ which is trivial on the smooth locus of the variety. By considering the image $\underline{N}^{\mathrm{gp}}$ of $N^{\mathrm{gp}}$ under the structure morphism, we can then relate its first cohomology to the diagonal log-motivic cohomology. Then we define logarithmic Milnor $K$-groups by applying the Milnor functor to the group $\underline{N}^{\mathrm{gp}}$, and prove that the corresponding sheaf is the cohomology sheaf $\mathcal{H}^{n}(\mathbb{Z}_{\log}(n))$, in analogy to the smooth case which was proved by Kerz \cite{Ker09}. We also relate the modulo $p^{n}$ residue of the log-Milnor $K$-group to modified logarithmic Hyodo-Kato Hodge-Witt sheaves, making precise an old result of Hyodo \cite{Hyo88}. 

Let $k$ be a perfect field of characteristic $p>0$, and let $K=\mathrm{Frac}\,W(k)$. Let $X$ be a $W(k)$-scheme with semistable reduction, with special fibre $Y$ and generic fibre $X_K$. For each $m\in\mathbb{N}$, let $X_{m}$ be the reduction of $X$ modulo $p^{m}$, so $X_{1}=Y$. Our main motivation comes from the problem of constructing $K$-cohomology classes (or cycles) on $X_{K}$. One strategy is to attempt to lift classes from the special fibre (this strategy is especially appealing if the reduction $Y$ is highly degenerate and thus has an abundance of easily accessible cycles). In the second half of the paper we state and prove an analogue of the $p$-adic variational Hodge conjecture \cite{BEK14} for semistable varieties, which yields a lifting criterion for motivic cohomology classes in terms of their logarithmic Chern class in Hyodo-Kato cohomology. For $n=1$, the lifting criterion coincides with the lifting criterion of the logarithmic Picard group considered by Yamashita \cite{Yam11}. It general it uses a semistable version of the pro-complexes $\mathbb{Z}_{X_{\bigcdot}}(n)$ of Bloch-Esnault-Kerz by gluing the complexes $\mathbb{Z}_{\log,Y}(n)$ and the log-syntomic complex of Kato-Tsuji along the modified logarithmic Hyodo-Kato de Rham-Witt sheaf. We will also use a construction of the log-syntomic complex by Nekov\'{a}\v{r}-Nizio\l \ \cite{NN16}. The pro-complexes $\mathbb{Z}_{\log,X_{\bigcdot}}(n)$ enjoy some of the nice properties of the pro-complexes $\mathbb{Z}_{X_{\bigcdot}}(n)$ in the smooth case stated in \cite[\S7]{BEK14}. For example, the top cohomology sheaf is the log-Milnor $K$-group as pro-sheaf $\mathcal{K}_{\log,X_{\bigcdot},n}^{\mathrm{Mil}}$ and, at least rationally, it is an extension of the log-motivic complex $\mathbb{Z}_{\log,Y}(n)$ by a truncated de Rham complex. In contrast to the smooth case (\cite[Proposition 7.3]{BEK14}) we do not expect this property to hold integrally since an integral version of the Hyodo-Kato comparison is not known. 
 
For $n<p$, we denote by $\mathbb{H}^{2n}_{\mathrm{cont}}(Y,\mathbb{Z}_{\log,X_{\bigcdot}}(n))$ the continuous logarithmic Chow group of $X_{\bigcdot}$. Our main result, which is a generalisation of \cite[Theorem 8.5]{BEK14} to the case of semistable reduction, can then be formulated as follows:

\begin{thm}\label{theorem intro}(= Theorem \ref{main theorem})
Let $n<p$. Let $X$ be a proper flat scheme over $\mathrm{Spec}\,W(k)$ with semistable reduction. Let $z\in\mathbb{H}^{2n}(Y,\mathbb{Z}_{\log,Y}(n))\otimes\mathbb{Q}$. Then its log-crystalline Chern class $c_{\mathrm{HK}}(z)\in H^{n}(Y,W_{\bigcdot}\omega_{Y/k,\log}^{n})\otimes\mathbb{Q}\rightarrow H_{\mathrm{log-cris}}^{2n}(Y/W(k))_{\mathbb{Q}}\simeq H_{\mathrm{dR}}^{2n}(X/W(k))_{\mathbb{Q}}\simeq H_{\mathrm{dR}}^{2n}(X_{K}/K)$ lies in $\mathrm{Fil}^{n}H_{\mathrm{dR}}^{2n}(X_{K}/K)$ if and only if $z$ lifts to $\hat{z}\in\mathbb{H}_{\mathrm{cont}}^{2n}(Y,\mathbb{Z}_{\log,X_{\bigcdot}}(n))\otimes\mathbb{Q}$.
\end{thm}

Note that we construct a map $\mathbb{Z}_{\log,Y}(n)\rightarrow\mathcal{K}_{\log,Y,n}^{\mathrm{Mil}}[-n]$ which induces a homomorphism 
\begin{equation*}
\pi_{n}:\mathbb{H}^{2n}(Y,\mathbb{Z}_{\log,Y}(n))\rightarrow H_{\mathrm{Zar}}^{n}(Y,\mathcal{K}^{\mathrm{Mil}}_{\log,Y,n})\,.
\end{equation*}
Cast in terms of Milnor $K$-sheaves, Theorem \ref{theorem intro} tells us that if a rational log-Milnor $K$-cohomology class on $Y$ is in the image of $\pi_{n}\otimes\mathbb{Q}$ and is Hodge (its log-crystalline Chern class lies in $\mathrm{Fil}^{n}H_{\mathrm{dR}}^{2n}(X_{K}/K)$), then it at least ``formally'' lifts to an element of $\varprojlim_{m} H^{n}(X_{m},\mathcal{K}_{\log,X_{m},n}^{\mathrm{Mil}})\otimes\mathbb{Q}$. 
\begin{rem}
\
\begin{enumerate}[(i)]
\item In the good reduction case, the main theorem \cite[Theorem 1.3]{BEK14} is concerned with deforming classes of vector bundles and its proof has two parts. The first is \cite[Theorem 8.5]{BEK14} which concerns lifting algebraic cycle classes to the continuous Chow group -- our Theorem \ref{theorem intro} is a generalisation of this to the semistable reduction case. The second part is the Chern character comparison isomorphism \cite[Theorem 11.1]{BEK14} between continuous $K$-theory and continuous Chow groups; here a restriction on the dimension of the special fibre is needed. We do not give a semistable analogue of this second theorem here, but we consider it to be an interesting problem to investigate the relationship between our logarithmic Chow groups and a logarithmic incarnation of $K$-theory for log-smooth schemes.
\item We are aware that the assumptions of unramifiedness and on the dimension in \cite[Theorem 1.3]{BEK14} have been removed in the recent work \cite[Theorem D]{AMMN20}, which uses topological cyclic homology as a new suitable tool in $p$-adic deformation theory. They also prove in \cite[Theorem E]{AMMN20} a more general result on lifting classes in higher $K$-theory to continuous $K$-theory by using $p$-adic derived de Rham cohomology. The subject of this note is different in the sense that we lift log-motivic cohomology classes by considering their Chern classes in log-crystalline cohomology. 
\end{enumerate}
\end{rem}
   
Finally, let us point out that motives and motivic complexes have been constructed for singular varieties in a series of papers, notably by Kahn-Miyazaki-Saito-Yamazaki \cite{KMSY21a}, \cite{KMSY21b}, \cite{KMSY21c} and Binda-Park-\O stv\ae r \cite{BPO20}. In the ``Motives with modulus'' series, Kahn-Miyazaki-Saito-Yamazaki construct a triangulated tensor category of motives with modulus $\mathrm{\bf{MDM}}_{\mathrm{gm}}^{\mathrm{eff}}$ in the same way as Voevodsky constructed his category $DM_{gm}^{eff}$ in \cite{Voe00}, starting from the category $\mathrm{Cor}_{k}$ of smooth varieties with finite correspondences as morphisms. A motive with modulus is a pair $(M,M^{\infty})$ where $M$ is a $k$-variety and $M^{\infty}$ is an effective Cartier divisor on $M$ such that $M-M^{\infty}$ is smooth. The category $\mathrm{Cor}_{k}$ is replaced by $\mathrm{MCor}_{k}$ of finite correspondences between $M-M^{\infty}$ and $N-N^{\infty}$ (for two modulus pairs $(M,M^{\infty})$, $(N,N^{\infty})$) that satisfy a certain condition on the Cartier divisors. One of their main results is a characterisation of Bloch's higher Chow groups and Voevodsky's motivic cohomology in terms of a derived internal Hom between two motives with modulus in $\mathrm{\bf{MDM}}_{\mathrm{gm}}^{\mathrm{eff}}$. A crucial difference to the construction of Voevodksy is that $\mathbb{A}^{1}$-invariance is replaced by $\overline{\Box}$-invariance, where $\overline{\Box}=(\mathbb{P}^{1},\infty)$ is the motive with modulus where $\infty$ is the reduced divisor on $\mathbb{P}^{1}$ at $\infty$.  
The theory of Kahn-Miyazaki-Saito-Yamazaki is then extended and translated into the language of logarithmic geometry by Binda-Park-\O stv\ae r. In fact, they construct a triangulated tensor category $\mathrm{\bf{logDM}}^{\mathrm{eff}}(k)$ of effective log-motives starting from the category $lSM/k$ of fine and saturated (fs) log-schemes that are log-smooth over $\mathrm{Spec}\,k$ equipped with the trivial log-structure, and where the category $\mathrm{MCor}_{k}$ is replaces by the category $lCor/k$ of finite log-correspondences. Any fs log-scheme $X\in lSm/k$ gives rise to a log-motive $M(X)\in\mathrm{\bf{logDM}}^{\mathrm{eff}}(k)$. Their construction generalises Voevodsky's category of effective motives. For example, if $X$ and $Y$ are fs log-schemes in $lSm/k$ such that $X-\partial X$ and $Y-\partial Y$ are smooth subschemes, where the log-structure is trivial, then
\begin{equation*}
Hom_{\mathrm{\bf{logDM}}^{\mathrm{eff}}(k)}(M(Y)[i],M(X))\cong Hom_{DM^{eff}}(M(Y-\partial Y)[i],M(X-\partial X))\,. 
\end{equation*}
In both works, the main example is the motive associated to a toroidal embedding $j:U\hookrightarrow X$ of a smooth $k$-variety into a normal variety $X$, with $M$ the log-structure defined by $\mathcal{O}_{X}\cap j_{\ast}\mathcal{O}_{U}^{\ast}$. In the present paper, we consider the category $\mathrm{SemiStab}_{k}$ of semistable varieties. These are normal crossing divisors inside $W(k)$-schemes that are
log-syntomic, but not log-smooth, over $\mathrm{Spec}\,W(k)$ equipped with the trivial log-structure. In analogy to \cite{BPO20}, we define a category $\mathrm{SemiStabCor}_{k}$ with objects the semistable varieties and morphisms finite log-correspondences. This leads to the notion of sheaves with transfer and allows us to define the log-motivic complexes $\mathbb{Z}_{\log}(r)$ in an ad-hoc fashion using the simplicial approach \cite{SV00a}. We hope to construct, in a future project, a derived category $\mathscr{C}$ of effective log-motives such that a semistable variety $X$ gives rise to a log-motive $M(X)$ in $\mathscr{C}$, complimentary to the works of Binda-Park-\O stv\ae r and Kahn-Miyazaki-Saito-Yamazaki.
  
\subsection{Conventions}
All schemes are assumed to be separated and of finite type over the base.

\section{Log-motivic cohomology}\label{log-motivic cohomology section}

\subsection{Finite log-correspondences}\label{finite log section}
\ \par

For a morphism of fine log-schemes $f:(X,M_{X})\rightarrow(Y,M_Y)$, the strict locus of $f$ is the locus of points $x\in X$ such that $(f^{\ast}M_{Y})_{\overline{x}}\xrightarrow{\sim} M_{X,\overline{x}}$. We shall abusively write the strict locus of $f$ simply as $X^{\mathrm{str}}$ without reference to $f$, since the morphism will always be clear from the context (it will be the structure morphism). Note that $X^{\mathrm{str}}\subset X$ is open by \cite[Proposition 2.3.1]{Shi00}. If the base $Y$ has trivial log-structure then $X^{\mathrm{str}}$ coincides with the trivial locus of $X$, denoted by $X^{\mathrm{triv}}$.

Recall that if $(X,M_{X})\rightarrow (B,M_{B})$ and $(Y,M_{Y})\rightarrow (B,M_{B})$ are morphisms of fs (fine and saturated) log-schemes, then the log-structure on the fibre product $(X,M_{X})\times_{(B,M_{B})}(Y,M_{Y})$ taken in the category of log-schemes is coherent but not necessarily fs. Instead, we may take the fibre product in the category of fs log-schemes, which we denote by $(X,M_{X})\times^{\mathrm{fs}}_{(B,M_{B})}(Y,M_{Y})$. Note that the underlying scheme of $(X,M_{X})\times_{(B,M_{B})}(Y,M_{Y})$ is $X\times_{B}Y$, but this is not case for  $(X,M_{X})\times^{\mathrm{fs}}_{(B,M_{B})}(Y,M_{Y})$ in general. There is however a natural morphism
\begin{equation*}
(X,M_{X})\times^{\mathrm{fs}}_{(B,M_{B})}(Y,M_{Y})\rightarrow (X,M_{X})\times_{(B,M_{B})}(Y,M_{Y})
\end{equation*}
which is a finite morphism on the underlying schemes \cite[Remark 12.2.36(i)]{GR18}, and is an isomorphism  over the trivial locus $(X\times_{B}\times Y)^{\mathrm{triv}}$. 

Let $(\mathrm{Spec}\,k,L)$ be the standard log-point, i.e. $L$ is the log-structure on $\mathrm{Spec}\,k$ associated to $\mathbb{N}\rightarrow k, 1\mapsto 0$. We shall abusively write $\mathrm{Spec}\,k$ denote to denote the log-scheme whose underlying scheme is $\mathrm{Spec}\,k$, and the log-structure is the trivial log-structure.

\begin{defn}
An fs log-scheme $(X,M_{X})$ over $(\mathrm{Spec}\,k,L)$ is called a semistable variety if \'{e}tale locally on $X$ the structure morphism $(X,M_{X})\rightarrow(\mathrm{Spec}\,k,L)$ factors as
\begin{equation*}
(X,M_{X})\xrightarrow{u}(\mathrm{Spec}\,k[T_{1},\ldots,T_{a}]/(T_{1}\cdots T_{b}),P)\xrightarrow{\delta}(\mathrm{Spec}\,k,L)
\end{equation*}
for some $a\geq b$, where $P$ is the log-structure associated to $\mathbb{N}^{b}\rightarrow k[T_{1},\ldots,T_{a}]/(T_{1}\cdots T_{b})$, $e_{i}\mapsto T_{i}$, where $u$ is strict and \'{e}tale, and $\delta$ is the morphism induced by the diagonal. 
\end{defn}

\begin{defn}\label{alternative log-structure definition}
In the following our base field $k$ is equipped with the trivial log-structure. For a semistable variety $(X,M_{X})$ we will use the log-structure $M_{X}$ to define an alternative log-structure $N_{X}$ on $X$ which will be very important in this paper. Let $U=\mathrm{Spec}\,A\subset X$ be an affine and let the structure morphism $\alpha:M_{X}\rightarrow\mathcal{O}_{X}$ be locally defined on $U$ by the homomorphism of monoids $\mathbb{N}^{r}\rightarrow\mathcal{O}(U)=A$, $e_{i}\mapsto\pi_{i}$. We define a new log-structure $N_{X}$ locally on $U$ by the homomorphism 
\begin{equation*}
\beta:\mathbb{N}^{r}\rightarrow A, \ \ \ e_{i}\mapsto g_{i}:=\pi_{i}+\displaystyle\prod_{\stackrel{j=1}{j\neq i}}^{r}\pi_{j}\,.
\end{equation*}
It is easy to see that $g_{i}\in\mathcal{O}(U)\cap j_{\ast}\mathcal{O}(U^{\mathrm{sm}})^{\ast}$ where $j:U^{\mathrm{sm}}\hookrightarrow U$ is the open immersion of the smooth part. Then, evidently, $X^{\mathrm{triv}}=X^{\mathrm{sm}}$, and we get a homomorphism of sheaves of monoids 
\begin{equation*}
\beta:N_{X}\rightarrow\mathcal{O}_{X}\cap j_{\ast}\mathcal{O}_{X^{\mathrm{sm}}}^{\ast}
\end{equation*}
where $\mathcal{O}_{X}\cap j_{\ast}\mathcal{O}_{X^{\mathrm{sm}}}^{\ast}$ is considered as a sheaf of monoids with respect to multiplication. Let 
\begin{equation*}
\beta^{\mathrm{gp}}:N_{X}^{\mathrm{gp}}\rightarrow(\mathcal{O}_{X}\cap j_{\ast}\mathcal{O}_{X^{\mathrm{sm}}}^{\ast})^{\mathrm{gp}}
\end{equation*}
be the associated homomorphism of sheaves of abelian groups and let $\underline{N}_{X}=\mathrm{Im}(\beta)$ and $\underline{N}_{X}^{\mathrm{gp}}=\mathrm{Im}(\beta^{\mathrm{gp}})\subset(\mathcal{O}_{X}\cap j_{\ast}\mathcal{O}_{X^{\mathrm{sm}}}^{\ast})^{\mathrm{gp}}$. Write $\mathrm{SemiStab}_{k}$ for the category of semistable varieties equipped with the log-structure $N_{X}$. We will consider $(X,N_{X})$ as a log-scheme over $\mathrm{Spec}\,k$ (equipped with the trivial log-structure). Note that $\mathrm{SemiStab}_{k}$ consists of objects which are not log-smooth over $k$.
\end{defn}

We wish to enlarge $\mathrm{SemiStab}_{k}$ into an additive category $\mathrm{SemiStabCor}_{k}$ by including the notion of finite log-correspondence, analogously to the classical smooth setting of Suslin-Voevodsky \cite[\S1]{SV00a}, \cite[Lecture 1]{MVW06}. 

\begin{defn}
Let $(X,N_{X})$ be an object of $\mathrm{SemiStab}_{k}$ and let $(Y,M_{Y})$ be any fs log-scheme over $\mathrm{Spec}\,k$. A finite log-correspondence from $(X,N_{X})$ to $(Y,M_{Y})$ is a finite correspondence $Z\in\mathrm{Cor}(X,Y)$ (see \cite[\S3]{SV00b}, \cite[Appendix 1A]{MVW06} and \cite[\S8 and \S9]{CD19} for finite correspondences between possibly singular schemes), such that the restriction $Z_{X^{\mathrm{triv}}}$ of $Z$ to $X^{\mathrm{triv}}\times_{k} Y$ has support in $X^{\mathrm{triv}}\times_{k} Y^{\mathrm{triv}}$. The group of finite log-correspondences from $(X,N_{X})$ to $(Y,M_{Y})$ is denoted by $\mathrm{Cor}((X,N_{X}),(Y,M_{Y}))^{\ast}$, or simply $\mathrm{Cor}(X,Y)^{\ast}$ when the log-structures are clear from the context.
\end{defn}

For example, let $f:(X,N_{X})\rightarrow (Y,M_{Y})$ be a morphism of fs log-schemes over $\mathrm{Spec}\,k$ where $(X,N_{X})$ is an object of $\mathrm{SemiStab}_{k}$. Let $\Gamma_{f}$ be the graph of the underlying morphism $f:X\rightarrow Y$. Then $\Gamma_{f}\subset X\times_{k}Y$ is closed because $Y$ is separated over $\mathrm{Spec}\,k$. Moreover, the projection $\mathrm{pr}_{X}:\Gamma_{f}\rightarrow X$ is an isomorphism, so $\Gamma_{f}$ is a universally integral relative cycle by \cite[Theorem 1A.6 \& Theorem 1A.10]{MVW06}, and hence $\Gamma_{f}\in\mathrm{Cor}(X,Y)$. By \cite[III. Proposition 1.2.8]{Ogu18} and \cite[Proposition 2.3.1]{Shi00}, we have $X^{\mathrm{triv}}$ and $Y^{\mathrm{triv}}$ are open in $X$ and $Y$ and we have $f(X^{\mathrm{triv}})\subset Y^{\mathrm{triv}}$. Hence $\Gamma_{f}$ restricted to $X^{\mathrm{triv}}\times_{k}Y$ has support in $X^{\mathrm{triv}}\times_{k}Y^{\mathrm{triv}}$. 

\begin{rem}
The definition of finite log-correspondence makes sense in much greater generality, but we only ever use it for log-schemes in $\mathrm{SemiStab}_{k}$.
\end{rem}

Let $(X,N_{X}), (Y,N_{Y}), (Z,N_{Z})$ be objects of $\mathrm{SemiStab}_{k}$, and let $V\in\mathrm{Cor}(X,Y)^{\ast}$, $W\in\mathrm{Cor}(Y,Z)^{\ast}$. Let $W\circ V\in\mathrm{Cor}(X,Z)$ be the composition of $V$ and $W$ as defined in \cite[Definition 1A.11]{MVW06}, so $W\circ V$ is the pushforward of $W_{V}$ along the projection $X\times_{k} Y\times_{k}Z\rightarrow X\times_{k}Z$, where $W_{V}$ is the relative cycle given by pulling back $W$ along the map $V\rightarrow Y$ \cite[Theorem 1A.8]{MVW06}. Since the restriction of $W_{V}$ to $X^{\mathrm{triv}}\times_{k}Y^{\mathrm{triv}}\times_{k}Z$ is the relative cycle $(W_{X^{\mathrm{triv}}})_{V_{Y^{\mathrm{triv}}}}$, we have that $(W\circ V)_{Y^{\mathrm{triv}}}=W_{X^{\mathrm{triv}}}\circ V_{Y^{\mathrm{triv}}}$, and $W_{X^{\mathrm{triv}}}\circ V_{Y^{\mathrm{triv}}}\in\mathrm{Cor}(X^{\mathrm{triv}},Y^{\mathrm{triv}})$ because $W$ and $V$ are finite log-correspondences. The composition of finite correspondences therefore gives a well-defined composition
\begin{align*}
\mathrm{Cor}(X,Y)^{\ast}\times& \mathrm{Cor}(Y,Z)^{\ast}\rightarrow\mathrm{Cor}(X,Z)^{\ast} \\
& (V,W)\mapsto W\circ V
\end{align*}
for finite log-correspondences. If $f:(X,N_{X})\rightarrow (Y,N_{Y})$ is a morphism, then $W\circ\Gamma_{f}$ is the relative cycle $W_{X}$. In particular, $id_{X}:=\Gamma_{\mathrm{id}}\in\mathrm{Cor}(X,X)^{\ast}$ is the identity with respect to composition.

\begin{defn}
Let $\mathrm{SemiStabCor}_{k}$ be the category whose objects are the same as those of $\mathrm{SemiStab}_{k}$ and whose morphisms from $(X,N_X)$ to $(Y,N_Y)$ are the elements of $\mathrm{Cor}(X,Y)^{\ast}$. 
\end{defn}

Then $\mathrm{SemiStabCor}_{k}$ is an additive category and there is a faithful functor $\mathrm{SemiStab}_{k}\rightarrow\mathrm{SemiStabCor}_{k}$ given by
\begin{equation*}
(X,N_{X})\mapsto (X,N_{X}), \ \ \ \ (f:(X,N_{X})\rightarrow (Y,N_{Y}))\mapsto\Gamma_{f}\,.
\end{equation*}

\begin{defn}
A presheaf with transfers is a contravariant additive functor $F:\mathrm{SemiStabCor}_{k}\rightarrow\mathrm{Ab}$ to the category of abelian groups.
\end{defn}

An important source of presheaves with transfers is as follows. If $(Y,M_{Y})$ is an fs log-scheme over $\mathrm{Spec}\,k$, the presheaf
\begin{align*}
\mathbb{Z}_{\mathrm{tr}}(Y)^{\ast}\,:\,
& \,\mathrm{SemiStab}_{k}\rightarrow\mathrm{Ab} \\
& \, (X,N_{X})\mapsto\mathrm{Cor}(X,Y)^{\ast} 
\end{align*}
is a presheaf with transfers by virtue of the composition product of finite log-correspondences.

We shall say that a presheaf of abelian groups $F:\mathrm{SemiStab}_{k}\rightarrow\mathrm{Ab}$ is a Zariski sheaf if the restriction of $F$ to each $(X,N_{X})$ in $\mathrm{SemiStab}_{k}$ is a Zariski sheaf on $X$. That is, if $i_{1}:(U_{1},N_{U_{1}})\hookrightarrow (X,N_{X})$ and $i_{2}:(U_{2},N_{U_{2}})\hookrightarrow (X,N_{X})$ are open immersions such that $X=U_{1}\cup U_{2}$, then the sequence
\begin{equation*}
\resizebox{1.0\hsize}{!}{$
0\rightarrow F(X,N_{X})\xrightarrow{\mathrm{diag}} F\left((U_{1},N_{U_{1}})\amalg(U_{2},N_{U_{2}})\right)\xrightarrow{(+,-)} F\left((U_{1},N_{U_{1}})\times^{fs}_{(X,M_{X})}(U_{2},N_{U_{2}})\right)
$}
\end{equation*}
is exact. Notice that the underlying scheme of $(U_{1},N_{U_{1}})\times^{fs}_{(X,N_{X})}(U_{2},N_{U_{2}})$ is $U_{1}\cap U_{2}$ because $i_{1}$ and $i_{2}$ are strict.

\begin{lemma}\label{sheaf}
Let $(Y,M_{Y})$ be an fs log-scheme over $\mathrm{Spec}\,k$. Then $\mathbb{Z}_{\mathrm{tr}}(Y)^{\ast}$ is a Zariski sheaf. In particular, $C_{\ast}(\mathbb{Z}_{\mathrm{tr}}(Y)^{\ast})$ is a chain complex of Zariski sheaves, where $C_{\ast}(-)$ is the simplicial construction given in \cite[\S0]{SV00a} and \cite[\S2]{MVW06}.
\end{lemma}
\begin{proof}
Let $(X,N_{X})$ be an object of $\mathrm{SemiStab}_{k}$ and let $(U_{1},N_{U_{1}})\hookrightarrow (X,N_{X})$, $(U_{2},N_{U_{2}})\hookrightarrow (X,N_{X})$ be open immersions such that $X=U_{1}\cup U_{2}$. The map $\mathbb{Z}_{\mathrm{tr}}(Y)^{\ast}(X,N_{X})\xrightarrow{\mathrm{diag}}\mathbb{Z}_{\mathrm{tr}}(Y)^{\ast}((U_{1},N_{U_{1}})\amalg(U_{2},N_{U_{2}}))$ is the pullback of cycles along the surjective morphism $(U_{1}\amalg U_{2})\times Y\rightarrow X\times Y$, and is therefore injective. To see that $\mathbb{Z}_{\mathrm{tr}}(Y)^{\ast}$ is a Zariski sheaf, it remains to show  that if $Z_{1}$ and $Z_{2}$ are finite log-correspondences from $(U_{1},N_{U_{1}})$ (resp. $(U_{2},N_{U_{2}})$) to $(Y,M_{Y})$ that coincide on $(U_{1}\cap U_{2})\times Y$, then there is a finite log-correspondence $Z$ from $(X,N_{X})$ to $(Y,M_{Y})$ whose restriction to $U_{i}\times Y$ is $Z_{i}$ for each $i=1,2$. By definition, $Z_{1}=\sum_{j=1}^{s}\lambda_{j}Z_{1,j}$ and $Z_{2}=\sum_{j=1}^{t}\mu_{j}Z_{2,j}$ are finite linear combinations, where the $Z_{1,j}$ (resp. $Z_{2,j}$) are universally integral relative cycles of $U_{1}\times Y$ (resp. $U_{2}\times Y$) which are finite and surjective over $U_{1}$ (resp. over $U_{2}$). For each $i=1,2$, let $\iota_{i}:(U_{1}\cap U_{2})\times Y\hookrightarrow U_{i}\times Y$ be the obvious open immersion. Then, by assumption, we have
\begin{equation*}
\sum_{j=1}^{s}\lambda_{j}\iota_{1}^{-1}(Z_{1,j})=\sum_{j=1}^{t}\mu_{j}\iota_{2}^{-1}(Z_{2,j})\,.
\end{equation*}
We see then that $s=t$. Re-labelling, we may assume that $\lambda_{j}=\mu_{j}$ and $\iota_{1}^{-1}(Z_{1,j})=\iota_{2}^{-1}(Z_{2,j})$ for all $j=1,\ldots, s$. But then the cycle
\begin{equation*}
Z=\sum_{j=1}^{s}\lambda_{j}(Z_{1,j}\cup Z_{2,j})
\end{equation*} 
is a finite correspondence from $X$ to $Y$ whose restriction to $U_{i}\times Y$ is $Z_{i}$ for each $i=1,2$. Moreover, it is clear that $Z$ is a finite log-correspondence. This proves that $\mathbb{Z}_{\mathrm{tr}}(Y)^{\ast}$ is a Zariski sheaf.

Now let $\Delta^{\bullet}$ be the cosimplicial $k$-scheme given by
\begin{equation*}
\Delta^{i}=\mathrm{Spec}\,k[X_{0},\ldots,X_{i}]/(X_{0}+\cdots+X_{i}-1)
\end{equation*}
with the $j$-th face map $\partial_{j}:\Delta^{i}\rightarrow\Delta^{i+1}$ given by setting $X_{j}=0$. We consider $\Delta^{\bullet}$ as a cosimplicial fs log-scheme over $\mathrm{Spec}\,k$ by endowing each $\Delta^{i}$ with trivial log-structure. Then for every fs log-scheme $(X,N_{X})$ over $\mathrm{Spec}\,k$, the underlying scheme of $(X,N_{X})\times^{\mathrm{fs}}_{\mathrm{Spec}\,k}\Delta^{i}$ is $X\times\Delta^{i}$ and the log-structure is the inverse image log-structure of $N_{X}$ along the projection $\mathrm{pr}_{X}:X\times\Delta^{i}\rightarrow X$.

Let $(Y,M_{Y})$ be an fs log-scheme over $\mathrm{Spec}\,k$. Since each $\Delta^{i}$ is (classically) smooth and $\mathbb{Z}_{\mathrm{tr}}(Y)^{\ast}$ is a Zariski sheaf, the presheaves
\begin{equation*}
C_{i}(\mathbb{Z}_{\mathrm{tr}}(Y)^{\ast}):(X,N_{X})\mapsto\mathbb{Z}_{\mathrm{tr}}(Y)^{\ast}((X\times\Delta^{i},\mathrm{pr}^{\ast}_{X}N_{X}))
\end{equation*}
are also Zariski sheaves for each $i$, and thus $C_{\ast}(\mathbb{Z}_{\mathrm{tr}}(Y)^{\ast})$ is a complex of Zariski sheaves.
\end{proof}

\subsection{Log-motivic cohomology}\label{log-motivic cohomology section}
\ \par

For $n\geq 1$, let $(\mathbb{A}^{n}_{k},D_{n})$ be the log-scheme over $\mathrm{Spec}\,k$ whose underlying scheme is $\mathbb{A}_{k}^{n}$, and whose log-structure is the log-structure associated to the divisor \begin{equation*}
D_{n}=\{0\}\times\mathbb{A}_{k}^{n-1}+\mathbb{A}^{1}_{k}\times\{0\}\times\mathbb{A}_{k}^{n-2}+\cdots+\mathbb{A}_{k}^{n-1}\times\{0\}\,.
\end{equation*}
Following \cite[\S3]{SV00a}, define $\mathbb{Z}_{\mathrm{tr}}(\mathbb{A}_{k}^{\wedge n})^{\ast}$ to be the presheaf with transfers $\mathbb{Z}_{\mathrm{tr}}((\mathbb{A}_{k}^{n},D_{n}))^{\ast}/\mathcal{E}_{n}$ where $\mathcal{E}_{n}$ is the sum of the images of the maps $\mathbb{Z}_{\mathrm{tr}}((\mathbb{A}_{k}^{n-1},D_{n-1}))^{\ast}\rightarrow\mathbb{Z}_{\mathrm{tr}}((\mathbb{A}_{k}^{n},D_{n}))^{\ast}$ induced by the embeddings $\mathbb{A}_{k}^{n-1}\hookrightarrow\mathbb{A}_{k}^{n}$ given by $(x_1,\ldots,x_{n-1})\mapsto(x_1,\ldots,1,\ldots,x_{n-1})$. 

In the definition of the log-motivic complex $\mathbb{Z}_{\log}(n)$ we need to work with a more restrictive class of finite log-correspondences in $\mathrm{Cor}(X,\mathbb{A}^{n}_{k})^{\ast}$. Let $Z\in\mathrm{Cor}(X,\mathbb{A}^{n}_{k})^{\ast}\subset\mathrm{Cor}(X,(\mathbb{P}^{1}_{k})^{\times n})$. Consider $Z_{i}:=\mathrm{pr}_{i}(Z)$, where $\mathrm{pr}_{i}:X\times(\mathbb{P}^{1}_{k})^{\times n}\rightarrow X\times\mathbb{P}^{1}_{k}$ is the $i$-th projection, which is finite over $X$ and $Z_{i}\in\mathrm{Cor}(X,\mathbb{A}^{1}_{k})^{\ast}\subset\mathrm{Cor}(X,\mathbb{P}^{1}_{k})$. Then $Z_{i}$ defines an element in $\mathrm{Pic}(X\times\mathbb{P}^{1}_{k})=\mathrm{Pic}(X)\times\mathbb{Z}$ and there exists a rational function $f_{i}$ on $X\times\mathbb{P}^{1}_{k}$ and a divisor $C_{i}$ on $X$ such that for ${f_{i}|}_{X\times\mathbb{A}^{1}_{k}}$ we have $Z_{i}=D(f_{i})+C_{i}\times\mathbb{A}^{1}_{k}$, see \cite[Lemma 4.4]{MVW06} and the erratum in \cite{W}. Define 
\begin{equation*}
\mathrm{Cor}_{0}(X,\mathbb{A}_{k}^{n})^{\ast}=\{Z\in\mathrm{Cor}(X,\mathbb{A}^{n}_{k})^{\ast}\,|\,f_{i}(0)\in\underline{N}^{\mathrm{gp}}_{X}(X)\text{ for all }i\}\,.
\end{equation*}
Here and in the following for a rational function $f$ on $Y=\mathrm{Spec}\,A$ we define its value at $0$ as follows: let $f=g(t)/h(t)$ where $g$ and $h$ are polynomials with coefficients in $A$. If $g(0)$ and $h(0)$ lie in $\mathcal{O}_{Y}(Y)\cap j_{\ast}\mathcal{O}_{Y^{\mathrm{sm}}}^{\times}(Y)$, then $f(0)$ is well-defined.
Note that the condition on $f_{i}$ is compatible with the general assumption that we deal with finite log-correspondences, namely ${f_{i}(0)|}_{X^{\mathrm{sm}}}\in\mathcal{O}(X^{\mathrm{sm}})^{\ast}$ if and only if ${Z_{i}}|_{X^{\mathrm{sm}}}\in\mathrm{Cor}(X^{\mathrm{sm}},\mathbb{G}_{m})$. For $Y=(\mathbb{A}^{n}_{k},D_{n})$, consider the modified presheaf with transfer, also denoted by $\mathbb{Z}_{\mathrm{tr}}(\mathbb{A}_{k}^{n},D_{n})^{\ast}$:
\begin{align*}
& \mathrm{SemiStab}_{k}\rightarrow\mathrm{Ab} \\
& (X,N_{X})\mapsto\mathrm{Cor}_{0}(X,\mathbb{A}_{k}^{n})^{\ast}\,.
\end{align*}
From now on, whenever we write $\mathbb{Z}_{\mathrm{tr}}(Y)^{\ast}$ for $Y=(\mathbb{A}^{n}_{k},D_{n})$ we shall always mean this restricted presheaf with transfer. As in Lemma \ref{sheaf} $\mathbb{Z}_{\mathrm{tr}}(Y)^{\ast}$ is in fact a Zariski sheaf.

\begin{defn}
The log-motivic complex $\mathbb{Z}_{\log}(n)$ of weight $n$ is the complex of sheaves with transfers $C_{\ast}(\mathbb{Z}_{\mathrm{tr}}(\mathbb{A}_{k}^{\wedge n})^{\ast})[-n]$.
\end{defn}

Since $\mathbb{Z}_{\log}(n)[n]$ is a direct summand of $C_{\ast}(\mathbb{Z}_{\mathrm{tr}}(\mathbb{A}_{k}^{n},D_{n})^{\ast})$, the log-motivic complex $\mathbb{Z}_{\log}(n)$ is a complex of Zariski sheaves. If $(X,N_{X})$ is object of $\mathrm{SemiStab}_{k}$, then $\mathbb{Z}_{\log,X}(n)$ denotes the restriction of $\mathbb{Z}_{\log}(n)$ to the Zariski site of $X$.

\begin{defn}
Let $(X,N_{X})$ be an object of $\mathrm{SemiStab}_{k}$. We define the log-motivic cohomology of $(X,N_{X})$ to be the hypercohomology of $\mathbb{Z}_{\log}(n)$ with respect to the Zariski topology:
 \begin{equation*}
H_{\log-\mathcal{M}}^{i}(X,\mathbb{Z}(n)):=\mathbb{H}^{i}_{\mathrm{Zar}}(X,\mathbb{Z}_{\log}(n))\,.
\end{equation*}
\end{defn}

Notice that if $X$ is a smooth scheme over $\mathrm{Spec}\,k$, considered as a log-scheme by endowing it with the trivial log-structure, then the log-motivic cohomology of $X$ coincides with the motivic cohomology of $X$ as defined by Suslin-Voevodsky.

\begin{rem}\label{special example}
Of course, it would be desirable to work with the ``full'' monoid sheaf $\mathcal{O}_{Y}\cap j_{\ast}\mathcal{O}_{Y^{\mathrm{sm}}}$ in the definition of the log-motivic complex. The main reason why we use the possibly smaller monoid sheaf $\underline{N}_{Y}$ is a comparison between logarithmic Milnor $K$-theory and the modified logarithmic Hyodo-Kato Hodge-Witt sheaf which provides a semistable version of the Bloch-Gabber-Kato theorem (Theorem \ref{mod-p K-theory}). The comparison map uses explicitly the elements $g_{i}\in\underline{N}_{Y}(Y)^{\mathrm{gp}}$ and is -- a priori -- not defined for $\mathcal{O}_{Y}\cap j_{\ast}\mathcal{O}_{Y^{\mathrm{sm}}}$. Moreover, the $p$-adic deformation theory carried out in Section \ref{p-adic deformation section} relies on a gluing argument along the logarithmic Hyodo-Kato sheaf, hence only makes sense for a log-motivic cohomology defined by using the more restrictive class of finite log-correspondences. We shall return to this point in Remark \ref{usual log-Milnor}.
\end{rem}

\subsection{The log-motivic complex of weight one}
\ \par

Let $\mathcal{M}^{\ast}(\mathbb{P}_{k}^{1};0,\infty):\mathrm{SemiStab}_{k}\rightarrow\mathrm{Ab}$ be the functor which sends a semistable variety $(X,N_{X})$ to the group of rational functions on $X\times\mathbb{P}_{k}^{1}$ which are regular in a neighbourhood of $X\times\{0,\infty\}$ and equal to $1$ on $X\times\{0,\infty\}$. Then $\mathcal{M}^{\ast}(\mathbb{P}_{k}^{1};0,\infty)$ is a sheaf for the Zariski topology on $\mathrm{SemiStab}_{k}$.

Let $(Y,N_{Y})$ be a semistable variety over $k$. Then the trivial locus $Y^{\mathrm{triv}}$ of the structure morphism $(Y,N_{Y})\rightarrow\mathrm{Spec}\,k$ coincides with the smooth locus $Y^{\mathrm{sm}}$ of $Y$. Let $j:Y^{\mathrm{sm}}\hookrightarrow Y$ be the open immersion. We have a short exact sequence of abelian groups
\begin{equation*}
0\rightarrow\mathcal{M}^{\ast}(\mathbb{P}_{k}^{1};0,\infty)(Y^{\mathrm{sm}})\rightarrow\mathbb{Z}_{\mathrm{tr}}(\mathbb{G}_{m})(Y^{\mathrm{sm}})\rightarrow\mathbb{Z}\oplus\mathcal{O}_{Y^{\mathrm{sm}}}^{\ast}(Y^{\mathrm{sm}})\rightarrow 0
\end{equation*}
by \cite[Lemma 4.4]{MVW06}. We shall extend this exact sequence over $Y$ as follows:

Recall that $\mathrm{Cor}(Y,\mathbb{A}^{1}_{k})\subset\mathrm{Cor}(Y,\mathbb{P}_{k}^{1})$ and $\mathrm{Pic}(Y\times\mathbb{P}_{k}^{1})=\mathrm{Pic}(Y)\times\mathbb{Z}$, so to any $Z\in\mathrm{Cor}(Y,\mathbb{A}_{k}^{1})$ we can associate a unique rational function $f$ on $Y\times\mathbb{P}_{k}^{1}$ and a divisor $C$ on $Y$ such that the Weil divisor $D(f)$ is equal to $Z+C\times\mathbb{A}^{1}_{k}$, and such that there exists $n\in\mathbb{Z}$ with $f/t^{n}=1$ on $Y\times\{\infty\}$ (see again \cite[Lemma 4.4]{MVW06} and its erratum in \cite{W}). Define 
\begin{align*}
\mathrm{Cor}_{0}(Y,\mathbb{A}_{k}^{1})^{\ast}:=\{Z\in\mathrm{Cor}(Y,\mathbb{A}^{1}_{k})\,|\, \text{ if } & D(f)=Z+C\times\mathbb{A}^{1}_{k}\text{ as above then}\\
&f(0)\in\underline{N}^{\mathrm{gp}}_{Y}(Y),\text{ and }Z|_{Y^{\mathrm{sm}}\times\mathbb{A}^{1}_{k}}\in\mathrm{Cor}(Y^{\mathrm{sm}},\mathbb{G}_{m})\}\,.
\end{align*}
In particular, if $D(f)=Z+C\times\mathbb{A}^{1}_{k}\in\mathrm{Cor}_{0}(Y,\mathbb{A}_{k}^{1})^{\ast}$ then $f(0)|_{Y^{\mathrm{sm}}}\in\mathcal{O}^{\ast}(Y^{\mathrm{sm}})$. Define
\begin{align*}
\lambda\,:\,
& \mathrm{Cor}_{0}(Y,\mathbb{A}_{k}^{1})^{\ast}\rightarrow\mathbb{Z}\oplus\underline{N}_{Y}(Y)^{\mathrm{gp}} \\
& \ \ \ \ Z\mapsto(n,(-1)^{n}f(0))\,.
\end{align*}
Then $\lambda$ is surjective and we can rewrite $\lambda$ as a surjective map
\begin{equation*}
\mathbb{Z}_{\mathrm{tr}}(\mathbb{A}_{k}^{1})^{\ast}((Y,N_Y))\rightarrow\mathbb{Z}\oplus\underline{N}_{Y}(Y)^{\mathrm{gp}}
\end{equation*} 
(see the construction of $\lambda$ in the proof of \cite[Lemma 4.4]{MVW06}). The kernel of $\lambda$ is exactly $\mathcal{M}^{\ast}(\mathbb{P}_{k}^{1};0,\infty)((Y,N_Y))$, so we get a short exact sequence
\begin{equation*}
0\rightarrow\mathcal{M}^{\ast}(\mathbb{P}_{k}^{1};0,\infty)((Y,N_Y))\rightarrow\mathbb{Z}_{\mathrm{tr}}(\mathbb{A}_{k}^{1})^{\ast}((Y,N_Y))\rightarrow\mathbb{Z}\oplus\underline{N}_{Y}(Y)^{\mathrm{gp}}\rightarrow 0\,.
\end{equation*}

Since $\lambda$ respects transfers \cite[Lemma 4.5]{MVW06} we can apply the functor $C_{\ast}$ to the exact sequence of sheaves with transfers
\begin{equation*}
0\rightarrow\mathcal{M}^{\ast}(\mathbb{P}^{1};0,\infty)\rightarrow\mathbb{Z}_{\mathrm{tr}}(\mathbb{A}_{k}^{1})^{\ast}\rightarrow\mathbb{Z}\oplus\underline{N}_{Y}^{\mathrm{gp}}\rightarrow 0
\end{equation*}
to get an exact sequence of complexes of sheaves with transfer
\begin{equation*}
0\rightarrow C_{\ast}(\mathcal{M}^{\ast}(\mathbb{P}^{1};0,\infty))\rightarrow C_{\ast}(\mathbb{Z}_{\mathrm{tr}}(\mathbb{A}_{k}^{1})^{\ast})\rightarrow C_{\ast}(\mathbb{Z}\oplus\underline{N}_{Y}^{\mathrm{gp}})\rightarrow 0
\end{equation*}
on $Y$. Splitting off $0\rightarrow C_{\ast}(\mathbb{Z})\rightarrow C_{\ast}(\mathbb{Z})\rightarrow 0$ yields an exact sequence
\begin{equation*}
0\rightarrow C_{\ast}(\mathcal{M}^{\ast}(\mathbb{P}^{1};0,\infty))\rightarrow \mathbb{Z}_{\log}(1)[1]\rightarrow C_{\ast}(\underline{N}_{Y}^{\mathrm{gp}})\rightarrow 0\,.
\end{equation*}
But $C_{\ast}(\underline{N}_{Y}^{\mathrm{gp}})=\underline{N}_{Y}^{\mathrm{gp}}$ because $N_{Y}^{\mathrm{gp}}(U\times\Delta^{n})=N_{Y}^{\mathrm{gp}}(U)$. By \cite[Lemma 4.6]{MVW06} (which applies to $Y$ since the smoothness assumption is not used in the proof, nor in \cite[Lemma 2.18]{MVW06}) the complex $C_{\ast}(\mathcal{M}^{\ast}(\mathbb{P}^{1};0,\infty))$ is an acyclic complex of sheaves. Then we have shown the following:
\begin{prop}\label{Z(1)}
Let $(Y,N_{Y})$ be a semistable variety over $k$. Then 
\begin{equation*}
\mathbb{Z}_{\log,Y}(1)\cong\underline{N}_{Y}^{\mathrm{gp}}[-1]=:\mathbb{G}_{m}^{\log}[-1]\,.
\end{equation*}
\end{prop}
This generalises the smooth case considered in \cite[Lemma 3.2]{SV00a}.
\begin{cor}\label{corollary weight one}
\begin{equation*}
H_{\log-\mathcal{M}}^{i}(Y,\mathbb{Z}(1))\cong\left\lbrace\begin{array}{cl}
H_{\mathrm{Zar}}^{i-1}(Y,\mathbb{G}_{m}^{\log}) & \text{ if }i=1,2 \\
0 & \text{ if }i\neq 1,2\,.  \\
\end{array}\right.
\end{equation*}
\end{cor}
\begin{rem}
Note that our definition of $H_{\log-\mathcal{M}}^{2}(Y,\mathbb{Z}(1))$ does not reproduce the logarithmic Picard group considered in \cite{Yam11}. We have equipped the semistable variety $Y$ with a modified log-structure in order to obtain a geometric interpretation which generalises to higher codimension, whereas we do not have such a geometric interpretation for $\mathrm{Pic}^{\log}(Y):=H^{1}(Y,M_{Y}^{\mathrm{gp}})$. On the other hand we will see that the $p$-adic deformation theory of $H^{1}(Y,\underline{N}_{Y}^{\mathrm{gp}})$ is very similar to that of the usual logarithmic Picard group. See Remark \ref{1 plus p sequence}, Proposition \ref{weight one pro-complex} and Remark \ref{Yamashita remark}.
\end{rem}

\section{Logarithmic Milnor $K$-groups}\label{log milnor section}
We are going to define logarithmic Milnor $K$-groups and relate them to the cohomology of the complexes $\mathbb{Z}_{\log}(n)$ in analogy to the smooth case proven by Kerz \cite[Theorem 1.1]{Ker09}. 

Let $(Y,N_{Y})$ be a semistable variety. According to \cite[Proposition 11.3]{Kat96} $Y$ has a covering by open affines $U$ such that $U=\mathrm{Spec}\,A/\pi_{1}\cdots\pi_{r}$ where $A$ is a smooth $k$-algebra and each $A/\pi_{i}$ is smooth. Let $U^{\mathrm{sm}}$ be the smooth locus of $U$ and $j:U^{\mathrm{sm}}\hookrightarrow U$ the open immersion. Define $\underline{N}_{Y}(U)$ as in Definition \ref{alternative log-structure definition} and define the functions 
\begin{equation*}
g_{i}:=\pi_{i}+\prod_{\stackrel{j=1}{j\neq i}}^{r}\pi_{j}\in\underline{N}_{Y}(U)\,.
\end{equation*}
\begin{defn}\label{log Milnor K-theory definition}
For $U\subset Y$ as above, define 
\begin{equation*}
\mathcal{K}_{\log,n}^{\mathrm{Mil}}(U):=\frac{(\underline{N}_{Y}(U)^{\mathrm{gp}})^{\otimes n}}{I}
\end{equation*}
where $I$ is the subgroup generated by elements of the form $a\otimes(1-a)$ with $a,1-a\in\underline{N}_{Y}(U)^{\mathrm{gp}}$, those of the form $a\otimes(-a)$ with $a\in\underline{N}_{Y}(U)^{\mathrm{gp}}$, and those of the form $g_{I}^{n_{I}}x\otimes(1-\pi_{I}^{n_{I}}x)$ ranging over subsets $I\subset\{1,\ldots, r\}$, where $g_{I}^{n_{I}}:=\displaystyle\prod_{i\in I}g_{i}^{n_{i}}$ with $n_{i}\geq 0$, $\pi_{I}^{n_{I}}:=\displaystyle\prod_{i\in I}\pi_{i}^{n_{i}}$ with $n_{i}\geq 0$, and $x\in A^{\times}$ such that $1-\pi_{I}^{n_{I}}x\in A^{\times}$. The elements of $I$ are called (as they are for the usual Milnor $K$-groups) Steinberg relations. The residue class of $a_{1}\otimes\cdots\otimes a_{n}$ in $\mathcal{K}^{\mathrm{Mil}}_{\log, n}(U)$ is denoted by the symbol $\{a_{1},\ldots, a_{n}\}$.
\end{defn}

\begin{rem}\label{special example 2}
Related to Remark \ref{special example}, it would be desirable to work with the full monoid sheaf $\mathcal{O}_{Y}\cap j_{\ast}\mathcal{O}_{Y^{\mathrm{sm}}}$ in the definition of the logarithmic Milnor $K$-theory. We are forced to use Definition \ref{log Milnor K-theory definition} for the comparison with the modified logarithmic Hyodo-Kato Hodge-Witt sheaf (Theorem \ref{mod-p K-theory}). We shall return to this point in Remark \ref{usual log-Milnor}.   
\end{rem}

\begin{prop}
\label{K-theory smooth locus} 
Let $U\subset Y$ be open and let $V=U^{\mathrm{sm}}$. Then there is a canonical map
\begin{equation*}
\mu:\mathcal{K}_{\log,n}^{\mathrm{Mil}}(U)\rightarrow\mathcal{K}_{n}^{\mathrm{Mil}}(V):=\bigoplus_{\eta\in U^{0}}\mathcal{K}_{n}^{\mathrm{Mil}}(V_{\eta})
\end{equation*} 
induced by a canonical map 
\begin{equation*}
\underline{N}^{\mathrm{gp}}_{Y}(U)\rightarrow\underline{N}^{\mathrm{gp}}_{Y}(V)=\bigoplus_{\eta_{i}\in U^{0}}\mathcal{O}_{Y}^{\ast}(V_{\eta_{i}})
\end{equation*}
defined by 
\begin{equation*}
{g_{i}|}_{V_{\eta_{j}}}=\begin{cases} 
      \pi_{i} & \text{if }i\neq j \\
      \prod_{l\neq i}\pi_{l}^{-1} & \text{if }i=j 
   \end{cases}
\end{equation*}
and the natural restriction map on $\mathcal{O}(U)^{\times}$, where $V_{\eta_{i}}:=U_{i}\cap V$ and $U_{i}=\mathrm{Spec}\,A/\pi_{i}$ is the component of $U$ with generic point $\eta_{i}$. The kernel of the map $\mu$ is generated by $G:=\prod_{i} g_{i}$. That is, an element in the kernel is a sum of symbols where at least one entry is a power of $G$.
\end{prop}
\begin{proof}
The map defined in the proposition
\begin{equation*}
\underline{N}_{Y}^{\mathrm{gp}}(U)\rightarrow\underline{N}_{Y}^{\mathrm{gp}}(V)=\bigoplus_{\eta\in U^0}\mathcal{O}_{Y}^{\ast}(V_{\eta})
\end{equation*}
induces maps 
\begin{equation*}
\underline{N}_{Y}^{\mathrm{gp}}(U)^{\otimes n}\rightarrow\bigoplus_{\eta\in U^0}(\mathcal{O}_{Y}^{\ast}\left(V_{\eta})^{\otimes n}\right)
\end{equation*}
and
\begin{equation*}
\mathcal{K}_{\log,n}^{\mathrm{Mil}}(U)\rightarrow\bigoplus_{\eta\in U^0}\mathcal{K}_{n}^{\mathrm{Mil}}(V_{\eta})\,.
\end{equation*}
It is easy to check that the Steinberg relations in Definition \ref{log Milnor K-theory definition} vanish in $\bigoplus_{\eta\in U^0}\mathcal{K}_{n}^{\mathrm{Mil}}(V_{\eta})$. Since $\mathcal{K}_{n}^{\mathrm{Mil}}(V_{\eta})\rightarrow\mathcal{K}_{n}^{\mathrm{Mil}}(k(\eta))$ is injective by \cite[Theorem 6.1]{Ker09}, it is enough to consider the composite map $\mathcal{K}_{\log,n}^{\mathrm{Mil}}(U)\rightarrow\displaystyle\coprod_{\eta\in U^{0}}\mathcal{K}_{n}^{\mathrm{Mil}}(k(\eta))$. The claim for $n=1$ trivially follows from the inclusions
\begin{equation*}
\underline{N}^{\mathrm{gp}}_{Y}(U)/G^{\mathbb{Z}}\hookrightarrow\left(\mathcal{O}_{Y}(U)\cap j_{\ast}\mathcal{O}_{Y}(U^{\mathrm{sm}})^{\ast}\right)^{\mathrm{gp}}\hookrightarrow\mathcal{O}_{Y}(U^{\mathrm{sm}})^{\ast}\hookrightarrow\coprod_{\eta_{i}\in U^{0}}k(\eta_{i})^{\ast}\,.
\end{equation*}

Now let $n\geq 2$. Assume that $\sum z$ is a finite sum of symbols $z=\{a_{1},\ldots,a_{n}\}$ in $\mathcal{K}_{\log,n}^{\mathrm{Mil}}(U)/\langle G\rangle$ that vanishes in $\mathcal{K}_{n}^{\mathrm{Mil}}(k(\eta_{i}))$ for all $i$. Let $T_{1}$ be the subgroup of $\mathcal{K}^{\mathrm{Mil}}_{\log,n}(U)/\langle G\rangle$ generated by symbols where at least on entry is of the form $1+\pi_{1}x$ with $x\in A$. We show that $\sum z\in T_{1}$. Since $\sum z$ vanishes in $\mathcal{K}_{n}^{\mathrm{Mil}}(k(\eta_{1}))$, it is a sum of sums of symbols $z_{1}+\cdots +z_{n}$ that become a multilinear relation modulo $\pi_{1}$ and a sum of individual symbols $z_{0}$ that become Steinberg relations modulo $\pi_{1}$. Let us first assume that in $\sum z$ all $z$ have entries in $A^{\times}$. Without loss of generality, let us assume that only bilinear relations modulo $\pi_1$ occur. Then consider a triple $z_{1}+z_{2}+z_{3}=\{c_{1},\ldots, c_{n}\}-\{a_{1},\ldots,a_{n}\}-\{b_{1},\ldots, b_{n}\}$ with $\overline{c}_{1}\equiv\overline{a}_{1}\overline{b}_{1}\mod\pi_{1}$ and $\overline{c}_{1}\equiv\overline{a}_{1}\equiv\overline{b}_{1}\mod\pi_{1}$ for all $i\geq 2$. Then it is clear that any lifting $z'_{1}, z'_{2}, z'_{3}$ of $z_{1}\mod\pi_1$, $z_{2}\mod\pi_1$, $z_{3}\mod\pi_1$ in $\mathcal{K}^{\mathrm{Mil}}_{\log,n}(U)/\langle G\rangle$ has entries $c'_{i},a'_{i},b'_{i}$ for $i=1,\ldots, n$ that differ from the original entries $c_{i},a_{i},b_{i}$ by a factor of $1+\pi_{1}x$ with $x\in A$ (depending on $c_{i},a_{i},b_{i})$. This implies that $z_{1}+z_{2}+z_{3}$ is equivalent (modulo bilinear relations which vanish in $\mathcal{K}_{\log,n}^{\mathrm{Mil}}(U)/\langle G\rangle$) to a sum of symbols that all contain an entry $1+\pi_{1}x$, $x\in A$, and hence is in $T_{1}$.

Similarly, we can argue for the element $z_{0}=(a_{1},\ldots, a_{n})$ that becomes a Steinberg relation modulo $\pi_{1}$ there exists a lifting $z'_{0}$ which is a Steinberg relation itself (hence vanishes in $\mathcal{K}_{\log,n}^{\mathrm{Mil}}(U)/\langle G\rangle$) and where all entries $a'_{i}$ in $z'_{0}$ differ from $a_{i}$ by a factor $1+\pi_{1}x_{i}$. This shows that $z_{0}$ is equivalent to a sum of symbols that all contain an entry $1+\pi_{1}x$ with $x\in A$, hence is in $T_{1}$.

Now let us allow in $z_{0}=(a_{1},\ldots, a_{n})$ entries in $\underline{N}_{Y}(U)$ so each entry $a_{s}$ can be written as $a_{s}=g_{I_{s}}^{n_{I_{s}}}x_{s}$ with $x_{s}\in A^{\times}$. If $1\notin I_{S}$ then $\overline{a}_{s}\equiv\pi_{I_{s}}^{n_{I_{s}}}\overline{x}_{s}\mod\pi_{1}$. If $1\in I_{S}$ then 
\begin{align*}
\overline{a}_{s} & 
\equiv\left(\prod_{j=2}^{r}\pi_{j}\right)^{-n_{1}}\cdot\prod_{j\in I_{S}, j\neq 1}\pi_{I_{S}\setminus\{1\}}^{n_{I_{S}\setminus\{1\}}}\overline{x}_{s}\mod\pi_{1} \\
& \equiv\prod_{i\in I_{S}\setminus\{1\}}g_{i}^{n_{i}}\cdot\prod_{j=2}^{r}g_{j}^{-n_{1}}\overline{x}_{s}\mod\pi_{1}\,.
\end{align*}
This shows that $\overline{a}_{s}\mod\pi_{1}$ always has a lifting $a'_{s}=g_{I'_{S}}^{n_{I'_{S}}}x'_{s}$ with $1\notin I'_{S}$. It is clear that two liftings of $\overline{a}_{s}\mod\pi_{1}$ differ by a power of $G$ times a $1$-unit in in $A^{\times}$.

Now consider the relevant Steinberg relations in $\mathcal{K}_{n}^{\mathrm{Mil}}(k(\eta_{i}))$:

- $\{\overline{a}_{1},\ldots,\overline{a}_{n}\}$ where $\overline{a}_{s}=\pi_{I_{s}}^{n_{I_{s}}}\overline{x}_{s}$, $\overline{a}_{t}=1-\pi_{I_{s}}^{n_{I_{s}}}\overline{x}_{s}$, $1\notin I_{s}$.

- $\{\overline{b}_{1},\ldots,\overline{b}_{n}\}$ where $\overline{b}_{s}=\pi_{I_{s}}^{n_{I_{s}}}\overline{y}_{s}$, $\overline{b}_{t}=-\pi_{I_{s}}^{n_{I_{s}}}\overline{y}_{s}$, $1\notin I_{s}$, for some pair $(s,t)$ with $s\neq t$.

Then any two liftings $a_{s},a'_{s}$ of $\overline{a}_{s}$ where we can assume $a_{s}=g_{I_{s}}^{n_{I_{s}}}x_{s}$ differ by a factor $G^{k}(1+\pi_{1}t)$ with $k\in\mathbb{Z}, t\in A$. Likewise for $\overline{b}_{s}$, $\overline{a}_{t}$ and $\overline{b}_{t}$. We conclude that $z_{0}=(a_{1},\ldots,a_{n})$ (or $z_{0}=(b_{1},\ldots, b_{n})$) is equivalent modulo the Steinberg relations in $\mathcal{K}^{\mathrm{Mil}}_{\log,n}(U)/\langle G\rangle$ to a sum of symbols that all contain an entry $1+\pi_{1}x$, $x\in A$. If we consider a finite sum of symbols in the general case, we argue again that a lifting of a multilinear relation is equivalent to a sum of symbols that contain an entry $1+\pi_{1}x$, $x\in A$ as before. Hence $\sum z\in T_{1}$. 

Now let $T_{2}$ be the subgroup of $\mathcal{K}_{\log,n}^{\mathrm{Mil}}(U)/\langle G\rangle$ generated by symbols that contain two entries of the forms $(1+\pi_{1}\lambda_{1})$, $(1+\pi_{2}\lambda_{2})$ or an entry of the form $(1+\pi_{1}\pi_{2}\lambda_{12})$. By repeating the above argument in the $T_{1}$-case, we conclude that the condition ``$\sum z\mod\pi_{i}$ vanishes in $\mathcal{K}_{n}^{\mathrm{Mil}}(k(\eta_{i}))$ for $i=1,2$'' implies that $\sum z\in T_{2}$. By induction we conclude that our element $\sum z$ is a sum of symbols with entries $1+\pi_{I_{1}}^{k_{I_{1}}}x_{1},\ldots, 1+\pi_{I_{s}}^{k_{I_{s}}}x_{s}$ such that $I_{1}\cup\cdots\cup I_{s}=\{1,\ldots, r\}$.

Using the vanishing of the Steinberg relations in Definition \ref{log Milnor K-theory definition} it suffices to prove the following statement: the symbol $\{1+\pi_{I_{1}}^{k_{I_{1}}}x_{1},\ldots,1+\pi_{I_{n}}^{k_{I_{n}}}x_{n}\}$ vanishes in $\mathcal{K}_{\log,n}^{\mathrm{Mil}}(U)$ for $x_{i}\in A$, $\pi_{I_{j}}^{k_{j}}=\prod_{i\in I_{j}}\pi_{i}^{k_{i}}$, $k_{i}\geq 1$, and $I_{1}\cup\cdots\cup I_{s}=\{1,\ldots, r\}$, so all $\pi_{i}$ occur. By induction it suffices to show the following claim: for $x,y\in A$ we have
\begin{equation*}
\{1+\pi_{1}x,1+\pi_{2}y\}=\{1+\pi_{1}\pi_{2}z,z'\}+\{1+\pi_{1}\pi_{2}\tilde{z},\tilde{z}'\}
\end{equation*}
for some $z,\tilde{z}\in A$ and $z',\tilde{z}'\in\underline{N}_{Y}(U)^{\mathrm{gp}}$. Notice that
\begin{equation*}
\left\{\frac{1+\pi_{1}x+\pi_{1}x\pi_{2}y}{1+\pi_{1}x},?\right\}=\{1+\pi_{1}\pi_{2}\tilde{z},?\}
\end{equation*}
for $\tilde{z}=\frac{xy}{1+\pi_{1}x}\in A$, so it suffices to show the claim for the elements of the form $\{1+\pi_{1}x(1+\pi_{2}y),1+\pi_{2}y\}$ where $x,y\in A$. But
\begin{align*}
\{1+\pi_{1}x(1+\pi_{2}y),1+\pi_{2}y\}
& = -\{1+\pi_{1}x(1+\pi_{2}y),-g_{1}x\} \\
& = -\left\{(1+\pi_{1}x)\left(1+\frac{\pi_{1}x\pi_{2}y}{1+\pi_{1}x}\right),-g_{1}x\right\} \\
& = -\{1+\pi_{1}\pi_{2}z,z'\}
\end{align*}
where $z=\frac{xy}{1+\pi_{1}x}$ and $z'=-g_{1}x$. This proves the claim and the proposition.
\end{proof}

\subsection{The relation with $\mathbb{Z}_{\log}(n)$}
\ \par

We construct a map
\begin{equation*}
\mathcal{H}^{n}(\mathbb{Z}_{\log}(n))(U)\rightarrow\mathcal{K}_{\log,n}^{\mathrm{Mil}}(U)
\end{equation*}
as follows. First recall the construction of the map in the classical case when $U=\mathrm{Spec}\,F$ is the spectrum of a field, as defined in \cite[\S5, paragraph 31]{MVW06}. There is a commutative diagram
\begin{equation*}
\begin{tikzpicture}[descr/.style={fill=white,inner sep=1.5pt}]
        \matrix (m) [
            matrix of math nodes,
            row sep=3em,
            column sep=5em,
            text height=1.5ex, text depth=0.25ex
        ]
        { \mathbb{Z}_{\mathrm{tr}}(\mathbb{G}_{m}^{\wedge n})(\Delta^{1}_{F}) & \mathbb{Z}_{\mathrm{tr}}(\mathbb{G}_{m}^{\wedge n})(\mathrm{Spec}\,F) & H^{n}(\mathrm{Spec}\,F,\mathbb{Z}(n))  \\
      \ & \ & K_{n}^{\mathrm{Mil}}(F) \\};

        \path[overlay,dashed,->, font=\scriptsize] 
        (m-1-2) edge node[below left]{$f$} (m-2-3);
        
        \path[overlay,->, font=\scriptsize] 
        (m-1-1) edge node[above]{$\partial_{0}-\partial_{1}$}(m-1-2)
        (m-1-2) edge (m-1-3)
        (m-1-3) edge node[right]{$\lambda$}(m-2-3)
         ;

\end{tikzpicture}
\end{equation*}
The map $f$ is constructed using the norm map in Milnor $K$-theory. For the precise construction of $f$ see the more general case of a semi-local regular ring below. It follows from Weil reciprocity \cite[Theorem 5.4, Corollary 5.5]{MVW06} that the composite of $f$ with the face operators vanishes, hence $\lambda$ is well-defined for the spectrum of a field. Now by purity or the exactness of Gersten resolutions for motivic cohomology and Milnor $K$-theory (see the introduction of \cite{Ker09}) the above diagram also exists for $A$ a semi-local regular ring. So we have a diagram (for $V=\mathrm{Spec}\,A$)
\begin{equation*}
\begin{tikzpicture}[descr/.style={fill=white,inner sep=1.5pt}]
        \matrix (m) [
            matrix of math nodes,
            row sep=3em,
            column sep=5em,
            text height=1.5ex, text depth=0.25ex
        ]
        { \mathbb{Z}_{\mathrm{tr}}(\mathbb{G}_{m}^{\wedge n})(\Delta^{1}_{F}\times _{k}V) & \mathbb{Z}_{\mathrm{tr}}(\mathbb{G}_{m}^{\wedge n})(V) & \mathcal{H}^{n}(\mathbb{Z}(n))(V)  \\
      \ & \ & \mathcal{K}_{n}^{\mathrm{Mil}}(V) \\};

        \path[overlay,dashed,->, font=\scriptsize] 
        (m-1-2) edge node[below left]{$f$} (m-2-3);
        
        \path[overlay,->, font=\scriptsize] 
        (m-1-1) edge node[above]{$\partial_{0}-\partial_{1}$}(m-1-2)
        (m-1-2) edge (m-1-3)
        (m-1-3) edge node[right]{$\lambda$}(m-2-3)
         ;

\end{tikzpicture}
\end{equation*}
In fact, the norm map on Milnor $K$-groups is also defined for finite extensions of semi-local rings over a field, see \cite[paragraph 5]{Ker09}. We describe the map $f$ in this case explicitly. An element in $\mathbb{Z}_{\mathrm{tr}}(\mathbb{G}_{m}^{\wedge n})(V)$ is given by the class of a finite correspondence $Z\subseteq V\times\mathbb{G}_{m}^{n}$ such that $Z=\mathrm{Spec}\,B$ for a semi-local finite extension $B$ over $A$. Define $Z_{i}=\mathrm{pr}_{i}(Z)$ via the projection 
$$
\mathrm{pr}_{i}:V\times\mathbb{G}_{m}^{n}\rightarrow V\times\mathbb{A}^{n}_{k}\subseteq V\times(\mathbb{P}^{1}_{k})^{n}\rightarrow V\times\mathbb{P}_{k}^{1}\,.
$$
This is a finite correspondence on $V\times\mathbb{G}_{m}$. The projection map to $\mathbb{G}_{m}$ defines a function $a_{i}\in\mathcal{O}(Z_{i})^{\times}$. We may assume that $Z_{i}=\mathrm{Spec}\,B_{i}$ is the spectrum of a semi-local ring $B_{i}$ which is contained in $B$. Then $f(Z)=N_{B/A}(\{a_{1},\ldots,a_{n}\})$ where $N_{B/A}$ is the transfer map in Milnor $K$-theory for semi-local rings. That $f$ factors through $\lambda:\mathcal{H}^{n}(\mathbb{Z}(n))(V)\rightarrow\mathcal{K}_{n}^{\mathrm{Mil}}(V)$ follows from the main result in \cite[Theorem 1.1]{Ker09}. 

Now we promote this map to a map $\mathcal{H}^{n}(\mathbb{Z}_{\log}(n))(U)\rightarrow\mathcal{K}_{\log,n}^{\mathrm{Mil}}(U)$ where $U=\mathrm{Spec}\,A^{0}$ is local semistable, that is a lolcalisation of a semistable affine scheme $\mathrm{Spec}\,A/\pi_{1}\cdots\pi_{r}$ as at the beginning of \S\ref{log milnor section}. Let $V=U^{\mathrm{sm}}$ be the smooth locus of $U$ and $A^{0}\hookrightarrow C$ an injection into a semi-local regular ring such that $\mathrm{Spec}\,C$ is contained in $V$. Then we consider the above diagram for $U$:
\begin{equation*}
\begin{tikzpicture}[descr/.style={fill=white,inner sep=1.5pt}]
        \matrix (m) [
            matrix of math nodes,
            row sep=3em,
            column sep=5em,
            text height=1.5ex, text depth=0.25ex
        ]
        { \mathbb{Z}_{\mathrm{tr}}(\mathbb{A}_{k}^{\wedge n})^{\ast}(\Delta^{1}_{F}\times _{k}U) & \mathbb{Z}_{\mathrm{tr}}(\mathbb{A}_{k}^{\wedge n})^{\ast}(U) & \mathcal{H}^{n}(\mathbb{Z}_{\log}(n))(U)  \\
      \ & \ & \mathcal{K}_{\log,n}^{\mathrm{Mil}}(U) \\};

        \path[overlay,dashed,->, font=\scriptsize] 
        (m-1-2) edge node[below left]{$``f"$} (m-2-3);
        
        \path[overlay,->, font=\scriptsize] 
        (m-1-1) edge node[above]{$\partial_{0}-\partial_{1}$}(m-1-2)
        (m-1-2) edge (m-1-3)
        (m-1-3) edge node[right]{$``\lambda"$}(m-2-3)
         ;

\end{tikzpicture}
\end{equation*}
We will use an equivalent description of our restrictive assumptions of log-correspondences in $\mathrm{Cor}_{0}(U,\mathbb{A}^{n}_{k})^{\ast}$ defined in \S\ref{log-motivic cohomology section}. For $n=1$ this leads to a logarithmic version of the identification of $\mathcal{H}^{1}(\mathbb{Z}(1))$ and $\mathcal{O}^{\ast}$ in the smooth case proven in \cite[Remark 3.2.0]{SV00a}. An element in $\mathbb{Z}_{\mathrm{tr}}(\mathbb{A}_{k}^{n})^{\ast}(U)=\mathrm{Cor}_{0}(U,\mathbb{A}^{n}_{k})^{\ast}$ is given by the class of a finite log-correspondence $Z\subseteq U\times\mathbb{A}_{k}^{n}$ where $Z=\mathrm{Spec}\,B^{0}$ for a semi-local ring $B^{0}$ finite over $A^{0}$. Define $Z_{i}$ as $\mathrm{pr}_{i}(Z)$ via the projection 
$$
\mathrm{pr}_{i}:U\times\mathbb{A}_{k}^{n}\subseteq U\times(\mathbb{P}^{1}_{k})^{n}\rightarrow U\times\mathbb{P}_{k}^{1}\,.
$$
This defines a finite log-correspondence on $U\times\mathbb{A}^{1}_{k}$. The projection map to $\mathbb{A}^{1}_{k}$ defines a function $a_{i}\in\mathcal{O}(Z_{i})\subset B^{0}$ with $a_{i}|_{V_{i}}\in\mathcal{O}(V_{i})^{\times}$ where $V_{i}=Z_{i}\times_{U}V$. Then we can assume that there exists an injection $B^{0}\hookrightarrow B$ with $B$ semi-local regular and finite over $C$ such that $a_{i}\in B^{0}$ and $a_{i}\in B^{\ast}$. Then we have the transfer map in Milnor $K$-theory for semi-local rings $N_{B/C}(\{a_{1},\ldots, a_{n}\})\in K_{n}^{\mathrm{Mil}}(C)$. To get a norm map on $\mathcal{K}^{\mathrm{Mil}}_{\log, n}(B^{0})$ we use the more restrictive assumption of log-correspondences in $\mathrm{Cor}_{0}(U,\mathbb{A}_{k}^{n})^{\ast}$, namely we require that $a_{i}$ is in the group associated to the multiplicative monoid of $B^{0}$ generated by $g_{i}\in A^{0}$ and $\mathcal{O}(B^{0})^{\ast}$. By standard properties in Milnor $K$-theory (bilinearity etc) we may assume that for the symbol $\{a_{1},\ldots, a_{n}\}$ we have $a_{1}=g_{I}^{n_{I}}$ and $a_{2},\ldots, a_{n}\in\mathcal{O}(B^{0})^{\times}$, or $a_{i}\in\mathcal{O}(B^{0})^{\ast}$ for all $i$. Then define 
$$
N_{B^0/A^0}(\{g_{I}^{n_{I}},a_{2},\ldots, a_{n}\}):=\{g_{I}^{n_{I}},N_{B^{0}/A^{0}}(\{a_{2},\ldots,a_{n}\})\}
$$
where $N_{B^{0}/A^{0}}$ is the norm map on usual Milnor $K$-theory of finite extensions of semi-local rings. Likewise, if $a_{1},\ldots, a_{n}\in\mathcal{O}(B^{0})^{\times}$ then $N_{B^0/A^0}(\{a_{1},a_{2},\ldots, a_{n}\})$ is defined by the transfer map in Milnor $K$-theory. So we can define $f(Z)=N_{B^0/A^0}(\{a_1,\ldots,a_n\})\in\mathcal{K}^{\mathrm{Mil}}_{\log,n}(A^0)$. 

We have a commutative diagram
\begin{equation*}
\begin{tikzpicture}[descr/.style={fill=white,inner sep=1.5pt}]
        \matrix (m) [
            matrix of math nodes,
            row sep=3em,
            column sep=5em,
            text height=1.5ex, text depth=0.25ex
        ]
        { \mathcal{K}_{\log,n}^{\mathrm{Mil}}(B^0) & \mathcal{K}_{n}^{\mathrm{Mil}}(B)  \\
      \mathcal{K}_{\log,n}^{\mathrm{Mil}}(A^0) & \mathcal{K}_{n}^{\mathrm{Mil}}(C) \\};

        \path[overlay,->, font=\scriptsize] 
        (m-1-1) edge (m-1-2)
        (m-2-1) edge (m-2-2)
        (m-1-1) edge node[left]{$N_{B^{0}/A^{0}}$} (m-2-1)
        (m-1-2) edge node[right]{$N_{B/C}$} (m-2-2)
         ;

\end{tikzpicture}
\end{equation*}
where the upper and lower horizontal arrows are defined as in Proposition \ref{K-theory smooth locus} with kernels generated by $\langle G\rangle$. So we have defined the map 
\begin{equation*}
``f":\mathbb{Z}_{\mathrm{tr}}(\mathbb{A}_{k}^{\wedge n})^{\ast}(U)\rightarrow\mathcal{K}_{\log,n}^{\mathrm{Mil}}(U)
\end{equation*}
in the above diagram. Later on we will show that $``f"$ factors through 
\begin{equation*}
``\lambda":\mathcal{H}^{n}(\mathbb{Z}_{\log}(n))(U)\rightarrow \mathcal{K}_{\log,n}^{\mathrm{Mil}}(U)
\end{equation*}
as suggested in the diagram.

Conversely, let $z=\{z_{1}\otimes\cdots\otimes z_{n}\}\in\mathcal{K}_{\log,n}^{\mathrm{Mil}}(U)$ where $z_{i}\in\underline{N}^{gp}_{Y}(U)\simeq\mathbb{Z}_{\log}(1)[1](U)$. There is a product map 
\begin{equation*}
\mathbb{Z}_{\log}(1)[1]\otimes\cdots\otimes\mathbb{Z}_{\log}(1)[1]\rightarrow\mathbb{Z}_{\log}(n)[n]
\end{equation*}
defined in \cite[page 141]{SV00a} or \cite[Construction 3.11]{MVW06}. This defines a class $[z]\in\mathcal{H}^{n}(\mathbb{Z}_{\log}(n))(U)$. We will also prove later that $z\mapsto[z]$ factors through $\mathcal{K}_{\log,n}^{\mathrm{Mil}}(U)$.

Consider the commutative diagram
\begin{equation}
\label{Cor diagram}
\begin{tikzpicture}
\node (a) at (0,0) {$\mathrm{Cor}_{0}(\Delta^{n+1-i}\times U,\mathbb{A}^{n}_{k})^{\ast}$};
\node (b) at (4.5,0) {$\mathrm{Cor}_{0}(\Delta^{n-i}\times U,\mathbb{A}^{n}_{k})^{\ast}$};
\node (c) at (9,0) {$\mathrm{Cor}_{0}(\Delta^{n-1-i}\times U,\mathbb{A}^{n}_{k})^{\ast}$};
\node (e) at (0,-2.5) {$ \mathrm{Cor}(\Delta^{n+1-i}\times V,\mathbb{G}^{n}_{m}) $};
\node (f) at (4.5,-2.5) {$ \mathrm{Cor}(\Delta^{n-i}\times V,\mathbb{G}^{n}_{m}) $};
\node (g) at (9,-2.5) {$ \mathrm{Cor}(\Delta^{n-1-i}\times V,\mathbb{G}^{n}_{m}) $};
\path[->,font=\scriptsize]
(a) edge[draw=none] node [sloped, auto=false,              allow upside down] {$\vdots$} (b)
([yshift= 10pt]a.east) edge node[above] {$\partial^{n-i}_0$} ([yshift= 10pt]b.west)
([yshift= 5pt]a.east) edge ([yshift= 5pt]b.west)
([yshift= -10pt]a.east) edge node[below] {$\partial^{n-i}_{n+1-i}$} ([yshift= -10pt]b.west)
(b) edge[draw=none] node [sloped, auto=false,              allow upside down] {$\vdots$} (c)
([yshift= 10pt]b.east) edge node[above] {$\partial^{n-1-i}_{0}$} ([yshift= 10pt]c.west)
([yshift= 5pt]b.east) edge ([yshift= 5pt]c.west)
([yshift= -10pt]b.east) edge node[below] {$\partial^{n-1-i}_{n-i}$} ([yshift= -10pt]c.west)
(e) edge[draw=none] node [sloped, auto=false,              allow upside down] {$\vdots$} (f)
([yshift= 10pt]e.east) edge node[above] {$\partial^{n-i}_{0}$} ([yshift= 10pt]f.west)
([yshift= 5pt]e.east) edge ([yshift= 5pt]f.west)
([yshift= -10pt]e.east) edge node[below] {$\partial^{n-i}_{n+1-i}$} ([yshift= -10pt]f.west)
(f) edge[draw=none] node [sloped, auto=false,              allow upside down] {$\vdots$} (g)
([yshift= 10pt]f.east) edge node[above] {$\partial^{n-1-i}_{0}$} ([yshift= 10pt]g.west)
([yshift= 5pt]f.east) edge ([yshift= 5pt]g.west)
([yshift= -10pt]f.east) edge node[below] {$\partial^{n-1-i}_{n-i}$} ([yshift= -10pt]g.west);
\path[->,font=\scriptsize]
(a) edge node [right]{$\iota_{n+1-i}$} (e)
(b) edge node [right]{$\iota_{n-i}$} (f)
(c) edge node [right]{$\iota_{n-1-i}$} (g);
\end{tikzpicture}
\end{equation}
with natural restriction maps $\iota_{j}$ sending $\alpha$ to $\alpha|_{V\times\mathbb{G}_{m}^{n}}$. We will manipulate these restriction maps $\iota_{j}$ in order to make them compatible with the canonical map $\mu$ on Milnor $K$-groups defined in Proposotion \ref{K-theory smooth locus}. We define maps $\tilde{\iota}_{j}$ as follows:

If for a finite log-correspondence $Z$ all $p_{2}(Z_{i})$ (where $p_{2}:Z_{i}\rightarrow\mathbb{A}^{1}_{k}$ is the projection map) are in $\mathcal{O}(Z_{i})^{\ast}$ we define $\tilde{\iota}_{j}(Z)=\iota_{j}(Z)$ as before. Now let $p_{2}(Z_{i})=g_{s}$ for some $1\leq i\leq n$ and some $1\leq s\leq r$ and $p_{2}(Z_{k})\in\mathcal{O}(Z_{k})^{\ast}$ for $k\neq i$. Note that $V$ is the disjoint union of smooth components $V_{l}$, for $1\leq l\leq r$. Then define 
$$
{\tilde{\iota}_{j}}|_{V_{l}}=\begin{cases} 
      {\iota_{j}(Z)}|_{V_{l}} & \text{if }l\neq s \\
      T_{i}\circ({\iota_{j}(Z)}|_{V_{s}}) & \text{for }l=s 
   \end{cases}
$$
where $T_{i}$ is the map induced by the map $T_{i}:\mathbb{G}_{m}^{n}\rightarrow\mathbb{G}_{m}^{n}$ given by $(x_{1},\ldots, x_{n})\mapsto(x_{1},\ldots, x_{i}^{-1},\ldots, x_{n})$. Extend the map linearly to all $Z\in\mathrm{Cor}_{0}(\Delta^{j}\times_{k}U,\mathbb{A}^{n}_{k})^{\ast}$. It then follows that $\tilde{\iota}_{j}:=T\circ\iota_{j}$, where $T$ is defined as above, maps any element $Z$ such that $p_{2}(Z_{i})\in\mathcal{O}(Z_{i})$ is a power of $G=\prod_{t=1}^{r}g_{t}$ for some $i$ to the image $E$ generated by the inclusions $\epsilon_{i}:\Delta^{j}\times V\times\mathbb{G}^{n-1}_{m}\rightarrow\Delta^{j}\times V\times\mathbb{G}_{m}^{n}$ given by $(x_{1},\ldots, x_{n-1})\mapsto (x_{1},\ldots, 1,\ldots, x_{n-1})$, hence vanishes in $\mathrm{Cor}(\Delta^{j}\times V,\mathbb{G}^{\wedge n}_{m})$. From our considerations on $\mathcal{K}_{\log, n}^{\mathrm{Mil}}$ we know that any other irreducible log-correspondence $Z$ such that its image under $\tilde{\iota}_{j}$ lies in $E$ is such that for certain $i\in{1,\ldots, n}$ $p_{2}(Z_{i})$ defines a $1$-unit $1+\pi_{I_{i}}^{k_{I_{i}}}x$ with $x\in\mathcal{O}(Z_{i})$ and such that $\bigcup_{i} I_{i}=\{1,\ldots, r\}$. We will show that all such $Z$ vanish in the cohomology. Note that $T$ induces an endomorphism on $\mathrm{Cor}(\Delta^{j}\times V,\mathbb{G}^{\wedge n}_{m})$ which maps $\iota_{j}(Z)$, where $Z$ is a correspondence such that $p_{2}(Z_{i})\in\mathcal{O}(Z_{i})$ is a power of $G$, to zero.
\begin{prop}
\label{motivic smooth locus}
For each $i,n\geq 0$, the modified restriction maps $\tilde{\iota}_{j}$ define corresponding restriction maps on the homotopy-invariant sheaf with transfers $\mathcal{H}^{n-j}(\mathbb{Z}_{\log}(n))$ from $U$ to $V=U^{\mathrm{sm}}$ with kernel $\ker\tilde{\iota}_{j}$ generated by the cup-product of $\mathcal{H}^{n-j-1}(\mathbb{Z}_{\log}(n-1))(U)$ with $G\in\mathcal{H}^{1}(\mathbb{Z}_{\log}(1))(U)$. For $i=n=1$, we recover the canonical map $\underline{N}^{\mathrm{gp}}_{Y}(U)\rightarrow\mathcal{O}_{Y^{\mathrm{sm}}}^{\times}(V)=\mathcal{H}^{1}(\mathbb{Z}_{\log}(1))(V)$ defined in Proposition \ref{K-theory smooth locus}.
\end{prop}
\begin{proof}
Consider the commutative diagram
\begin{equation}
\label{Cor diagram}
\begin{tikzpicture}
\node (a) at (0,0) {$\mathrm{Cor}_{0}(\Delta^{n+1-i}\times U,\mathbb{A}^{\wedge n}_{k})^{\ast}$};
\node (b) at (4.5,0) {$\mathrm{Cor}_{0}(\Delta^{n-i}\times U,\mathbb{A}^{\wedge n}_{k})^{\ast}$};
\node (c) at (9,0) {$\mathrm{Cor}_{0}(\Delta^{n-1-i}\times U,\mathbb{A}^{\wedge n}_{k})^{\ast}$};
\node (e) at (0,-2.5) {$ \mathrm{Cor}(\Delta^{n+1-i}\times V,\mathbb{G}^{\wedge n}_{m}) $};
\node (f) at (4.5,-2.5) {$ \mathrm{Cor}(\Delta^{n-i}\times V,\mathbb{G}^{\wedge n}_{m}) $};
\node (g) at (9,-2.5) {$ \mathrm{Cor}(\Delta^{n-1-i}\times V,\mathbb{G}^{\wedge n}_{m}) $};
\path[->,font=\scriptsize]
(a) edge[draw=none] node [sloped, auto=false,              allow upside down] {$\vdots$} (b)
([yshift= 10pt]a.east) edge node[above] {$\partial^{n-i}_0$} ([yshift= 10pt]b.west)
([yshift= 5pt]a.east) edge ([yshift= 5pt]b.west)
([yshift= -10pt]a.east) edge node[below] {$\partial^{n-i}_{n+1-i}$} ([yshift= -10pt]b.west)
(b) edge[draw=none] node [sloped, auto=false,              allow upside down] {$\vdots$} (c)
([yshift= 10pt]b.east) edge node[above] {$\partial^{n-1-i}_{0}$} ([yshift= 10pt]c.west)
([yshift= 5pt]b.east) edge ([yshift= 5pt]c.west)
([yshift= -10pt]b.east) edge node[below] {$\partial^{n-1-i}_{n-i}$} ([yshift= -10pt]c.west)
(e) edge[draw=none] node [sloped, auto=false,              allow upside down] {$\vdots$} (f)
([yshift= 10pt]e.east) edge node[above] {$\partial^{n-i}_{0}$} ([yshift= 10pt]f.west)
([yshift= 5pt]e.east) edge ([yshift= 5pt]f.west)
([yshift= -10pt]e.east) edge node[below] {$\partial^{n-i}_{n+1-i}$} ([yshift= -10pt]f.west)
(f) edge[draw=none] node [sloped, auto=false,              allow upside down] {$\vdots$} (g)
([yshift= 10pt]f.east) edge node[above] {$\partial^{n-1-i}_{0}$} ([yshift= 10pt]g.west)
([yshift= 5pt]f.east) edge ([yshift= 5pt]g.west)
([yshift= -10pt]f.east) edge node[below] {$\partial^{n-1-i}_{n-i}$} ([yshift= -10pt]g.west);
\path[->,font=\scriptsize]
(a) edge node [right]{$\tilde{\iota}_{n+1-i}$} (e)
(b) edge node [right]{$\tilde{\iota}_{n-i}$} (f)
(c) edge node [right]{$\tilde{\iota}_{n-1-i}$} (g);
\end{tikzpicture}
\end{equation}
with vertical maps $\tilde{\iota}_{j}$. Take two elements $Z,Z'\in\ker\left(\sum_{j=0}^{n-i}(-1)^{j}\partial_{j}^{n-1-i}\right)\subset\mathrm{Cor}_{0}(\Delta^{n-i}\times U,\mathbb{A}_{k}^{\wedge n})^{\ast}$ such that $Z-Z'$ is not in the subgroup $\Gamma$ of correspondences $Q$ such that $p_{2}(Q_{t})$ is a power of $G$ for some $t$. Assume that there exists $Y\in\mathrm{Cor}(\Delta^{n+1-i}\times V,\mathbb{G}_{m}^{\wedge n})$ such that 
\begin{equation*}
\tilde{\iota}_{n-i}(Z)-\tilde{\iota}_{n-i}(Z')=\sum_{j=0}^{n+1-i}(-1)^{j}\partial_{j}^{n-i}(Y)\,.
\end{equation*}
Define $\overline{Y}$ to be the closure of $T(Y)$ in $\Delta^{n+1-i}\times U\times\mathbb{P}_{k}^{n}$. 

We claim that $\overline{Y}\in\mathrm{Cor}_{0}(\Delta^{n+1-i}\times U,\mathbb{A}_{k}^{\wedge n})^{\ast}$. Indeed, suppose that the support of $\overline{Y}$ is not contained in $\Delta^{n+1-i}\times U\times\mathbb{A}_{k}^{n}$. Then there exists a $t\in\{1,\ldots,n\}$ such that $\mathrm{pr}_{t}(\overline{Y})\subset \Delta^{n+1-i}\times U\times\mathbb{P}_{k}^{1}$ is not contained in $\Delta^{n+1-i}\times U\times\mathbb{A}_{k}^{1}$ (where $\mathrm{pr}_{t}$ is induced by the $t$-th projection $\mathrm{pr}_{t}:(\mathbb{P}_{k}^{1})^{\times n}\rightarrow\mathbb{P}_{k}^{1}$). On the other hand, $\mathrm{pr}_{t}(\overline{Y})$ is the closure of $\mathrm{pr}_{t}(T(Y))\in\mathrm{Cor}(\Delta^{n+1-i}\times V,\mathbb{G}^{\wedge 1}_{m})$. The analogous commutative diagram to \eqref{Cor diagram} 
\begin{equation*}
\begin{tikzpicture}
\node (a) at (0,0) {$\mathrm{Cor}_{0}(\Delta^{n+1-i}\times U,\mathbb{A}^{\wedge 1}_{k})^{\ast}$};
\node (b) at (4.5,0) {$\mathrm{Cor}_{0}(\Delta^{n-i}\times U,\mathbb{A}^{\wedge 1}_{k})^{\ast}$};
\node (c) at (9,0) {$\mathrm{Cor}_{0}(\Delta^{n-1-i}\times U,\mathbb{A}^{\wedge 1}_{k})^{\ast}$};
\node (e) at (0,-2.5) {$ \mathrm{Cor}(\Delta^{n+1-i}\times V,\mathbb{G}^{\wedge 1}_{m}) $};
\node (f) at (4.5,-2.5) {$ \mathrm{Cor}(\Delta^{n-i}\times V,\mathbb{G}^{\wedge 1}_{m}) $};
\node (g) at (9,-2.5) {$ \mathrm{Cor}(\Delta^{n-1-i}\times V,\mathbb{G}^{\wedge 1}_{m}) $};
\path[->,font=\scriptsize]
(a) edge[draw=none] node [sloped, auto=false,              allow upside down] {$\vdots$} (b)
([yshift= 10pt]a.east) edge node[above] {$\partial^{n-i}_0$} ([yshift= 10pt]b.west)
([yshift= 5pt]a.east) edge ([yshift= 5pt]b.west)
([yshift= -10pt]a.east) edge node[below] {$\partial^{n-i}_{n+1-i}$} ([yshift= -10pt]b.west)
(b) edge[draw=none] node [sloped, auto=false,              allow upside down] {$\vdots$} (c)
([yshift= 10pt]b.east) edge node[above] {$\partial^{n-1-i}_{0}$} ([yshift= 10pt]c.west)
([yshift= 5pt]b.east) edge ([yshift= 5pt]c.west)
([yshift= -10pt]b.east) edge node[below] {$\partial^{n-1-i}_{n-i}$} ([yshift= -10pt]c.west)
(e) edge[draw=none] node [sloped, auto=false,              allow upside down] {$\vdots$} (f)
([yshift= 10pt]e.east) edge node[above] {$\partial^{n-i}_{0}$} ([yshift= 10pt]f.west)
([yshift= 5pt]e.east) edge ([yshift= 5pt]f.west)
([yshift= -10pt]e.east) edge node[below] {$\partial^{n-i}_{n+1-i}$} ([yshift= -10pt]f.west)
(f) edge[draw=none] node [sloped, auto=false,              allow upside down] {$\vdots$} (g)
([yshift= 10pt]f.east) edge node[above] {$\partial^{n-1-i}_{0}$} ([yshift= 10pt]g.west)
([yshift= 5pt]f.east) edge ([yshift= 5pt]g.west)
([yshift= -10pt]f.east) edge node[below] {$\partial^{n-1-i}_{n-i}$} ([yshift= -10pt]g.west);
\path[->,font=\scriptsize]
(a) edge node [right]{$\tilde{\iota}_{n+1-i}$} (e)
(b) edge node [right]{$\tilde{\iota}_{n-i}$} (f)
(c) edge node [right]{$\tilde{\iota}_{n-1-i}$} (g);
\end{tikzpicture}
\end{equation*}
shows that
\begin{equation}
\label{formula}
\sum_{j=0}^{n+1-i}(-1)^{j}\partial_{j}^{n-i}(\mathrm{pr}_{t}(\overline{Y}))=\mathrm{pr}_{t}(Z-Z')=\mathrm{pr}_{t}(Z)-\mathrm{pr}_{t}(Z')\,.
\end{equation}
Indeed, for $n=i$, $\mathrm{pr}_{t}(Z)$ and $\mathrm{pr}_{t}(Z')$ are, up to correspondences defining a power of $G$, the closures of $\mathrm{pr}_{t}(T(\tilde{\iota}(Z)))$ and $\mathrm{pr}_{t}(T(\tilde{\iota}(Z')))$. The exact sequence
$$
1\rightarrow\langle G\rangle\rightarrow\underline{N}_{Y}^{\mathrm{gp}}(U)\rightarrow\mathcal{O}^{\times}_{Y^{\mathrm{sm}}}(V)
$$
implies that $\mathrm{pr}_{t}(Z)$ and $\mathrm{pr}_{t}(Z')$ in $\mathrm{Cor}_{0}(U,\mathbb{A}^{1}_{k})^{\ast}$ define the same cohomology class modulo log-correspondences corresponding to a power of $G$. Since a non-trivial power of $G$ defines a non-trivial cohomology class, we have 
$$
(\partial_{1}^{1}-\partial_{0}^{1})(\mathrm{pr}_{t}(\overline{Y})=\mathrm{pr}_{t}(Z)-\mathrm{pr}_{t}(Z')
$$
and we conclude that $\mathrm{pr}_{t}(\overline{Y}\in\mathrm{Cor}(\Delta^{1}\times U,\mathbb{A}^{\wedge 1}_{k})^{\ast}$. For $i<n$ the cohomology of both complexes (for $\mathbb{A}^{1}_{k}$ and $\mathbb{G}_{m}$) vanishes because the complexes $\mathbb{Z}_{\log}(1)$ and $\mathbb{Z}(1)$ are acyclic in degrees $<1$ (Proposition \ref{Z(1)}). Hence $\mathrm{pr}_{t}(Z)$ and $\mathrm{pr}_{t}(Z')$ in $\mathrm{Cor}_{0}(\Delta^{n-1}\times U,\mathbb{A}^{\wedge 1}_{k})^{\ast}$ vanish in the cohomology and the formula \eqref{formula} then implies that $\mathrm{pr}_{t}(\overline{Y})\in\mathrm{Cor}_{0}(\Delta^{n+i-1}\times U,\mathbb{A}^{\wedge 1}_{k})$. Hence we conclude that $\overline{Y}$ is closed in $\Delta^{n+1-i}\times U\times\mathbb{A}^{n}_{k}$ and hence proper over $\Delta^{n+1-i}\times U$, because the projection $\Delta^{n+1-i}\times U\times(\mathbb{P}_{k}^{1})^{\times n}\rightarrow\Delta^{n+i-1}\times U$ is proper. Since all $\mathrm{pr}_{t}(\overline{Y})$ are quasi-finite over $\Delta^{n+1-i}\times U$ for all $t$, $\overline{Y}$ is itself quasi-finite over $\Delta^{n+i-1}\times U$, hence finite. This shows the claim.

The above argument also shows that an element $Z\in\mathrm{Cor}(\Delta^{n-i}\times U,\mathbb{A}^{\wedge n}_{k})^{\ast}$ is modulo the subgroup $\Gamma$ the closure of $T\circ\tilde{\iota}_{n-1}(Z)\in\mathrm{Cor}(\Delta^{n-i}\times V,\mathbb{G}^{\wedge n}_{m})$ and therefore the vertical maps in \eqref{Cor diagram} are injective modulo $\Gamma$. This proves the proposition. If $Z-Z'$ is contained in $\Gamma$ then it is clear that its cohomology class lies in the cup-product of $\mathcal{H}^{n-i-1}(\mathbb{Z}_{\log}(n-i))$ with a power of $G$.
\end{proof}
\begin{rem}\label{mod p coefficients}
Proposition \ref{motivic smooth locus} shows that we have an exact sequence
\begin{equation*}
\{G\}\cup\mathcal{H}^{i-1}(\mathbb{Z}_{\log,Y}(n-1)\otimes^{\mathbb{L}}\mathbb{Z}/p^{r})\rightarrow\mathcal{H}^{i}(\mathbb{Z}_{\log,Y}(n)\otimes^{\mathbb{L}}\mathbb{Z}/p^{r})\rightarrow u_{\ast}\mathcal{H}^{i}(\mathbb{Z}_{Y^{\mathrm{sm}}}(n)\otimes^{\mathbb{L}}\mathbb{Z}/p^{r})
\end{equation*}
Since all of the terms of $\mathbb{Z}_{\log,Y}(n)$, $\mathbb{Z}_{\log,Y}(n-1)$ and $\mathbb{Z}_{Y^{\mathrm{sm}}}(n)$ are free abelian groups, the complexes $\mathbb{Z}_{\log,Y}(n)\otimes\mathbb{Z}/p^{r}$, $\mathbb{Z}_{\log,Y}(n-1)\otimes\mathbb{Z}/p^{r}$ and $\mathbb{Z}_{Y^{\mathrm{sm}}}(n)\otimes\mathbb{Z}/p^{r}$ represent the derived tensor products $\mathbb{Z}_{\log,Y}(n)\otimes^{\mathbb{L}}\mathbb{Z}/p^{r}$, $\mathbb{Z}_{\log,Y}(n-1)\otimes^{\mathbb{L}}\mathbb{Z}/p^{r}$ and $\mathbb{Z}_{Y^{\mathrm{sm}}}(n)\otimes^{\mathbb{L}}\mathbb{Z}/p^{r}$. At one point in the proof of Proposition \ref{motivic smooth locus}, we argue by projecting down to $\mathbb{P}^{1}_{k}$ and use acyclicity of $\mathbb{Z}_{\log}(1)$ and $\mathbb{Z}(1)$ for $i<n-1$ and injectivity modulo correspondences defined by powers of $G$ for $n=i-1$. The short exact sequence $0\rightarrow\mathbb{Z}\rightarrow\mathbb{Z}\rightarrow\mathbb{Z}/p^{r}\rightarrow 0$ shows that this remains true after tensoring with $\mathbb{Z}/p^{r}$. The rest of the proof remains the same.
\end{rem}

In Proposition \ref{motivic smooth locus} we implicitly used the following lemma:

\begin{lemma}
\label{new lemma}
Let $U$ be affine semistable over $k$ as before.
\begin{enumerate}[(i)]
\item Let $a\in\mathcal{O}^{\times}(U)\subset\mathcal{H}^{1}(\mathbb{Z}_{\log}(1))(U)$ be such that $1-a\in\mathcal{O}^{\times}(U)$. Then $a\cup(1-a)$ vanishes in $\mathcal{H}^{2}(\mathbb{Z}_{\log}(2))(U)$.

\item Assume $Z$ is a log-correspondence in $\mathrm{Cor}_{0}(U,\mathbb{A}^{n}_{k})^{\ast}$ such that for certain $i\in\{1,\ldots, n\}$ we have $p_{2}(Z_{i})=1+\pi_{I_{i}}^{k_{I_{i}}}x_{i}$ with $x_{i}\in\mathcal{O}(U)$ and such that $\cup_{i}I_{i}=\{1,\ldots, r\}$. Then the class of $Z$ vanishes in $\mathcal{H}^{n}(\mathbb{Z}_{\log}(n))(U)$.

\item Assume $Z$ is a log-correspondence such that for some $i,j$ with $i\neq j$ we have $p_{2}(Z_{i})=1-\pi_{I}^{k_{I}}x$ and $p_{2}(Z_{j})=g_{I}^{k_{I}}x$ for some $I\subseteq\{1,\ldots\}$, $k_{I}\in\mathbb{N}^{(I)}$ and $x\in\mathcal{O}^{\times}(U)$. Then the class of $Z$ vanishes in $\mathcal{H}^{n}(\mathbb{Z}_{\log}(n))(U)$. 
\end{enumerate}
\end{lemma}
\begin{proof}
The proof of (i) is very similar to the proof of \cite[Proposition 5.9]{MVW06} and is omitted. For (iii) it is clear that $\tilde{\iota}(Z)$ corresponds, under the isomorphism $\mathcal{H}^{n}(\mathbb{Z}(n))(U^{\mathrm{sm}})\cong\mathcal{K}_{n}^{\mathrm{Mil}}(U)$ to a Steinberg relation and hence vanishes; it is in the image of $\partial_{0}-\partial_{1}$. The proof of Proposition \ref{motivic smooth locus} then shows that $Z$ vanishes in $\mathcal{H}^{n}(\mathbb{Z}_{\log}(n))(U)$. 

For part (ii) we adopt certain arguments in the proof of Proposition \ref{K-theory smooth locus}. By induction it suffices to show that if $Z$ is a correspondence with $p_{2}(Z_{i})=1+\pi_{1}x_{1}$ and $p_{2}(Z_{j})=1+\pi_{2}x_{2}$ for some $(i,j)$ with $i\neq j$, and some $x_{1},x_{2}\in\mathcal{O}(U)$, then the class of $Z$ is the sum of classes of correspondences $Q,P$ with $p_{2}(Q_{i})=1+\pi_{1}\pi_{2}y_{12}$ and $p_{2}(P_{i})=1+\pi_{1}\pi_{2}\tilde{y}_{12}$ with $y_{12},\tilde{y}_{12}\in\mathcal{O}(U)$. But this is achieved using the vanishing properties in (i) and (iii) and following the reasoning in the proof of Proposition \ref{K-theory smooth locus}.
\end{proof}

As a corollary of Proposition \ref{motivic smooth locus}, we obtain
\begin{thm}
\label{K-theory and motivic}
If $k$ is infinite then there is a canonical isomorphism
\begin{equation*}
\mathcal{K}_{\log,n}^{\mathrm{Mil}}(U)\rightarrow\mathcal{H}^{n}(\mathbb{Z}_{\log}(n))(U)\,.
\end{equation*}
\end{thm}
\begin{proof}
Setting $j=0$ in Proposition \ref{motivic smooth locus} and setting $\mathscr{U}_{G}$ to be the image of the cup-product of $\mathcal{H}^{n-1}(\mathbb{Z}_{\log}(n-1))(U)$ with $G\in\mathcal{H}^{1}(\mathbb{Z}_{\log}(1))(U)$ in $\mathcal{H}^{n}(\mathbb{Z}_{\log}(n))$, we have a commutative diagram of injective maps
\begin{equation*}
\begin{tikzpicture}[descr/.style={fill=white,inner sep=1.5pt}]
        \matrix (m) [
            matrix of math nodes,
            row sep=3em,
            column sep=5em,
            text height=1.5ex, text depth=0.25ex
        ]
        { \mathcal{H}^{n}(\mathbb{Z}_{\log}(n))(U)/\mathscr{U}_{G} & \mathcal{H}^{n}(\mathbb{Z}(n))(V) \\
      \mathcal{K}_{\log,n}^{\mathrm{Mil}}(U)/\langle G\rangle & \mathcal{K}_{n}^{\mathrm{Mil}}(V) \\};

        \path[overlay, right hook->, font=\scriptsize]
        (m-1-1) edge node[left]{$\overline{\lambda}$} (m-2-1)        
        (m-1-1) edge (m-1-2)
        (m-2-1) edge (m-2-2)
        ;
        
        \path[overlay,->, font=\scriptsize]
        (m-1-2) edge node[right]{$\cong$} (m-2-2)
        ;
                 
\end{tikzpicture}
\end{equation*}
where the upper and lower horizontal maps are injective by propositions \ref{K-theory smooth locus} and \ref{motivic smooth locus} and the right vertical map is an isomorphism by \cite[Theorem 1.1]{Ker09}. This shows that the previously defined map $``f"$ factors through an injective map $\overline{\lambda}$. For the same reasons, the symbol map
$$
\mathbb{Z}_{\log}(1)[1]\otimes\cdots\otimes\mathbb{Z}_{\log}(1)[1]\rightarrow\mathbb{Z}_{\log}(n)[n]\rightarrow\mathcal{H}^{n}(\mathbb{Z}_{\log}(n))
$$ 
gives rise to a commutative diagram
\begin{equation*}
\begin{tikzpicture}[descr/.style={fill=white,inner sep=1.5pt}]
        \matrix (m) [
            matrix of math nodes,
            row sep=3em,
            column sep=5em,
            text height=1.5ex, text depth=0.25ex
        ]
        {  \mathcal{K}_{\log,n}^{\mathrm{Mil}}(U)/\langle G\rangle & \mathcal{H}^{n}(\mathbb{Z}_{\log}(n))(U)/\mathscr{U}_{G} \\
      \mathcal{K}_{n}^{\mathrm{Mil}}(V) & \mathcal{H}^{n}(\mathbb{Z}(n))(V) \\};

        \path[overlay,->, font=\scriptsize]
        (m-1-1) edge node[above]{$\overline{h}$} (m-1-2)
        (m-2-1) edge node[below]{$\cong$} (m-2-2)
        ;
        
        \path[overlay, right hook->, font=\scriptsize]
        (m-1-1) edge (m-2-1)
        (m-1-2) edge (m-2-2)
        ;
                 
\end{tikzpicture}
\end{equation*}

This shows that $\overline{\lambda}$ and $\overline{h}$ are the restrictions of the corresponding isomorphisms between the Milnor $K$-group and motivic cohomology on the smooth scheme $V$ and hence are isomorphisms themselves.

Now consider the commutative diagram

\begin{equation*}
\begin{tikzpicture}[descr/.style={fill=white,inner sep=1.5pt}]
        \matrix (m) [
            matrix of math nodes,
            row sep=3em,
            column sep=3em,
            text height=1.5ex, text depth=0.25ex
        ]
        { 0 & \langle G\rangle & \mathcal{K}_{\log, n}^{\mathrm{Mil}}(U) & \mathcal{K}_{\log, n}^{\mathrm{Mil}}(U)/\langle G\rangle & 0 \\
      0 & \mathscr{U}_{G} & \mathcal{H}^{n}(\mathbb{Z}_{\log}(n))(U) & \mathcal{H}^{n}(\mathbb{Z}_{\log}(n))(U)/\mathscr{U}_{G} & 0  \\};

        \path[overlay,dashed,->, font=\scriptsize] 
        (m-2-3) edge node[right]{$\lambda$} (m-1-3)
        ;
        
        \path[overlay,->, font=\scriptsize] 
        (m-1-1) edge (m-1-2)
        (m-1-2) edge (m-1-3)
        (m-1-3) edge (m-1-4)
        (m-1-4) edge (m-1-5)
        (m-2-1) edge (m-2-2)
        (m-2-2) edge (m-2-3)
        (m-2-3) edge (m-2-4)
        (m-2-4) edge (m-2-5)
        (m-2-2) edge (m-1-2)
        (m-2-4) edge node[left]{$\cong$} node[right]{$\overline{\lambda}$} (m-1-4)
         ;

\end{tikzpicture}
\end{equation*}
We wish to lift the isomorphism $\overline{\lambda}$ to an isomorphism $\lambda$. We prove this by induction on $n$, the case $n=1$ being clear. Note that under the canonical map 
$$
\langle G\rangle\otimes_{\mathbb{Z}}\mathcal{K}^{\mathrm{Mil}}_{\log,n-1}(U)\rightarrow \mathcal{K}^{\mathrm{Mil}}_{\log,n}(U)
$$
the only new vanishing Steinberg relations in $\mathcal{K}^{\mathrm{Mil}}_{\log,n}(U)$ are the ones $\{G^{k},a_{1},\ldots, a_{n-1}\}$ for $a_{i}=1-(\pi_{1}\cdots\pi_{r})^{k}$ for some $i$ and some $k\geq 0$, hence $a_{i}=1$ so $\{a_{1},\ldots, a_{n-1}\}=0$, or the relation $\{G^{k},-G^{k}\}$ with $k\in\mathbb{Z}$. On the other hand, an element $G^{k}\otimes z$ with $z\in\mathcal{H}^{n-1}(\mathbb{Z}_{\log}(n-1))(U)$ vanishes in $\mathcal{H}^{n}(\mathbb{Z}_{\log}(n))(U)$ under the cup-product if $z$ is represented by a log-correspondence $Z$ where $p_{2}(Z_{i})=-G^{k}$. Indeed, we can then adopt the argument in \cite[Example 5.7]{MVW06} that for $a\in\mathcal{H}^{1}(\mathbb{Z}_{\log}(1))(U)$, we have that $a\cup -a\in\mathcal{H}^{2}(\mathbb{Z}_{\log}(2))(U)$ vanishes; the proof in \cite[Example 5.7]{MVW06} easily passes over. Hence in the above diagram we have $\langle G\rangle\cong\mathscr{U}_{G}$. If $U$ is local we have defined $``f"$ on $\mathrm{Cor}_{0}(U,\mathbb{A}_{k}^{\wedge n})^{\ast}$. By a diagram chase we see that $``f"$ factors through the isomorphism $\lambda$. Now Lemma \ref{new lemma} implies that the image of all Steinberg relations in $\mathcal{K}_{\log, n}^{\mathrm{Mil}}(U)$ vanishes in $\mathcal{H}^{n}(\mathbb{Z}_{\log}(n))(U)$, hence the lifting $h$ of the map $\overline{h}$ is well-defined. We have seen that $h$ - as inverse of $\lambda$ - is an isomorphism if $U$ is local, and therefore is an isomorphism for general affine $U$ as well. 
\end{proof}

\begin{rem}\label{refined remark}
When $k$ is finite, instead of the map $\mu:\mathcal{K}_{\log,n}^{\mathrm{Mil}}(U)\rightarrow\mathcal{K}_{n}^{\mathrm{Mil}}(U^{\mathrm{sm}})$ of Proposition \ref{K-theory smooth locus}, one must consider the map $\hat{\mu}:\mathcal{K}_{\log,n}^{\mathrm{Mil}}(U)\rightarrow\hat{\mathcal{K}}_{n}^{\mathrm{Mil}}(U^{\mathrm{sm}})$ to improved Milnor $K$-theory \cite{Ker10} defined as follows: we consider the composite 
$$
(\underline{N}^{\mathrm{gp}}_{Y}(U))^{\otimes n}\rightarrow\mathcal{K}_{n}^{\mathrm{Mil}}(U^{\mathrm{sm}})\rightarrow\hat{\mathcal{K}}_{n}^{\mathrm{Mil}}(U^{\mathrm{sm}})
$$
where the last map is the natural homomorphism in loc. cit. Note that $\mathcal{K}_{n}^{\mathrm{Mil}}(U^{\mathrm{sm}})\rightarrow\hat{\mathcal{K}}_{n}^{\mathrm{Mil}}(U^{\mathrm{sm}})$ is an isomorphism whenever $k$ is infinite \cite[Proposition 10(5)]{Ker10}. Since the Gersten conjecture holds for improved Milnor $K$-theory by \cite[Proposition 10]{Ker10}, this composite map factors through $\mathcal{K}_{\log,n}^{\mathrm{Mil}}(U)$ because all the Steinberg relations there are killed in the Milnor $K$-groups of the generic points of the smooth locus. Since $\hat{\mathcal{K}}_{n}^{\mathrm{Mil}}(U^{\mathrm{sm}})\cong\mathcal{H}^{n}(\mathbb{Z}(n))(U^{\mathrm{sm}})$ by \cite[Proposition 10(11)]{Ker10}, using improved Milnor $K$-theory in the above proofs gives the results for finite fields $k$ as well.
\end{rem}

\begin{rem}\label{usual log-Milnor}
Related to Remarks \ref{special example} and \ref{special example 2}, it seems natural to work with the full monoid sheaf $\mathcal{O}_{Y}\cap j_{\ast}\mathcal{O}_{Y^{\mathrm{sm}}}$ in the definition of the logarithmic Milnor $K$-theory. In the following we will construct an example where we can work with $\mathcal{O}_{Y}\cap j_{\ast}\mathcal{O}_{Y^{\mathrm{sm}}}$ and all the results in the paper will hold verbatim:

Assume the semistable variety is locally given by $U=\mathrm{Spec}\,A/(\pi_{1}\pi_{2}\pi_{3})$ for a smooth $k$-algebra $A$ and irreducible smooth components $Y_{1}=\mathrm{Spec}\,A/(\pi_{1})$, $Y_{2}=\mathrm{Spec}\,A/(\pi_{2})$ and $Y_{3}=\mathrm{Spec}\,A/(\pi_{3})$. Let $P_{12}, P_{23}, P_{13}$ be the generic points of the codimension $1$ intersections $Y_{1}\cap Y_{2}$, $Y_{2}\cap Y_{3}$, $Y_{1}\cap Y_{3}$. Then the function $g_{1}g_{2}^{-1}g_{3}$ (with $g_{i}$ as in Definition \ref{alternative log-structure definition}) lies in $\mathcal{O}_{Y}(U)\cap j_{\ast}\mathcal{O}_{Y^{\mathrm{sm}}}(U^{\mathrm{sm}})$ and has a zero of order $2$ in $P_{13}$ and no other zeros or poles. Likewise, $g_{2}g_{3}^{-1}g_{1}$ has a zero of order $2$ in $P_{12}$ and $g_{3}g_{1}^{-1}g_{2}$ has a zero of order $2$ in $P_{23}$. Then it is clear that for any $h\in\mathcal{O}_{Y}(U)\cap j_{\ast}\mathcal{O}_{Y^{\mathrm{sm}}}(U^{\mathrm{sm}})$ we have $h^{2}\in\underline{N}_{Y}(U)^{\mathrm{gp}}$ and so 
\begin{equation*}
(\mathcal{O}_{Y}(U)\cap j_{\ast}\mathcal{O}_{Y^{\mathrm{sm}}}(U^{\mathrm{sm}}))^{\mathrm{gp}}/\underline{N}_{Y}(U)^{\mathrm{gp}}
\end{equation*}
is killed by $2$. When we define $\tilde{\mathcal{K}}_{\log,n}^{\mathrm{Mil}}(U)$ -- by replacing in Definition \ref{log Milnor K-theory definition} the group $\underline{N}_{Y}(U)^{\mathrm{gp}}$ by $(\mathcal{O}_{Y}(U)\cap j_{\ast}\mathcal{O}_{Y^{\mathrm{sm}}}(U^{\mathrm{sm}}))^{\mathrm{gp}}$ -- we have that  
\begin{equation*}
\tilde{\mathcal{K}}_{\log,n}^{\mathrm{Mil}}(U)/\mathcal{K}_{\log,n}^{\mathrm{Mil}}(U)
\end{equation*}
is a $2$-primary torsion group. If we assume that $p$ is bigger than $2$ then we get an isomorphism 
\begin{equation*}
\tilde{\mathcal{K}}_{\log,n}^{\mathrm{Mil}}(U)/p^{s}\simeq\mathcal{K}_{\log,n}^{\mathrm{Mil}}(U)/p^{s}
\end{equation*}
and all results in the paper in \S\ref{log HK subsection} and \S\ref{p-adic deformation section} will hold. In particular, in this special situation we can work with $\mathbb{Z}_{\log}(n)$ using all finite log-correspondences and its top cohomology sheaf $\mathcal{H}^{n}(\mathbb{Z}_{\log}(n))$ is isomorphic to $\tilde{\mathcal{K}}_{\log,n}^{\mathrm{Mil}}$.
\end{rem}

\subsection{Relation with the modified logarithmic Hyodo-Kato Hodge-Witt sheaf}\label{log HK subsection}
\ \par

Let $W_{s}\tilde{\omega}^{n}_{Y/k,\log}$ denote the modified logarithmic Hyodo-Kato Hodge-Witt sheaf on $Y_{\mathrm{\acute{e}t}}$. It is defined as follows: For $Y$ the closed fibre of a regular $W(k)$-scheme $X$ with semistable reduction, let $M_{Y}:=i^{\ast}M_{X}=i^{\ast}j_{\ast}\mathcal{O}_{X_{K}}$ be the usual log-structure on $Y$ (where $i:Y\hookrightarrow X$ is the closed immersion and $j:X_{K}\hookrightarrow X$ is the open immersion of the generic fibre). Let $u:Y^{\mathrm{sm}}\hookrightarrow Y$ be the open immersion of the smooth part. Then the modified (extended) Hyodo-Kato complex $W_{s}\tilde{\omega}^{\bullet}_{Y/k}$ is the $W_{s}(\mathcal{O}_{Y})$-subalgebra of $\mathcal{A}^{\bullet}:=u_{\ast}W_{s}\Omega_{Y^{\mathrm{sm}}/k}^{\bullet}[\theta]/\theta^{2}$, where $\theta$ is an indeterminate in degree one satisfying $\theta a=(-1)^{q}a\theta$ for $a\in u_{\ast}W_{s}\Omega^{q}_{Y^{\mathrm{sm}}/k}$ and $d\theta=0$, generated by $dW_{s}(\mathcal{O}_{Y})$ and the image of $d\log:M_{Y}\rightarrow\mathcal{A}^{1}$ defined on $u^{-1}i^{-1}(\mathcal{O}_{X}^{\ast})$ by the composition 
\begin{equation*}
u^{-1}i^{-1}(\mathcal{O}_{X}^{\ast})\rightarrow\mathcal{O}^{\ast}_{Y^{\mathrm{sm}}}\xrightarrow{d\log[-]}W_{s}\Omega^{1}_{Y^{\mathrm{sm}}/k}
\end{equation*}
and on $K^{\ast}$ by $a\mapsto\mathrm{ord}_{K}(a)\theta$ (see \cite[1.4]{HK94}). Then we recall \cite[Proposition 1.5]{HK94}:
\begin{prop}\label{theta ses}
The sequence 
\begin{align*}
& 0\rightarrow W_{s}\omega_{Y/k}^{\bullet}[-1]\rightarrow W_{s}\tilde{\omega}^{\bullet}_{Y/k}\rightarrow W_{s}\omega^{\bullet}_{Y/k}\rightarrow 0 \\
& \ \ \ \ \ \ \ \ \ \ \ \ \ a \ \ \ \ \mapsto \ \ \ \ a\theta \\
& \ \ \ \ \ \ \ \ \ \ \ \ \ \ \ \ \ \ \ \ \ \ \ \ \ \ \ \theta\ \ \ \ \mapsto \ \ \ \ 0
\end{align*}
is exact.
\end{prop}
The map $d\log:M_{Y}^{\mathrm{gp}}\rightarrow W_{s}\tilde{\omega}^{1}_{Y/k}$ induces a map $d\log:(M_{Y}^{\mathrm{gp}})^{\otimes n}\rightarrow W_{s}\tilde{\omega}^{n}_{Y/k}$. Write $W_{s}\tilde{\omega}^{n}_{Y/k,\log}$ for the image. As a corollary of Proposition \ref{theta ses} using $1-\varphi$ on $W_{s}\tilde{\omega}_{Y/k}^{n}$ we obtain
\begin{prop}\label{theta sequence}
There is an exact sequence
\begin{align*}
& 0\rightarrow W_{s}\omega_{Y/k,\log}^{n-1}[-1]\rightarrow W_{s}\tilde{\omega}^{n}_{Y/k,\log}\rightarrow W_{s}\omega^{n}_{Y/k,\log}\rightarrow 0 \\
& \ \ \ \ \ \ \ \ \ \ \ \ \ \ \ a \ \ \ \ \ \mapsto \ \ \ \ \ a\theta \\
& \ \ \ \ \ \ \ \ \ \ \ \ \ \ \ \ \ \ \ \ \ \ \ \ \ \ \ \ \ \ \ \theta\ \ \ \ \ \mapsto \ \ \ \ \ 0\,.
\end{align*}
\end{prop}

In the next section we will glue the log-motivic complex $\mathbb{Z}_{\log}(n)$ through its top cohomology $\mathcal{H}^{n}(\mathbb{Z}_{\log}(n))(U)\simeq\mathcal{K}_{\log,n}^{\mathrm{Mil}}(U)$ with the log-syntomic complex of Kato-Tsuiji via the modified logarithmic Hyodo-Kato sheaf $W_{s}\tilde{\omega}^{n}_{Y/k,\log}$ in order to achieve a semistable analogue of the deformational part of the main result of \cite{BEK14}. We construct a canonical map
\begin{equation*}
d\log:\mathcal{K}_{\log,n}^{\mathrm{Mil}}(U)\rightarrow W_{s}\tilde{\omega}^{n}_{Y/k,\log}(U)
\end{equation*}
as follows. For $n=1$, the map
\begin{equation*}
d\log:\underline{N}_{Y}^{\mathrm{gp}}(U)\rightarrow W_{s}\tilde{\omega}^{1}_{Y/k,\log}(U)
\end{equation*}
is given by the assignments
\begin{align*}
& x\in\mathcal{O}_{Y}(U)^{\ast}\mapsto d\log[x]\in W_{s}\omega^{1}_{Y/k,\log} \\
& g_{i}=\beta(e_{i})\mapsto d\log(\pi_{i})\in W_{s}\tilde{\omega}^{1}_{Y/k,\log}
\end{align*}
where $\pi_{i}=\alpha(e_{i})$ for the structure map $\alpha:M_{Y}\rightarrow\mathcal{O}_{Y}$, and $d\log$ is the canonical map on $M_{Y}^{\mathrm{gp}}$ that we recalled above (so $\prod_{i=1}^{r}g_{i}\mapsto\theta$). Note that $d\log(\pi_{i})$ in $\bigoplus_{\eta_i\in U^{0}}W_{s}\Omega^{1}_{k(\eta_{i})/k}$ has $j$-component $d\log[\pi_{i}]$ for $j\neq i$ and $i$-component $-\sum_{j\neq i}d\log[\pi_{j}]$. The above is extended to a map $d\log:\mathcal{K}_{\log,n}^{\mathrm{Mil}}(U)\rightarrow W_{s}\tilde{\omega}^{n}_{Y/k,\log}(U)$ by taking exterior products.

Since the $d\log$ map is surjective on $M_{Y}^{\otimes n}$ by definition, it is also surjective on $\mathcal{K}_{\log,n}^{\mathrm{Mil}}(U)$. By composing $d\log$ with the canonical injective map
\begin{equation*}
W_{s}\omega^{n}_{U/k,\log}\hookrightarrow\bigoplus_{\eta\in U^{0}}W_{s}\Omega^{n}_{k(\eta)/k}
\end{equation*}
it is clear that symbols $\{a,-a\}$ and $\{a,1-a\}$, $a\in\underline{N}_{Y}(U)^{\mathrm{gp}}$ and $\{1-\pi_{i}x,g_{i}x\}$ vanish in $\bigoplus_{\eta\in U^{0}}W_{r}\Omega^{n}_{k(\eta)/k}$, since they vanish in $\bigoplus_{\eta\in U^{0}}\mathcal{K}_{n}^{\mathrm{Mil}}(k(\eta))$, hence $d\log$ is well-defined on $\mathcal{K}^{\mathrm{Mil}}_{\log,n}(U)$. We will prove the following semistable analogue of the Bloch-Kato-Gabber theorem:
\begin{thm}\label{mod-p K-theory}
\begin{enumerate}[a)]
\item The group $\mathcal{K}_{\log, n}^{\mathrm{Mil}}(U)$ is $p$-torsion-free.

\item We have an isomorphism
\begin{equation*}
\mathcal{K}_{\log,n}^{\mathrm{Mil}}(U)/p^{s}\simeq W_{s}\tilde{\omega}^{n}_{Y/k,\log}(U)\,.
\end{equation*}
\end{enumerate}
\end{thm}
\begin{proof}
\begin{enumerate}[a)]
\item Let $V=U^{\mathrm{sm}}$. Since $\mathcal{K}_{n}^{\mathrm{Mil}}(V)$ is $p$-torsion-free by \cite{Izh91}, Proposition \ref{K-theory smooth locus} shows that $\mathcal{K}_{\log,n}^{\mathrm{Mil}}(U)/\langle G\rangle$ is $p$-torsion-free. We prove by induction on $n$ that $\mathcal{K}_{\log,n}^{\mathrm{Mil}}(U)$ is $p$-torsion-free as well. The case $n=1$ is clear. Assume it holds for $n-1$; mapping $\{G^{\mathbb{Z}}\}\otimes\mathcal{K}_{\log,n-1}^{\mathrm{Mil}}(U)$ to $\mathcal{K}_{\log,n}^{\mathrm{Mil}}(U)$ we see that the only new relations involving $G$ are $\{G^{k},1-(\pi_{1}\cdots\pi_{r})^{k}\}$ (which is $\{G^{k},1\}$ but a symbol with entry $1$ already vanishes in $\mathcal{K}_{\log,n-1}^{\mathrm{Mil}}(U)$) and $\{G^{k},-G^{k}\}$. But an element $\{G,a\}$ such that $p\{G,a\}=\{G,-G\}$ is already zero (compare with the argument in \cite[Lemma 5.8]{MVW06}). This proves part (a).

\item By part (a) we have an isomorphism
\begin{equation*}
\mathcal{K}_{\log,n}^{\mathrm{Mil}}(U)/p\cong p^{s-1}\mathcal{K}_{\log,n}^{\mathrm{Mil}}(U)/p^{s}\,.
\end{equation*}
Then consider the commutative diagram with exact rows
\begin{equation*}
\begin{tikzpicture}[descr/.style={fill=white,inner sep=1.5pt}]
        \matrix (m) [
            matrix of math nodes,
            row sep=3em,
            column sep=3em,
            text height=1.5ex, text depth=0.25ex
        ]
        { 0 & \mathcal{K}_{\log,n}^{\mathrm{Mil}}(U)/p & \mathcal{K}_{\log,n}^{\mathrm{Mil}}(U)/p^{s} & \mathcal{K}_{\log,n}^{\mathrm{Mil}}(U)/p^{s-1} & 0  \\
      0 & \tilde{\omega}^{n}_{Y/k,\log}(U) & W_{s}\tilde{\omega}^{n}_{Y/k,\log}(U) & W_{s-1}\tilde{\omega}^{n}_{Y/k,\log}(U) & 0 \\};

        \path[overlay,->, font=\scriptsize] 
        (m-1-1) edge (m-1-2)
        (m-1-2) edge node[above]{$\times p^{s-1}$}(m-1-3)
        (m-1-3) edge (m-1-4)
        (m-1-4) edge (m-1-5)
        (m-2-1) edge (m-2-2)
        (m-2-2) edge node[above]{$\times p^{s-1}$}(m-2-3)
        (m-2-3) edge (m-2-4)
        (m-2-4) edge (m-2-5)
        (m-1-2) edge node[right]{$d\log$} (m-2-2)
        (m-1-3) edge node[right]{$d\log$} (m-2-3)
        (m-1-4) edge node[right]{$d\log$} (m-2-4)
         ;
                 
\end{tikzpicture}
\end{equation*}
By induction, it suffices to show that the left vertical arrow is an isomorphism. Since it is surjective by definition, we need to show injectivity. The proof of Proposition \ref{K-theory smooth locus} implies that the map $\left(\mathcal{K}^{\mathrm{Mil}}_{\log,n}(U)/\langle G\rangle\right)/p\rightarrow K_{n}^{\mathrm{Mil}}(V)/p$ is injective as well. Indeed, if a symbol $\{a_{1},\ldots, a_{n}\}$ vanishes in $\displaystyle\prod_{\eta_{i}\in U^{0}}\mathcal{K}^{\mathrm{Mil}}_{n}(k(\eta_{i})/p$ then for each $\eta_{i}$ there exists $j\in\{1,\ldots, n\}$ such that $a_{j}=b_{j}^{p}+\pi_{i}x=b_{j}^{p}(1+\pi_{i}\frac{x}{b_{j}^{p}})$. One then follows the proof of Proposition \ref{K-theory smooth locus} to conclude. 

Let $\mathcal{K}_{n}^{\mathrm{Mil}}(U)$ be the image of $(\mathcal{O}(U)^{\ast})^{\otimes n}$ in $\mathcal{K}_{\log,n}^{\mathrm{Mil}}(U)$. Since the composite map $(\mathcal{O}(U)^{\ast})^{\otimes n}\rightarrow\tilde{\omega}_{Y/k,\log}^{n}(U)$ factors through the injection $j_{\ast}\mathcal{K}^{\mathrm{Mil}}_{n}(V)/p\rightarrow j_{\ast}\Omega^{n}_{Y^{\mathrm{sm}}/k,\log}(V)$ (which is an isomorphism by the Bloch-Kato-Gabber theorem \cite[Corollary 2.8]{BK86}), we see that $d\log$ restricted to $\mathcal{K}^{\mathrm{Mil}}_{n}(U)/p$ is injective. Using the exact sequence
\begin{equation*}
0\rightarrow\omega^{n-1}_{Y/k,\log}\xrightarrow{\wedge\theta}\tilde{\omega}_{Y/k,\log}^{n}\rightarrow\omega_{Y/k,\log}^{n}\rightarrow 0
\end{equation*}
we will conclude the proof as below.

Consider the composite map (which is surjective)
\begin{equation*}
\mathcal{K}^{\mathrm{Mil}}_{\log,n}(U)/p\rightarrow\tilde{\omega}^{n}_{Y/k,\log}(U)\rightarrow\omega^{n}_{Y/k,\log}(U)\,.
\end{equation*}
For $g_{i}=\pi_{i}+\displaystyle\prod_{j\neq i}\pi_{j}$ the image $d\log(g_{i})$ in $\omega^{1}_{Y/k,\log}$ has $j$-component (in $\Omega^{1}_{k(\eta_{j})/k}$) $d\log\pi_{i}$ for $j\neq i$, and $i$-component (in $\Omega^{1}_{k(\eta_{i})/k}$) $-\displaystyle\sum_{j\neq i}d\log\pi_{j}$. It is then clear that the kernel of the map
\begin{equation*}
\left.(\underline{N}_{Y}^{\mathrm{gp}}(U)/p)\middle/(\mathcal{O}(U)^{\ast}/p)\right.\rightarrow\omega^{1}_{Y/k,\log}(U)/\mathrm{image}(\mathcal{O}(U)^{\ast})
\end{equation*}
is generated by $\left(\displaystyle\prod^{r}_{i=1}g_{i}\right)^{\pm 1}$.

Since for all $a\in\underline{N}^{\mathrm{gp}}_{Y}(U)$, the symbol $\{a,a\}$ vanishes in $\mathcal{K}_{\log,2}^{\mathrm{Mil}}(U)/p$, because $\{a,-1\}$ is $p$-divisible, we see that the $\mathbb{F}_{p}$-rank of the kernel of the map 
\begin{equation*}
\left.(\mathcal{K}_{\log,n}^{\mathrm{Mil}}(U)/p)\middle/(\mathcal{K}_{n}^{\mathrm{Mil}}(U)/p)\right.\rightarrow\omega^{n}_{Y/k,\log}(U)/\mathrm{image}((\mathcal{O}(U)^{\ast})^{\otimes n})
\end{equation*}
is equal to the $\mathbb{F}_{p}$-rank of $\omega_{Y/k,\log}^{n-1}/\mathrm{image}((\mathcal{O}(U)^{\ast})^{\otimes n-1})$. But this is also the $\mathbb{F}_{p}$-rank of $\omega_{Y/k,\log}^{n-1}/\mathrm{image}((\mathcal{O}(U)^{\ast})^{\otimes n-1})\wedge\theta$. Hence the $d\log$ map 
\begin{equation*}
\mathcal{K}_{\log,n}^{\mathrm{Mil}}(U)/p\rightarrow\tilde{\omega}^{n}_{Y/k,\log}(U)\,,
\end{equation*}
which is already known to be surjective, must be an isomorphism. 
\end{enumerate}
\end{proof}

Define $\mathcal{K}_{\log,Y,\ast}^{\mathrm{Mil}}$ to be the Zariski sheafification of the presheaf $U\mapsto\mathcal{K}_{\log,Y,\ast}^{\mathrm{Mil}}(U)$. Then we have the following semistable analogue of \cite[Theorem 8.5]{GL00}:
\begin{prop}
\label{semistable Geisser-Levine}
For each $n,s\geq 0$ there is a quasi-isomorphism
\begin{equation*}
\mathbb{Z}_{\log,Y}(n)\otimes^{\mathbb{L}}\mathbb{Z}/p^{s}\simeq W_{s}\tilde{\omega}^{n}_{Y/k,\log}[-n]
\end{equation*}
in $D(Y_{\mathrm{Zar}})$.
\end{prop}
\begin{proof}
Recall from Remark \ref{mod p coefficients} that we have exact sequence
\begin{equation*}
\{G\}\cup\mathcal{H}^{i-1}(\mathbb{Z}_{\log,Y}(n-1)\otimes^{\mathbb{L}}\mathbb{Z}/p^{s})\rightarrow \mathcal{H}^{i}(\mathbb{Z}_{\log,Y}(n)\otimes^{\mathbb{L}}\mathbb{Z}/p^{s})\rightarrow u_{\ast}\mathcal{H}^{i}(\mathbb{Z}_{Y^{\mathrm{sm}}}(n)\otimes^{\mathbb{L}}\mathbb{Z}/p^{s})\,
\end{equation*}
where $u:Y^{\mathrm{sm}}\hookrightarrow Y$ is the inclusion of the smooth locus. By \cite[Theorem 8.3]{GL00} we have 
\begin{equation*}
\mathcal{H}^{i}(\mathbb{Z}_{Y^{\mathrm{sm}}}(n)\otimes^{\mathbb{L}}\mathbb{Z}/p^{s})\simeq\left\{
\begin{array}{ll}
      0 & \text{if }i\neq n \\
      W_{s}\Omega^{n}_{Y^{\mathrm{sm}}/k,\log} & \text{if }i=n \\
\end{array} 
\right.
\end{equation*}
By induction on $n$ (the case $n=1$ being clear) we may assume that 
\newline $\mathcal{H}^{i}(\mathbb{Z}_{\log,Y}(n-1)\otimes^{\mathbb{L}}\mathbb{Z}/p^{s})$ vanishes for $i\neq n-1$. Using the above exact sequence we see that $\mathbb{Z}_{\log,Y}(n)\otimes^{\mathbb{L}}\mathbb{Z}/p^{s}$ is acyclic outside of cohomological degree $n$. It therefore suffices to show that $\mathcal{H}^{n}(\mathbb{Z}_{\log,Y}(n)\otimes^{\mathbb{L}}\mathbb{Z}/p^{s})\simeq W_{s}\tilde{\omega}_{Y/k,\log}^{n}$. To see this, the above vanishing and the exact triangle
\begin{equation*}
\mathbb{Z}_{\log,Y}(n)\rightarrow\mathbb{Z}_{\log,Y}(n)\rightarrow\mathbb{Z}_{\log,Y}(n)\otimes^{\mathbb{L}}\mathbb{Z}/p^{s}\xrightarrow{+1}
\end{equation*}
 gives a short exact sequence 
\begin{equation*}
0\rightarrow\mathcal{H}^{n}(\mathbb{Z}_{\log,Y}(n))\rightarrow\mathcal{H}^{n}(\mathbb{Z}_{\log,Y}(n))\rightarrow\mathcal{H}^{n}(\mathbb{Z}_{\log,Y}(n)\otimes^{\mathbb{L}}\mathbb{Z}/p^{s})\rightarrow 0
\end{equation*}
fitting into the following commutative diagram
\begin{equation*}
\begin{tikzpicture}[descr/.style={fill=white,inner sep=1.5pt}]
        \matrix (m) [
            matrix of math nodes,
            row sep=3em,
            column sep=2em,
            text height=1.5ex, text depth=0.25ex
        ]
        { 0 & \mathcal{H}^{n}(\mathbb{Z}_{\log,Y}(n)) & \mathcal{H}^{n}(\mathbb{Z}_{\log,Y}(n)) & \mathcal{H}^{n}(\mathbb{Z}_{\log,Y}(n)\otimes^{\mathbb{L}}\mathbb{Z}/p^{s}) & 0 \\
       0 & \mathcal{K}_{\log,Y,n}^{\mathrm{Mil}} & \mathcal{K}_{\log,Y,n}^{\mathrm{Mil}} & \mathcal{K}_{\log,Y,n}^{\mathrm{Mil}}/p^{s} & 0 \\};

        \path[overlay,->, font=\scriptsize] 
        (m-1-1) edge (m-1-2)
        (m-1-2) edge (m-1-3)
        (m-1-3) edge (m-1-4)
        (m-1-4) edge (m-1-5)
        (m-2-1) edge (m-2-2)
        (m-2-2) edge node[above]{$p^{s}$} (m-2-3)
        (m-2-3) edge (m-2-4)
        (m-2-4) edge (m-2-5)
        (m-2-2) edge node[right]{$\wr$} (m-1-2)
        (m-2-3) edge node[right]{$\wr$} (m-1-3)
              
         ;
                                        
\end{tikzpicture}
\end{equation*}
where the isomorphisms $\mathcal{K}_{\log,Y,n}^{\mathrm{Mil}}\simeq\mathcal{H}^{n}(\mathbb{Z}_{\log,Y}(n))$ are by Theorem \ref{K-theory and motivic} (when $k$ is finite use the improved logarithmic $K$-theory as in Remark \ref{refined remark}(ii)). The map $p^{s}:\mathcal{K}_{\log,Y,n}^{\mathrm{Mil}}\rightarrow \mathcal{K}_{\log,Y,n}^{\mathrm{Mil}}$ in the lower sequence is injective because $\mathcal{K}_{\log,Y,n}^{\mathrm{Mil}}$ is $p$-torsion free by Theorem \ref{mod-p K-theory}(a). Hence the lower sequence is also exact and we conclude that there is an induced isomorphism $\mathcal{K}_{\log,Y,n}^{\mathrm{Mil}}/p^{s}\xrightarrow{\sim}\mathcal{H}^{n}(\mathbb{Z}_{\log,Y}(n)\otimes^{\mathbb{L}}\mathbb{Z}/p^{s})$. Then the proposition follows from Theorem \ref{mod-p K-theory}(b).
\end{proof}

\section{Log-syntomic cohomology and the $p$-adic variational Hodge conjecture}\label{p-adic deformation section}

Let $k$ be a perfect field of characteristic $p>2$, and let $K=\mathrm{Frac}\,W(k)$. In this section we fix a natural number $n<p$.  Let $X$ be a scheme over $W(k)$ with semistable reduction, that is \'{e}tale locally on $X$ the structure morphism factors as \begin{equation*}
X\xrightarrow{u}\mathrm{Spec}\,W(k)[t_{1},\ldots,t_{a}]/(t_{1}\cdots t_{b}-p)\xrightarrow{\delta}\mathrm{Spec}\,W(k)
\end{equation*}
for some $a\geq b$, where $u$ is a smooth morphism and $\delta$ is induced by the diagonal map. Then the generic fibre $X_{K}$ is smooth and the special fibre $Y$ is a reduced normal crossings divisor on $X$. If  $Y$ is endowed with the inverse image $M_{Y}$ of the divisorial log-structure $M_{X}$ associated to $Y\hookrightarrow X$, then $(Y,M_{Y})$ is a semistable variety in the sense of \S\ref{finite log section}. For each $m\in\mathbb{N}$, set $X_{m}=X\times_{W(k)}W_{m}(k)$ and let $M_{X_{m}}$ be the pullback (in the sense of log-structures) of $M_{X}$ along the closed immersion $\iota_{m}:X_{m}\hookrightarrow X$. Then $(X_{m},M_{X_{m}})$ is a log-scheme over $(\mathrm{Spec}\,W_{m}(k),L_{m})$ where $L_{m}$ is the log-structure associated to $\mathbb{N}\rightarrow W_{m}(k)$, $1\mapsto p$. In the case $m=1$ we have $(X_{1},M_{X_{1}})=(Y,M_{Y})$.

In order to construct a log-motivic complex $\mathbb{Z}_{\log,X_{\bigcdot}}(n)$ as a pro-complex in the derived category in the sense of \cite{BEK14}, we need a good definition of log-syntomic complexes. By this we mean a complex that allows us to glue the log-motivic complex $\mathbb{Z}_{\log,Y}(n)$ defined in \S\ref{log-motivic cohomology section} along a logarithmic (Hyodo-Kato) Hodge-Witt sheaf, using Theorem \ref{K-theory and motivic} and Theorem \ref{mod-p K-theory}. In \cite[\S3]{NN16} a complex $R\Gamma(X,s_{\log}(n))$ is defined and is identified with the homotopy limit (or iterated fibre) of the square 
\begin{equation}
\label{homotopy limit}
\begin{tikzpicture}[descr/.style={fill=white,inner sep=1.5pt}]
        \matrix (m) [
            matrix of math nodes,
            row sep=3em,
            column sep=5em,
            text height=1.5ex, text depth=0.25ex
        ]
        { R\Gamma_{\mathrm{HK}}(X)_{\mathbb{Q}} & R\Gamma_{\mathrm{HK}}(X)_{\mathbb{Q}}\oplus R\Gamma_{\mathrm{dR}}(X)/\mathrm{Fil}^{n}  \\
       R\Gamma_{\mathrm{HK}}(X)_{\mathbb{Q}} & R\Gamma_{\mathrm{HK}}(X)_{\mathbb{Q}} \\};

        \path[overlay,->, font=\scriptsize] 
        (m-1-1) edge node[left]{$N$} (m-2-1)
        (m-1-1) edge node[above]{$(1-\varphi_{n},\iota_{\mathrm{dR})}$} (m-1-2)
        (m-1-2) edge node[right]{$(N,0)$} (m-2-2)
        (m-2-1) edge node[above]{$1-\varphi_{n-1}$} (m-2-2)
         ;
                                        
\end{tikzpicture}
\end{equation}
where $R\Gamma_{\mathrm{HK}}(X)_{\mathbb{Q}}$ is the Hyodo-Kato cohomology, $\iota_{\mathrm{dR}}$ is induced by the Hyodo-Kato isomorphism and $\varphi_{n}$ is the divided Frobenius ``$\frac{\varphi}{p^{n}}$''. We will give an equivalent description of $R\Gamma(X,s_{\log}(n))$ using the logarithmic Hyodo-Kato sheaves. We can reconstruct the commutative diagram \eqref{homotopy limit} by applying $R\Gamma$ to a commutative diagram of pro-sheaves in the category $\mathbb{Q}\otimes D_{\mathrm{pro}}(Y_{\mathrm{\acute{e}t}})$ (the isogeny category of the category $D_{\mathrm{pro}}(Y_{\mathrm{\acute{e}t}})$ in \cite[Definition A.3]{BEK14}). Namely 
\begin{equation}
\label{HK diagram}
\begin{tikzpicture}[descr/.style={fill=white,inner sep=1.5pt}]
        \matrix (m) [
            matrix of math nodes,
            row sep=3em,
            column sep=5em,
            text height=1.5ex, text depth=0.25ex
        ]
        { \mathbb{Q}\otimes W_{\bigcdot}\omega_{Y/k}^{\bullet} & \mathbb{Q}\otimes W_{\bigcdot}\omega_{Y/k}^{\bullet}\oplus \omega_{X_{\bigcdot}/W(k)}^{\bullet}\otimes\mathbb{Q}/\mathrm{Fil}^{n}  \\
       \mathbb{Q}\otimes W_{\bigcdot}\omega_{Y/k}^{\bullet} & \mathbb{Q}\otimes W_{\bigcdot}\omega_{Y/k}^{\bullet} \\};

        \path[overlay,->, font=\scriptsize] 
        (m-1-1) edge node[left]{$N$} (m-2-1)
        (m-1-1) edge node[above]{$(1-\varphi_{n},\iota_{\mathrm{dR})}$} (m-1-2)
        (m-1-2) edge node[right]{$(N,0)$} (m-2-2)
        (m-2-1) edge node[above]{$1-\varphi_{n-1}$} (m-2-2)
         ;
                                        
\end{tikzpicture}
\end{equation}
Here $\iota_{dR}:\mathbb{Q}\otimes W_{\bigcdot}\omega_{Y/k}^{\bullet}\rightarrow\omega_{X_{\bigcdot}/W(k)}^{\bullet}\otimes\mathbb{Q}$ is the Hyodo-Kato isomorphism \cite[5.4]{HK94}, where $\omega_{X_{\bigcdot}/W(k)}^{\bullet}$ is the logarithmic de Rham pro-complex induced by $\omega_{X/W(k)}^{\bullet}$ with locally free components $\omega^{i}_{X/W(k)}=\bigwedge^{i}\omega^{1}_{X/W(k)}$, where $\omega^{1}_{X/W(k)}$ is generated by $dt_{i}/t_{i}$ for $1\leq i\leq b$ and $dt_{i}$ for $i> b$, subject to the relation $\sum_{i=1}^{b}dt_{i}/t_{i}=0$. Using the Hyodo-Kato exact sequence \cite[Proposition 1.5]{HK94}
\begin{equation}
\label{HK sequence}
0\rightarrow W_{\bigcdot}\omega^{\bullet}_{Y/k}[-1]\xrightarrow{\wedge\theta} W_{\bigcdot}\tilde{\omega}^{\bullet}_{Y/k}\rightarrow W_{\bigcdot}\omega^{\bullet}_{Y/k}\rightarrow 0
\end{equation} we can redefine the homotopy limit of \eqref{HK diagram} as 
\begin{equation}
\label{syntomic cone}
\mathfrak{S}_{\log, X_{\bigcdot}}(n)_{\mathrm{\acute{e}t}}=\mathrm{Cone}(W_{\bigcdot}\tilde{\omega}_{Y/k}^{\bullet}\otimes\mathbb{Q}\xrightarrow{(1-\varphi_{n},\iota_{\mathrm{dR}})}W_{\bigcdot}\tilde{\omega}_{Y/k}^{\bullet}\otimes\mathbb{Q}\oplus\omega_{X_{\bigcdot}/W(k)}^{\bullet}/\mathrm{Fil}^{n}\otimes\mathbb{Q})\,.
\end{equation}
Here $\mathrm{Fil}^{n}$ is the Hodge filtration and $\iota_{\mathrm{dR}}$ is the composite map
\begin{equation*}
W_{\bigcdot}\tilde{\omega}_{Y/k}^{\bullet}\otimes\mathbb{Q}\twoheadrightarrow W_{\bigcdot}\omega_{Y/k}^{\bullet}\otimes\mathbb{Q}\xrightarrow{\simeq}\omega^{\bullet}_{X_{\bigcdot}/W(k)}\otimes\mathbb{Q}\rightarrow\omega_{X_{\bigcdot}/W(k)}^{\bullet}/\mathrm{Fil}^{n}\otimes\mathbb{Q}\,.
\end{equation*}
Note that the mapping cone of the monodromy operator $N$ is by definition the Hyodo-Kato complex $W_{\bigcdot}\tilde{\omega}^{\bullet}_{Y/k}$, so the diagram involving $N$ and $1-\varphi$ only in \eqref{HK diagram} is equivalent to the cone of the map $1-\varphi_{\cdot}$ on $W_{\bigcdot}\tilde{\omega}^{\bullet}_{Y/k}$. Since the right vertical map on $\omega_{X_{\bigcdot}/W(k)}^{\bullet}\otimes\mathbb{Q}/\mathrm{Fil}^{n}$ is the zero map, the map $\iota_{\mathrm{dR}}$ in \eqref{HK diagram} becomes the composite map $\iota_{\mathrm{dR}}$ in \eqref{syntomic cone}. Hence the homotopy limits of \eqref{HK diagram} and \eqref{syntomic cone} coincide.

We can further simplify the construction by introducing the Nygaard complexes on the level of $W_{\bigcdot}\tilde{\omega}_{Y/k}^{\bullet}$: for each $s\geq 0$ they are defined via an exact sequence
\begin{equation}
\label{Nygaard HK sequence}
0\rightarrow N^{s-1}W_{\bigcdot}\omega^{\bullet}_{Y/k}[-1]\xrightarrow{\wedge\theta}N^{s}W_{\bigcdot}\tilde{\omega}^{\bullet}_{Y/k}\rightarrow N^{s}W_{\bigcdot}\omega^{\bullet}_{Y/k}\rightarrow 0\,.
\end{equation} 
with relations $\varphi(\theta)=p\theta$, $d\theta=\theta d=0$ and $V(\theta)=\theta$.
\begin{lemma}\label{Frobenius invariants Nygaard}
For each $s\geq 0$ there is an exact sequence of pro-complexes
\begin{equation*}
0\rightarrow W_{\bigcdot}\tilde{\omega}^{s}_{Y/k,\log}[-s]\rightarrow N^{s}W_{\bigcdot}\tilde{\omega}^{\bullet}_{Y/k}\xrightarrow{1-\varphi_{s}} W_{\bigcdot}\tilde{\omega}^{\bullet}_{Y/k}\rightarrow 0
\end{equation*} 
on $Y_{\mathrm{\acute{e}t}}$.
\end{lemma} 
\begin{proof}
Consider the following commutative diagram
\begin{equation*}
\begin{tikzpicture}[descr/.style={fill=white,inner sep=1.5pt}]
        \matrix (m) [
            matrix of math nodes,
            row sep=2.5em,
            column sep=2.5em,
            text height=1.5ex, text depth=0.25ex
        ]
        { \ & 0 & 0 & 0 & \ \\       
        0 & W_{\bigcdot}\omega^{s-1}_{Y/k,\log}[-s] & N^{s-1}W_{\bigcdot}\omega^{\bullet}_{Y/k}[-1] & W_{\bigcdot}\omega_{Y/k}^{\bullet}[-1] & 0  \\
        0 & W_{\bigcdot}\tilde{\omega}^{s}_{Y/k,\log}[-s] & N^{s}W_{\bigcdot}\tilde{\omega}^{\bullet}_{Y/k} & W_{\bigcdot}\tilde{\omega}_{Y/k}^{\bullet} & 0 \\
    0 & W_{\bigcdot}\omega^{s}_{Y/k,\log}[-s] & N^{s}W_{\bigcdot}\omega^{\bullet}_{Y/k} & W_{\bigcdot}\omega_{Y/k}^{\bullet} & 0  \\  
    \ & 0 & 0 & 0 & \ \\ };

        \path[overlay,->, font=\scriptsize] 
        (m-1-2) edge (m-2-2)
        (m-1-3) edge (m-2-3)
        (m-1-4) edge (m-2-4)
        (m-2-2) edge node[right]{$\wedge\theta$} (m-3-2)
        (m-2-3) edge node[right]{$\wedge\theta$} (m-3-3)
        (m-2-4) edge node[right]{$\wedge\theta$} (m-3-4)
        (m-3-2) edge (m-4-2)
        (m-3-3) edge (m-4-3)
        (m-3-4) edge (m-4-4)
        (m-4-2) edge (m-5-2)
        (m-4-3) edge (m-5-3)
        (m-4-4) edge (m-5-4)
        (m-4-1) edge (m-4-2)
        (m-4-2) edge (m-4-3)
        (m-4-3) edge node[above] {$1-\varphi_{s}$}(m-4-4)
        (m-4-4) edge (m-4-5)
        (m-2-1) edge (m-2-2)
        (m-2-2) edge (m-2-3)
        (m-2-3) edge node[above] {$1-\varphi_{s-1}$}(m-2-4)
        (m-2-4) edge (m-2-5)
        (m-3-1) edge (m-3-2)
        (m-3-2) edge (m-3-3)
        (m-3-3) edge node[above] {$1-\varphi_{s}$}(m-3-4)
        (m-3-4) edge (m-3-5)
         ;
                                        
\end{tikzpicture}
\end{equation*}
The vertical sequences are exact; the rightmost sequence is \eqref{HK sequence}, the middle sequence is \eqref{Nygaard HK sequence} and the leftmost sequence is exact by the definition of $W_{r}\omega^{s}_{Y/k,\log}$ and $W_{r}\tilde{\omega}^{s}_{Y/k,\log}$. The statement of the lemma is therefore equivalent to the exactness of the sequence
\begin{equation*}
0\rightarrow W_{\bigcdot}\omega^{s}_{Y/k,\log}[-s]\rightarrow N^{s}W_{\bigcdot}\omega^{\bullet}_{Y/k}\xrightarrow{1-\varphi_{s}} W_{\bigcdot}\omega_{Y/k}^{\bullet}\rightarrow 0\,.
\end{equation*}
for each $s\geq 0$.
To see this, first note that $1-\varphi_{s}:W_{\bigcdot}\omega_{Y/k}^{s+i}\rightarrow W_{\bigcdot}\omega_{Y/k}^{s+i}$ is an isomorphism for all $i>0$ and $s\geq 0$ by the same proof as \cite[I. Lemme 3.30]{Ill79}. Next, observe that $1-\varphi_{s}:\tau_{<s}N^{s}W_{\bigcdot}\omega_{Y/k}^{\bullet}\rightarrow \tau_{<s}W_{\bigcdot}\omega_{Y/k}^{\bullet}$ is an isomorphism. Indeed, let $i\leq s-1$. Then for $\beta$ a local section of $W_{\bigcdot}\omega_{Y/k}^{i}$ we have $\beta=(p^{s-1-i}V-\mathrm{id})\alpha$ where $\alpha=-(p^{s-1-i}V)\displaystyle\sum_{m=0}^{\infty}(p^{s-1-i}V)^{m}\beta$, so $1-\varphi_{s}$ is surjective. On the other hand, if $\alpha$ is a local section of $W_{\bigcdot}\omega_{Y/k}^{i}$ such that $\alpha=p^{s-1-i}V\alpha$, we get $\alpha\in(p^{s-1-i}V)^{m}W_{\bigcdot}\omega_{Y/k}^{i}$ for all $m\geq 0$, and hence $\alpha=0$ so $1-\varphi_{s}$ is injective. Finally, we must show that the sequence
\begin{equation}
\label{log kernel}
0\rightarrow W_{r}\omega^{s}_{Y/k,\log}\rightarrow W_{r}\omega^{s}_{Y/k}/dVW_{r-1}\omega_{Y/k}^{s-1}\xrightarrow{1-\varphi_{s}} W_{r}\omega^{s}_{Y/k}/dW_{r}\omega_{Y/k}^{s-1}\rightarrow 0
\end{equation}
is exact for each $r\geq 1$. To see this, consider the following commutative diagram
\begin{equation*}
\begin{tikzpicture}[descr/.style={fill=white,inner sep=1.5pt}]
        \matrix (m) [
            matrix of math nodes,
            row sep=2.5em,
            column sep=1.5em,
            text height=1.5ex, text depth=0.25ex
        ]
        { \ & \ & 0 & 0 & \ \\       
        \ & \ & dVW_{r-1}\omega^{s-1}_{Y/k} & dW_{r}\omega_{Y/k}^{s-1}/dV^{r-1}\omega_{Y/k}^{s-1} & \  \\
        0 & W_{r}\omega^{s}_{Y/k,\log} & W_{r}\omega^{s}_{Y/k} & W_{r}\omega_{Y/k}^{s}/dV^{r-1}\omega_{Y/k}^{s-1} & 0 \\
    0 & W_{r}\omega^{s}_{Y/k,\log} & W_{r}\omega^{s}_{Y/k}/dVW_{r-1}\omega^{s-1}_{Y/k} & W_{r}\omega_{Y/k}^{s}/dW_{r}\omega_{Y/k}^{s-1} & 0  \\  
    \ & \ & 0 & 0 & \ \\ };

        \path[overlay,->, font=\scriptsize] 
        
        (m-1-3) edge (m-2-3)
        (m-1-4) edge (m-2-4)
        (m-2-3) edge (m-3-3)
        (m-2-4) edge (m-3-4)
        (m-3-2) edge node[right]{$=$}(m-4-2)
        (m-3-3) edge (m-4-3)
        (m-3-4) edge (m-4-4)
        (m-4-3) edge (m-5-3)
        (m-4-4) edge (m-5-4)
        (m-4-1) edge (m-4-2)
        (m-4-2) edge (m-4-3)
        (m-4-3) edge node[above] {$1-\varphi_{s}$}(m-4-4)
        (m-4-4) edge (m-4-5)
        (m-2-3) edge node[above] {$1-\varphi_{s}$}(m-2-4)
        (m-3-1) edge (m-3-2)
        (m-3-2) edge (m-3-3)
        (m-3-3) edge node[above] {$1-\varphi_{s}$}(m-3-4)
        (m-3-4) edge (m-3-5)
         ;
                                        
\end{tikzpicture}
\end{equation*}
The two vertical sequences are obviously exact, and the middle horizontal sequence is exact by \cite[Corollary 2.13]{Lor02}. Therefore \eqref{log kernel} is exact if and only if $1-\varphi_{s}:dVW_{r-1}\omega^{s-1}_{Y/k}\rightarrow dW_{r}\omega_{Y/k}^{s-1}/dV^{r-1}\omega_{Y/k}^{s-1}$ is an isomorphism. The map $V:dW_{r}\omega_{Y/k}^{s-1}\rightarrow W_{r+1}\omega_{Y/k}^{s}$ factors through $p:W_{r}\omega_{Y/k}^{s}\rightarrow W_{r+1}\omega^{s}_{Y/k}$, as
\begin{equation*}
\begin{tikzpicture}[descr/.style={fill=white,inner sep=1.5pt}]
        \matrix (m) [
            matrix of math nodes,
            row sep=2.5em,
            column sep=1.5em,
            text height=1.5ex, text depth=0.25ex
        ]
        { dW_{r}\omega_{Y/k}^{s-1} & \ & W_{r+1}\omega_{Y/k}^{s}\\
        \ & W_{r}\omega_{Y/k}^{s} & \ \\ };

        \path[overlay,->, font=\scriptsize] 
        
        (m-1-1) edge node[above]{$V$}(m-1-3)
        (m-2-2) edge node [below right]{$p$} (m-1-3) 
        (m-1-1) edge node [below left]{$\psi$} (m-2-2);
                                        
\end{tikzpicture}
\end{equation*}
and since $Vd=pdV$, the map $\psi$ has image contained in $dVW_{n-1}\omega_{Y/k}^{r-1}$. The map $\psi+\psi^{2}+\psi^{3}+\cdots$ is the inverse of $1-\varphi_{s}:dVW_{r-1}\omega^{s-1}_{Y/k}\rightarrow dW_{r}\omega_{Y/k}^{s-1}/dV^{r-1}\omega_{Y/k}^{s-1}$.
\end{proof}
Since $N^{n}W_{\bigcdot}\tilde{\omega}^{\bullet}_{Y/k}\otimes\mathbb{Q}\simeq W_{\bigcdot}\tilde{\omega}^{\bullet}_{Y/k}\otimes\mathbb{Q}$ we get that \eqref{syntomic cone} is quasi-isomorphic to
\begin{equation*}
\mathrm{Cone}(W_{\bigcdot}\tilde{\omega}_{Y/k,\log}^{n}[-n]\otimes\mathbb{Q}\rightarrow \omega_{X_{\bigcdot}/W(k)}^{\bullet}/\mathrm{Fil}^{n}\otimes\mathbb{Q})\,.
\end{equation*}
We can then modify the definition of $\mathfrak{S}_{\log, X_{\bigcdot}}(n)_{\mathrm{\acute{e}t}}$ again to get the following interpretation
\begin{equation}
\mathfrak{S}_{\log, X_{\bigcdot}}(n)_{\mathrm{\acute{e}t}}=\mathrm{Cone}(W_{\bigcdot}\tilde{\omega}_{Y/k,\log}^{n}[-n]\otimes\mathbb{Q}\rightarrow\omega_{X_{\bigcdot}/W(k)}^{<n}\otimes\mathbb{Q})
\end{equation}
where the map is defined by the composition
\begin{align*}
W_{\bigcdot}\tilde{\omega}_{Y/k,\log}^{n}[-n]\otimes\mathbb{Q}\rightarrow W_{\bigcdot}\omega_{Y/k,\log}^{n}[-n]\otimes\mathbb{Q}\rightarrow 
& N^{n}W_{\bigcdot}\omega_{Y/k}^{\bullet}\otimes\mathbb{Q}\simeq W_{\bigcdot}\omega_{Y/k}^{\bullet}\otimes\mathbb{Q}\\
&\simeq\omega_{X_{\bigcdot}/W(k)}^{\bullet}\otimes\mathbb{Q}\rightarrow\omega_{X_{\bigcdot}/W(k)}^{<n}\otimes\mathbb{Q}\,.
\end{align*}
where $ W_{\bigcdot}\omega_{Y/k}^{\bullet}\otimes\mathbb{Q}\simeq\omega_{X_{\bigcdot}/W(k)}^{\bullet}\otimes\mathbb{Q}$ is the Hyodo-Kato isomorphism. Then we still have that $R\Gamma(X,\mathfrak{S}_{\log, X_{\bigcdot}}(n)_{\mathrm{\acute{e}t}})$ is quasi-isomorphic to the Nekov\'{a}\v{r}-Nizio{\l} complex $R\Gamma(X,s_{\log}(n))$. By definition, we have have an exact triangle
\begin{equation}
\label{etale fundamental triangle}
\omega_{X_{\bigcdot}/W(k)}^{<n}\otimes\mathbb{Q}[-1]\rightarrow\mathfrak{S}_{\log, X_{\bigcdot}}(n)_{\mathrm{\acute{e}t}}\rightarrow W_{\bigcdot}\tilde{\omega}_{Y/k,\log}^{n}[-n]\otimes\mathbb{Q}\xrightarrow{+1}
\end{equation}
in $\mathbb{Q}\otimes D_{\mathrm{pro}}(Y_{\mathrm{\acute{e}t}})$. Define $\mathfrak{S}_{\log, X_{\bigcdot}}(n):=\tau_{\leq n}R\epsilon_{\ast}\mathfrak{S}_{\log, X_{\bigcdot}}(n)_{\mathrm{\acute{e}t}}$ where $\epsilon:(X_{\bigcdot})_{\mathrm{\acute{e}t}}\rightarrow (X_{\bigcdot})_{\mathrm{Zar}}$ is the morphism of sites. Since we have an isomorphism
\begin{equation*}
\epsilon_{\ast}\omega_{X_{\bigcdot}/W(k)}^{<n}\otimes\mathbb{Q}[-1]\xrightarrow{\sim}R\epsilon_{\ast}\omega_{X_{\bigcdot}/W(k)}^{<n}\otimes\mathbb{Q}[-1]
\end{equation*}
in $D_{\mathrm{pro}}(Y_{\mathrm{Zar}})$, the complex $R\epsilon_{\ast}\omega_{X_{\bigcdot}/W(k)}^{<n}\otimes\mathbb{Q}[-1]$ has cohomological support in degrees $[1,n]$. By \cite[Lemma A.1]{BEK14}, applying $\tau_{\leq n}\circ R\epsilon_{\ast}$ to \eqref{etale fundamental triangle} therefore gives an exact triangle
  \begin{equation}
\label{fundamental triangle}
\omega_{X_{\bigcdot}/W(k)}^{<n}\otimes\mathbb{Q}[-1]\rightarrow\mathfrak{S}_{\log, X_{\bigcdot}}(n)\rightarrow W_{\bigcdot}\tilde{\omega}_{Y/k,\log}^{n}[-n]\otimes\mathbb{Q}\xrightarrow{+1}
\end{equation}
in $\mathbb{Q}\otimes D_{\mathrm{pro}}(Y_{\mathrm{Zar}})$, which is a log-syntomic analogue of the ``Fundamental triangle'' of \cite[Theorem 5.4]{BEK14}.

In order to glue a log-syntomic complex with the log-motivic complex $\mathbb{Z}_{\log,Y}(n)$ along the logarithmic Hyodo-Kato sheaf $W_{\bigcdot}\tilde{\omega}^{n}_{Y/k,\log}$ using the canonical map 
\begin{equation*}
\mathbb{Z}_{\log,Y}(n)\rightarrow\mathcal{H}^{n}(\mathbb{Z}_{\log,Y}(n))[-n]\simeq\mathcal{K}_{\log,Y,n}^{\mathrm{Mil}}[-n]\rightarrow W_{\bigcdot}\tilde{\omega}_{Y/k,\log}^{n}[-n]
\end{equation*}
we need an integral version of the complex $\mathfrak{S}_{\log,X_{\bigcdot}}(n)$. Let $M_{X}$ be the divisorial log-structure associated to $Y\hookrightarrow X$, and for each $m\in\mathbb{N}$ let $M_{X_{m}}$ be the pullback log-structure on $X_{m}$. Let $(\mathrm{Spec}\,W_{m}(k),W_{m}(L))\rightarrow(\mathrm{Spec}\,W_{m}(k)[T],\mathscr{L})$ be the closed immersion with log-structure $\mathscr{L}$ on $\mathrm{Spec}\,W_{m}(k)[T]$ associated to $\mathbb{N}\rightarrow W_{m}(k)[T]$, $1\mapsto T$ as in \cite[(3.6)]{HK94}. Then $(\mathrm{Spec}\,W_{m}[T],\mathscr{L})\rightarrow\mathrm{Spec}\,W_{m}(k)$ equipped with the trivial log-structure is smooth. Let $(X^{\bigcdot},M^{\bigcdot})\hookrightarrow (Z^{\bigcdot},N^{\bigcdot})$ be an embedding system for $(X,M_{X})\rightarrow\mathrm{Spec}\,W(k)$ which -- under the composite map $(Y,M_{Y})\rightarrow (X^{\bigcdot},M^{\bigcdot})\rightarrow (Z^{\bigcdot},N^{\bigcdot})$ -- is an embedding system for $(Y,M_{Y})\rightarrow(\mathrm{Spec}\,W_{m}(k)[T],\mathscr{L})$. We may therefore use it for the integral definition of the log-syntomic complex due to Kato \cite{Kat94} and Tsuji \cite{Tsu99}, which we now recall. Note that $(Z^{\bigcdot},N^{\bigcdot})$ is smooth over $(\mathrm{Spec}\,W(k)[T],\mathscr{L})$ and we can assume that
\begin{equation*}
\begin{tikzpicture}[descr/.style={fill=white,inner sep=1.5pt}]
        \matrix (m) [
            matrix of math nodes,
            row sep=2.5em,
            column sep=1.5em,
            text height=1.5ex, text depth=0.25ex
        ]
        { (X^{\bigcdot},M^{\bigcdot}) & (Z^{\bigcdot},N^{\bigcdot}) \\
        (\mathrm{Spec}\,W(k),L) & (\mathrm{Spec}\,W(k)[T],\mathscr{L}) \\ };

        \path[overlay,->, font=\scriptsize] 
        
        (m-1-1) edge (m-1-2)
        (m-2-1) edge node[above]{$T\mapsto p$} (m-2-2) 
        (m-1-1) edge (m-2-1)
        (m-1-2) edge (m-2-2);
                                        
\end{tikzpicture}
\end{equation*}
is cartesian. Let $X_{m}^{i}=X^{i}\otimes\mathbb{Z}/p^{m}\mathbb{Z}$ and $Z^{i}_{m}=Z^{i}\otimes\mathbb{Z}/p^{m}\mathbb{Z}$, with induced log-structures $M_{m}^{i}$ and $N_{m}^{i}$, respectively. We assume that there exists a lifting of Frobenius $F:(Z^{\bigcdot},N^{\bigcdot})\rightarrow (Z^{\bigcdot},N^{\bigcdot})$ of the absolute Frobenius on $(Z_{1}^{\bigcdot},N_{1}^{\bigcdot})$. Let $(D_{m}^{i},M_{D_{m}^{i}})\rightarrow (Z^{i}_{m},N^{i}_{m})$ be the PD-envelope of $(X^{i}_{m},M^{i}_{m})\rightarrow (Z^{i}_{m},N^{i}_{m})$, and let $J^{[n]}_{D^{i}_{m}}\subset\mathcal{O}_{D^{i}_{m}}$ be the $n$-th divided power of $J_{D_{m}^{i}}:=\ker(\mathcal{O}_{D_{m}^{i}}\rightarrow\mathcal{O}_{X_{m}^{i}})$. Let $j^{\log}_{m,X^{\bigcdot}}(n)$ be the complex on the \'{e}tale site of $X^{\bigcdot}$ which on each $X^{i}$ is the complex 
\begin{equation*}
J^{[n]}_{D_{m}^{i}}\xrightarrow{d}J^{[n-1]}_{D_{m}^{i}}\otimes_{\mathcal{O}_{Z_{m}^{i}}}\omega^{1}_{Z_{m}^{i}/W_{m}(k)}\xrightarrow{d}\cdots\xrightarrow{d}J^{[n-q]}_{D_{m}^{i}}\otimes_{\mathcal{O}_{Z_{m}^{i}}}\omega^{q}_{Z_{m}^{i}/W_{m}(k)}\xrightarrow{d}\cdots\,.
\end{equation*}
Let $\varphi:\mathcal{O}_{D_{m}^{i}}\rightarrow\mathcal{O}_{D_{m}^{i}}$ be the Frobenius induced by $F$. Then we have $\varphi(J^{[n]}_{D_{m}^{i}})\subset p^{n}\mathcal{O}_{D_{m}^{i}}$. Define $p^{-n}\varphi:J^{[n]}_{D_{m}^{i}}\rightarrow\mathcal{O}_{D_{m}^{i}}$ by $p^{-n}\varphi(a\mod p^{m})=b\mod p^{m}$ for $a\in J^{[n]}_{D_{m+n}^{i}}$ and $b\in\mathcal{O}_{D_{m+n}^{i}}$ such that $\varphi(a)=p^{n}b$. This induces a homomorphism of complexes $p^{-n}\varphi:j^{\log}_{m,X^{\bigcdot}}(n)\rightarrow j^{\log}_{m,X^{\bigcdot}}(0)$ which is $p^{q-n}\varphi$ on $J^{[n-q]}_{D_{m}^{i}}$ and $p^{-q}\varphi$ on $\omega_{Z_{m}^{i}/W_{m}(k)}^{q}$. We make the assumption that there exist sections $T_{1},\ldots, T_{d}$ of $M_{Z}$ such that $d\log T_{i}$ $(1\leq i\leq d)$ form a basis of $\omega^{1}_{Z/W(k)}$ and $F^{\ast}(T_{i})=T_{i}^{p}$ ($1\leq i\leq d$) (see \cite[(2.1.1)]{Tsu99}). Define ${s_{m,X^{\bigcdot}}^{\log}}'(n)$ to be the mapping fibre of $1-p^{-n}\varphi:j_{m,X^{\bigcdot}}^{\log}(n)\rightarrow j_{m,X^{\bigcdot}}^{\log}(0)$, and set ${s_{m,X}^{\log}}'(n)=R\theta_{\ast}{s_{m,X^{\bigcdot}}^{\log}}'(n)$ where $\theta:(X^{\bigcdot})_{\mathrm{\acute{e}t}}^{\sim}\rightarrow (X)^{\sim}_{\mathrm{\acute{e}t}}$ is the morphism induced by the hypercovering $X^{\bigcdot}\rightarrow X$.

Next we shall recall Tsuji's definition of log-syntomic regulators \cite[\S2.2]{Tsu99}. Let $C_{m}$ be the complex (which is quasi-isomorphic to $\theta^{\ast}M_{X_{m}}^{\mathrm{gp}}[-1]$) given by
\begin{align*}
& 1+J_{D_{m}^{\bigcdot}}\rightarrow M_{D_{m}^{\bigcdot}}^{\mathrm{gp}} \\
& \ \deg 0 \ \ \ \ \ \ \deg 1
\end{align*}
Define a homomorphism $C_{m+1}\rightarrow {s_{m,X^{\bigcdot}}^{\log}}'(1)$ 
given by
\begin{align*}
& 1+J_{D_{m+1}^{i}}\rightarrow{s_{m,X^{i}}^{\log}}'(1)^{0}=J_{D_{m}^{i}} \\
& \ \ \ \ \ a\mapsto\log a\mod p^{m}
\end{align*}
in degree $0$ and
\begin{align*}
& M_{D_{m+1}^{i}}^{\mathrm{gp}}\rightarrow{s_{m,X^{i}}^{\log}}'(1)^{1}=\mathcal{O}_{D_{m}^{i}}\otimes_{\mathcal{O}_{Z_{m}^{i}}}\omega^{1}_{Z_{m}^{i}/W_{m}(k)}\oplus\mathcal{O}_{D_{m}^{i}} \\
& \ \ \ \ b\mapsto(d\log b\mod p^{m}, \, p^{-1}\log b^{p}\varphi(b)^{-1})
\end{align*}
in degree $1$. Note that $\log(b^{p}\varphi(b)^{-1})$ is in $p\mathcal{O}_{D_{m+1}^{i}}\xleftarrow[p]{\sim}\mathcal{O}_{D_{m}^{i}}$ because $b^{p}\varphi(b)^{-1}\in 1+p\mathcal{O}_{D_{m+1}^{i}}$. By composing with $R\theta_{\ast}$ we get a map
\begin{equation}\label{map}
M_{X_{m+1}}^{\mathrm{gp}}\rightarrow {s_{m,X}^{\log}}'(1)[1]\,.
\end{equation}
For any $0\leq n,n',n+n'\leq p-1$ there is a product structure
\begin{equation*}
{s_{m,X^{\bigcdot}}^{\log}}'(n)\otimes{s_{m,X^{\bigcdot}}^{\log}}'(n')\rightarrow{s_{m,X^{\bigcdot}}^{\log}}'(n+n')
\end{equation*}
\cite[\S2.2]{Tsu99}. Applying $R\theta_{\ast}$ gives
\begin{equation}\label{product}
{s_{m,X}^{\log}}'(n)\otimes^{\mathbb{L}}{s_{m,X}^{\log}}'(n')\rightarrow{s_{m,X}^{\log}}'(n+n')\,.
\end{equation}
Together, \eqref{map} and \eqref{product} induce symbols maps
\begin{equation*}
(M_{X_{m+1}}^{\mathrm{gp}})^{\otimes q}\rightarrow\mathcal{H}^{q}({s_{m,X}^{\log}}'(q))
\end{equation*}
for each $q\geq 0$ \cite[(2.2.1)]{Tsu99}. These constructions are independent of the choice of  embedding system and lifting of Frobenius. We have the following
\begin{prop}\cite[Lemma 3.4.11, Proposition 2.4.1]{Tsu99} The symbol map $(M_{X_{m+1}}^{\mathrm{gp}})^{\otimes q}\rightarrow\mathcal{H}^{q}({s_{m,X}^{\log}}'(q))$ is surjective.
\end{prop}

It follows from \cite[Proposition 3.8]{NN16} that the complex $R\Gamma(X,{s_{\bigcdot,X}^{\log}}'(n)\otimes\mathbb{Q})$ is isomorphic to the complex $R\Gamma(X,\mathfrak{S}_{\log,X_{\bigcdot}}(n))$ which we defined before. A crucial point for this comparison is the existence of an isomorphism 
\begin{equation*}
R\Gamma(X/W(k),\mathcal{O}_{\mathrm{cris}}/J^{[n]}_{X/W(k)})_{\mathbb{Q}}\cong R\Gamma_{\mathrm{dR}}(X_{K})/\mathrm{Fil}^{n}
\end{equation*}
which links the log-crystalline cohomology of $(X,M_{X})$ over $W(k)$ equipped with the trivial log-structure to the de Rham cohomology of the generic fibre. This is proven in \cite[Corollary 2.4]{NN16} and is a consequence of Beilinson's comparison (\cite[Theorem 2.1]{NN16}) using derived log de Rham complexes (\cite[(1.9.2)]{Bei13}). It was also proven in \cite[Lemma 2.7]{Lan99} based on the original proof of Kato-Messing \cite[Lemma 4.5]{KM92} for syntomic schemes in the absence of log-structures.

For a $W(k)$-scheme $X$ with semistable reduction let $X_{m},M_{X_{m}},\iota_{m}$ be defined at the beginning of this section. Then $X_{m}$ is equipped with the log-structure $M_{X_{m}}=i_{m}^{\ast}M_{X}$ locally defined by 
\begin{align*}
& \mathbb{N}^{r}\rightarrow\mathcal{O}_{X_{m}} \\
& e_{i}\mapsto\pi_{i}^{(m)}
\end{align*}
if $X$ is locally given by $\mathrm{Spec}\,W(k)[T_{1},\ldots, T_{n}]/(\pi_{1}\cdots\pi_{r}-p)$ and where $\pi_{i}^{(m)}=\pi_{i}\mod p^{m}$. Let $j:X_{m}^{\mathrm{sm}}\hookrightarrow X_{m}$ be the open subscheme of $X_{m}$ such that $X_{m}^{\mathrm{sm}}\rightarrow\mathrm{Spec}\,W_{m}(k)$ is smooth. We consider the log-structure $N_{X_{m}}$ associated to 
\begin{align*}
& \mathbb{N}^{r}\rightarrow\mathcal{O}_{X_{m}} \\
& e_{i}\mapsto g_{i}^{(m)}:=\pi_{i}^{(m)}+\prod_{j\neq i}\pi_{j}^{(m)}\,.
\end{align*}
Then $\underline{N}_{X_{m}}:=\mathrm{Im}(N_{X_{m}})\subset\mathcal{O}_{X_{m}}\cap j_{\ast}(\mathcal{O}^{\ast}_{X_{m}^{\mathrm{sm}}}$.
\begin{rem}\label{1 plus p sequence}
We have an exact sequence (with $X_{1}=Y$)
\begin{equation*}
0\rightarrow U_{1}\rightarrow\underline{N}_{X_{m}}^{\mathrm{gp}}\rightarrow\underline{N}_{Y}^{\mathrm{gp}}\rightarrow 0
\end{equation*}
where $U_{1}:=\langle 1+px\, |\, x\in\mathcal{O}_{X_{m}}\rangle$.
\end{rem}

Define $\mathcal{K}^{\mathrm{Mil}}_{\log,X_{m},n}$ to be the Zariski sheafification of the presheaf on $X_{m}$ given by
\begin{equation*}
U\mapsto (\underline{N}_{X_{m}}(U)^{\mathrm{gp}})^{\otimes n}/I_{m}
\end{equation*}
where $I_{m}$ is the subgroup generated by elements of the form $a\otimes(1-a)$ with $a,1-a\in\underline{N}_{X_{m}}(U)^{\mathrm{gp}}$, those of the form $a\otimes(-a)$ with $a\in\underline{N}_{X_{m}}(U)^{\mathrm{gp}}$, and those of the form $g_{I}^{(m)n_{I}}x\otimes(1-\pi_{I}^{(m)n_{I}}x)$ ranging over subsets $I\subset\{1,\ldots, r\}$, where $g_{I}^{(m)n_{I}}:=\displaystyle\prod_{i\in I}g_{i}^{(m)n_{i}}$ with $n_{i}\geq 0$, $\pi_{I}^{(m)n_{I}}:=\displaystyle\prod_{i\in I}\pi_{i}^{(m)n_{i}}$ with $n_{i}\geq 0$, and $x\in\mathcal{O}_{X_{m}}(U)^{\ast}$. Consider again the surjective symbol map
\begin{equation*}
(M_{X_{m+1}}^{\mathrm{gp}})^{\otimes q}\rightarrow\mathcal{H}^{q}({s_{m,X}^{\log}}'(q))
\end{equation*}
of \cite{Tsu99} and \cite{Kat87}. For $x\in\mathcal{O}_{X_{m+1}}(U)^{\ast}$, $R_{\mathrm{log-syn}}(x)\in H^{1}({s_{m,X}^{\log}}'(1))$ is defined as for $M_{X_{m+1}}^{\mathrm{gp}}$. For $g_{i}^{(m+1)}\in N_{X_{m+1}}(U)$, define $R_{\mathrm{log-syn}}(g_{i}^{(m+1)}):=R_{\mathrm{log-syn}}(e_{i})$ where $e_{i}\in M_{X_{m+1}}$ is the element mapping to $\pi^{(m+1)}_{i}$ under $M_{X_{m+1}}\rightarrow\mathcal{O}_{X_{m+1}}$. This extends to a map
\begin{equation*}
R_{\mathrm{log-syn}}:(\underline{N}_{X_{m+1}}^{\mathrm{gp}})^{\otimes q}\rightarrow\mathcal{H}^{q}({s_{m,X}^{\log}}'(q)_{\mathrm{\acute{e}t}})
\end{equation*}
which factors through
\begin{equation}\label{log-syntomic regulator}
R_{\mathrm{log-syn}}:\mathcal{K}_{\log, X_{m+1}, q}^{\mathrm{Mil}}\rightarrow\mathcal{H}^{q}({s_{m,X}^{\log}}'(q)_{\mathrm{\acute{e}t}})\,.
\end{equation}

\begin{prop}\label{ses prop}
We have an exact sequence
\begin{equation*}
0\rightarrow\frac{p\omega^{q-1}_{X_{m}/W_{m}(k)}}{p^{2}d\omega^{q-2}_{X_{m}/W_{m}(k)}}\rightarrow\mathcal{H}^{q}({s_{m,X}^{\log}}'(q)_{\mathrm{\acute{e}t}})\rightarrow W_{m}\omega^{q}_{Y/k,\log}\rightarrow 0
\end{equation*}
where the second map is defined by
\begin{equation*}
(d\log b_{1}\wedge\cdots\wedge d\log b_{q},s_{\varphi,q}(\{b_{1},\ldots,b_{q}\}))\mapsto d\log\bar{b}_{1}\wedge\cdots\wedge d\log \bar{b}_{q}
\end{equation*} 
for $b_{i}\in M_{X_{m+1}}^{\mathrm{gp}}$, where $s_{\varphi,q}(\{b_{1},\ldots,b_{q}\})$ is defined as in \cite[2.7 pg 208]{Kur98} (see also \cite[Lemma 2.4.6]{Tsu99}), where $\bar{b}_{i}$ is the image of $b_{i}$ in $M_{Y}^{\mathrm{gp}}$ and $d\log$ is the Hyodo-Kato map \cite[(1.1)]{HK94}. The first map is defined as follows: take a lifting of $z=p\omega\in p\omega^{q-1}_{X_{m}/W_{m}(k)}$ in $p\mathcal{O}_{D_{m}}\otimes_{\mathcal{O}_{Z_{m}}}\omega^{q-1}_{Z_{m}/W_{m}(k)}$, say $px\tilde{\omega}$ with $\tilde{\omega}=d\log b_{2}\wedge\cdots\wedge d\log b_{q}$. Then the image of $z$ under the first map is the class of
\begin{equation*}
\left(d\log(\exp px)\tilde{\omega},s_{\varphi,q}(\{\exp(px),b_{2},\ldots,b_{q}\})\right)\,.
\end{equation*}
It is clear that the class of this element is well-defined in $\mathcal{H}^{q}({s_{m,X}^{\log}}'(q))$. Note that, in order to simplify the notation, we omit the index and work with an embedding $X_{m}\hookrightarrow D_{m}$.
\end{prop}
\begin{proof}
We recall that there is an isomorphism between the cohomology of the original log-syntomic complex of Kato and Tsuji and the sheaf of $p$-adic vanishing cycles:
\begin{equation*}
\mathcal{H}^{q}({s^{\log}_{m,X}}'(q))\cong i^{\ast}R^{q}j_{\ast}\mathbb{Z}/p^{m}(q)
\end{equation*}
for $q<p$ (see \cite[Theorem 3.2.2]{Tsu99}) and the sheaf $i^{\ast}R^{q}j_{\ast}\mathbb{Z}/p^{m}(q)$ is generated by symbols, that is the map 
\begin{equation*}
i^{\ast}j_{\ast}\mathcal{O}_{X_{K}}^{\times}\otimes\cdots\otimes i^{\ast}j_{\ast}\mathcal{O}_{X_{K}}^{\times}\rightarrow i^{\ast}R^{q}j_{\ast}\mathbb{Z}/p^{m}(q)
\end{equation*}
defined by taking the cup-product of the boundary map
\begin{equation*}
i^{\ast}j_{\ast}\mathcal{O}_{X_{K}}^{\times}\rightarrow i^{\ast}R^{1}j_{\ast}\mathbb{Z}/p^{m}(1)
\end{equation*}
arising from the Kummer sequence, is surjective. Moreover, $i^{\ast}R^{q}j_{\ast}\mathbb{Z}/p^{m}(q)$ is equipped with a filtration $U^{0}\supset U^{1}\supset\ldots$ such that $U^{0}/U^{1}$ is isomorphic to $W_{m}\tilde{\omega}^{q}_{Y/k,\log}$, the sheaf defined in \eqref{HK sequence} and Lemma \ref{Frobenius invariants Nygaard} (see \cite[Theorem 1.6]{Hyo88}), and $U^{1}$ is generated by symbols $\{i^{\ast}(1+pz),x_{2},\ldots, x_{q}\}$ with $z\in\mathcal{O}_{X}$ and $x_{i}\in i^{\ast}j_{\ast}\mathcal{O}_{X_{K}}^{\times}$, $i=2,\ldots, q$. This shows that the kernel of 
\begin{equation*}
\mathcal{H}^{q}({s_{m,X}^{\log}}'(q)_{\mathrm{\acute{e}t}}))\rightarrow W_{m}\omega^{q}_{Y/k,\log}
\end{equation*}
consists of classes of elements where the first component is of the form
\begin{equation*}
d\log(1+px)\wedge d\log b_{2}\wedge\cdots d\log b_{q}
\end{equation*}
with $x\in\mathcal{O}_{D_{n}}$ and $b_{i}\in M_{D_{n}}^{\mathrm{gp}}$, $i=2,\ldots, q$. This element is the image of $\log(1+p\hat{x})\wedge d\log\hat{b}_{2}\wedge\cdots\wedge d\log\hat{b}_{q}$, (where $\hat{x}$ is the image of $x$ in $\mathcal{O}_{X_{m}}$ and $\hat{b}_{i}$ is the image of $b_{i}$ in $M_{X_{m}}^{\mathrm{gp}}$), which is an element of $p\omega_{X_{m}/W_{m}(k)}^{q-1}$.

We show that the kernel of $p\omega_{X_{m}/W_{m}(k)}^{q-1}\rightarrow\mathcal{H}^{q}({s_{m,X}^{\log}}'(q)_{\mathrm{\acute{e}t}})$ contains $p^{2}d\omega_{X_{m}/W_{m}(k)}^{q-2}$. If $p\tilde{\omega}\in\omega^{q-1}_{D_{m}/W_{m}(k)}$ is a lifting of $p\omega$, then a necessary condition for the image of $p\omega$ to vanish is that $p\tilde{\omega}$ is closed. If $p\tilde{\omega}=pd\log b_{1}\wedge\cdots\wedge d\log b_{q-1}$ then
\begin{align*}
s_{\varphi,q}(\{\exp(p),b_{1},\ldots,b_{q-1}\})
& =\left(\frac{\varphi}{p^{q}}-1\right)p\tilde{\omega} \\
& = \frac{\varphi(p)}{p}\cdot\frac{\varphi}{p^{q-1}}(d\log b_{1}\wedge\cdots\wedge d\log b_{q-1})-pd\log b_{1}\wedge\cdots\wedge d\log b_{q-1} \\
& = (1-p)(d\log b_{1}\wedge\cdots\wedge d\log b_{q-1}) \text{ modulo an exact form}
\end{align*}
hence is not exact. The same argument holds for any other multiple $c\tilde{\omega}$, $c\in W_{m}(k)$. Hence for $p\tilde{\omega}$ to vanish in $\mathcal{H}^{q}({s_{m,X}^{\log}}'(q)_{\mathrm{\acute{e}t}})$ it is necessary that 
\begin{equation*}
p\tilde{\omega}=pdz=pdb_{1}\wedge\frac{db_{2}}{b_{2}}\wedge\cdots\wedge\frac{db_{q-1}}{b_{q-1}}=pdb_{1}\wedge d\log b_{2}\wedge\cdots\wedge d\log b_{q-1}\,.
\end{equation*}
The second component of the image of $p\omega$ is then $s_{\varphi,q}(\{\exp(pb_{1}),b_{1},\ldots,b_{q-1}\})$. In order to decide whether it is a boundary of an element in $\mathcal{O}_{D_{m}}\otimes_{\mathcal{O}_{Z_{m}}}\omega_{Z_{m}/W_{m}(k)}^{q-2}$ it suffices to consider the case $q=2$ (the proof shows that the general case follows from this using the formula for $s_{\varphi,q}$ in \cite{Kur98}). Then
\begin{align*}
s_{\varphi,2}(\{\exp(pb_{1}),b_{1}\})
& =\frac{1}{p}\log\left(\frac{\exp(\varphi(pb_{1}))}{\exp(p^{2}b_{1})}\right)\left(\frac{1}{p}d\log\varphi(b_{1})\right)-\frac{1}{p}\log\frac{\varphi(b_{1})}{b_{1}^{p}}d(pb_{1}) \\
& =(\varphi(b_{1})-pb_{1})\frac{1}{p}d\log\varphi(b_{1})-\frac{1}{p}\log\frac{\varphi(b_{1})}{b_{1}^{p}}d(pb_{1}) \\
& =\frac{1}{p}d\varphi(b_{1})-b_{1}d\log\varphi(b_{1})-\log\frac{\varphi(b_{1})}{b_{1}^{p}}db_{1} \\
\end{align*}
Let $\varphi(b_{1})=b_{1}^{p}+px$. Then the above continues as
\begin{align*}
s_{\varphi,2}(\{\exp(pb_{1}),b_{1}\})
& = \frac{1}{p}d\varphi(b_{1})-b_{1}d\log b_{1}^{p}\left(1+\frac{px}{b_{1}^{p}}\right)-\log\left(1+\frac{px}{b_{1}^{p}}\right)db_{1} \\
& = \frac{1}{p}d\varphi(b_{1})-pdb_{1}-b_{1}d\log\left(1+\frac{px}{b_{1}^{p}}\right)-\log\left(1+\frac{px}{b_{1}^{p}}\right)db_{1} \\
& = \frac{1}{p}d\varphi(b_{1})-pdb_{1}-d\left(b_{1}\log\left(1+\frac{px}{b_{1}^{p}}\right)\right) \\
& = \frac{1}{p}d\varphi(b_{1}) \text{ modulo an exact form}\,.
\end{align*}
Therefore $s_{\varphi,2}(\{\exp(pb_{1}),b_{1}\})$ is exact if $b_{1}=pb_{1}'$ for some $b_{1}'$, which gives $p\tilde{\omega}=p^{2}dz'$, hence $p\omega=p^{2}d\hat{\omega}$ for $\hat{\omega}\in\omega^{q-2}_{Z_{m}/W_{m}(k)}$.

We have shown that $s_{\varphi,2}(\{\exp(pb_{1}),b_{1}\})$ vanishes in $\mathcal{H}^{q}({s_{m,X}^{\log}}'(q))$ if and only if $b_{1}=pb'_{1}$, yielding an injection $p\omega_{X_{m}/W_{m}(k)}^{q-1}/p^{2}d\omega_{X_{m}/W_{m}(k)}^{q-2}\hookrightarrow\mathcal{H}^{q}({s_{m,X}^{\log}}'(q))$ in analogy to the good reduction case considered in \cite{BEK14}. This completes the proof of Proposition \ref{ses prop}.
\end{proof}

We define ${s_{\bigcdot,X}^{\log}}'(n):=\tau_{\leq n}R\epsilon_{\ast}{s_{\bigcdot,X}^{\log}}'(n)_{\mathrm{\acute{e}t}}$ where $\epsilon:(X_{\bigcdot})_{\mathrm{\acute{e}t}}\rightarrow (X_{\bigcdot})_{\mathrm{Zar}}$ is the morphism of sites.
\begin{defn}
Since ${s_{\bigcdot,X}^{\log}}'(n)$ is acyclic in degrees $>n$, we can define the log-motivic pro-complex $\mathbb{Z}_{\log,X_{\bigcdot}}(n)$  in $D_{\mathrm{pro}}(Y_{\mathrm{Zar}})$ via the homotopy cartesian diagram
\begin{equation*}
\begin{tikzpicture}[descr/.style={fill=white,inner sep=1.5pt}]
        \matrix (m) [
            matrix of math nodes,
            row sep=2.5em,
            column sep=1.5em,
            text height=1.5ex, text depth=0.25ex
        ]
        { \mathbb{Z}_{\log,X_{\bigcdot}}(n) & \ & \mathbb{Z}_{\log,Y}(n) \\
        {s_{\bigcdot,X}^{\log}}'(n) & \mathcal{H}^{n}({s_{\bigcdot,X}^{\log}}'(n))[-n] & W_{\bigcdot}\tilde{\omega}_{Y/k,\log}^{n}[-n] \\ };

        \path[overlay,->, font=\scriptsize] 
        
        (m-1-1) edge (m-1-3)
        (m-2-1) edge (m-2-2) 
        (m-1-1) edge (m-2-1)
        (m-1-3) edge node[right]{$d\log$}(m-2-3)
        (m-2-2) edge (m-2-3);
                                        
\end{tikzpicture}
\end{equation*}
where ``$d\log$'' is defined using that $\mathbb{Z}_{\log,Y}(n)$ is acyclic in degrees $>n$ by definition, the map $\lambda:\mathcal{H}^{n}(\mathbb{Z}_{\log,Y}(n))\rightarrow\mathcal{K}_{\log,Y,n}^{\mathrm{Mil}}$, and the map $d\log$ defined after Proposition \ref{theta sequence}.
\end{defn}

We do not quite have a semistable analogue of the fundamental triangle in \cite[Theorem 5.4]{BEK14}. In any case, we have such a triangle by considering $R\Gamma(X,{s_{\bigcdot,X}^{\log}}'(n)_{\mathbb{Q}}(n))$, namely we have an exact triangle \\
\begin{equation}\label{log triangle}
R\Gamma(X,\omega^{<n}_{X_{\bigcdot}/W(k)}\otimes\mathbb{Q}[-1])\rightarrow R\Gamma(X,{s_{\bigcdot,X}^{\log}}'(n)_{\mathbb{Q}})\rightarrow R\Gamma(Y,W_{\bigcdot}\tilde{\omega}_{Y/k,\log}^{n}[-n]\otimes\mathbb{Q})\xrightarrow{+1}
\end{equation}
which is sufficient for proving our main result Theorem \ref{main theorem}. The point is that we have a corresponding triangle for the Nekov\'{a}\v{r}-Nizio\l \ complex $R\Gamma(X,\mathfrak{S}_{\log,X_{\bigcdot}}(n))$.

As in the smooth case we have
\begin{lemma}\label{cohomological support}
The log-motivic pro-complex $\mathbb{Z}_{\log,X_{\bigcdot}}(n)$ is acyclic in degrees $>n$.
\end{lemma}
\begin{proof}
Note that ${s_{\bigcdot,X}^{\log}}'(n)_{\mathbb{Q}}(n)$ and $\mathbb{Z}_{\log,Y}(n)$ are acyclic in degrees $>n$. By the definition of $\mathbb{Z}_{\log,X_{\bigcdot}}(n)$ we have a long exact sequence
\begin{equation*}
\cdots\rightarrow\mathcal{H}^{i}(\mathbb{Z}_{\log,X_{\bigcdot}}(n))\rightarrow\mathcal{H}^{i}({s_{\bigcdot,X}^{\log}}'(n))\oplus\mathcal{H}^{i}(\mathbb{Z}_{\log,Y}(n))\rightarrow \mathcal{H}^{i}(W_{\bigcdot}\tilde{\omega}_{Y/k,\log}^{n}[-n])\rightarrow\cdots\,.
\end{equation*}
Since $\mathcal{H}^{n}({s_{\bigcdot,X}^{\log}}'(n)_{\mathbb{Q}})\rightarrow W_{\bigcdot}\tilde{\omega}^{n}_{Y/k,\log}$ is surjective, $\mathbb{Z}_{\log,X_{\bigcdot}}(n)$ is acyclic in degrees $>n$.
\end{proof}

\begin{prop}\label{pro-motivic k-theory}
Suppose that $k$ is infinite. For each $n\geq 0$ (with $n<p$) there is a canonical isomorphism \begin{equation*}
\mathcal{H}^{n}(\mathbb{Z}_{\log,X_{\bigcdot}}(n))\simeq\mathcal{K}_{\log,X_{\bigcdot},n}^{\mathrm{Mil}}
\end{equation*}
in $\mathrm{Sh}_{\mathrm{pro}}(Y_{\mathrm{Zar}})$. 
\end{prop}
\begin{proof}
The exact sequences
\begin{equation*}
0\rightarrow p\omega_{X_{\bigcdot}/W(k)}^{n-1}/p^{2}d\omega_{X_{\bigcdot}/W(k)}^{n-2}\rightarrow\mathcal{H}^{n}({s_{\bigcdot,X}^{\log}}'(n))\rightarrow W_{\bigcdot}\tilde{\omega}_{Y/k,\log}^{n}\rightarrow 0
\end{equation*}
and 
\begin{equation*}
0\rightarrow\mathcal{H}^{n}(\mathbb{Z}_{\log,X_{\bigcdot}}(n))\rightarrow\mathcal{H}^{n}({s_{\bigcdot,X}^{\log}}'(n))\oplus\mathcal{H}^{n}(\mathbb{Z}_{\log,Y}(n))\rightarrow W_{\bigcdot}\tilde{\omega}_{Y/k,\log}^{n}\rightarrow 0
\end{equation*}
induce the exact sequence at the bottom of the following commutative diagram
\begin{equation*}
\begin{tikzpicture}[descr/.style={fill=white,inner sep=1.5pt}]
        \matrix (m) [
            matrix of math nodes,
            row sep=2.5em,
            column sep=1.3em,
            text height=1.5ex, text depth=0.25ex
        ]
        { 0 & U^{1}\mathcal{K}^{\mathrm{Mil}}_{\log, X_{\bigcdot},n} & \mathcal{K}_{\log,X_{\bigcdot},n}^{\mathrm{Mil}} & \mathcal{K}_{\log,Y,n}^{\mathrm{Mil}} & 0 \\
        0 & p\omega_{X_{\bigcdot}/W(k)}^{n-1}/p^{2}d\omega_{X_{\bigcdot}/W(k)}^{n-2} & \mathcal{H}^{n}(\mathbb{Z}_{\log,X_{\bigcdot}}(n)) & \mathcal{H}^{n}(\mathbb{Z}_{\log,Y}(n)) & 0 \\ };

        \path[overlay,->, font=\scriptsize] 
        (m-1-1) edge (m-1-2)
        (m-2-1) edge (m-2-2)
        (m-1-2) edge (m-1-3)
        (m-2-2) edge (m-2-3)
        (m-2-3) edge (m-2-4) 
        (m-1-2) edge node[right]{$(\ast)$}(m-2-2)
        (m-1-3) edge node[right]{$R_{\mathrm{log-syn}}$}(m-2-3)
        (m-1-3) edge (m-1-4)
        (m-1-4) edge (m-1-5)
        (m-2-4) edge (m-2-5)
        (m-1-4) edge node[right]{$\wr$}(m-2-4)
        ;
                                        
\end{tikzpicture}
\end{equation*}
The right vertical map is the isomorphism in Proposition \ref{K-theory and motivic} and the map $R_{\mathrm{log-syn}}$ is induced by the log-syntomic regulator \eqref{log-syntomic regulator}. We shall show the map $(\ast)$ is an isomorphism. Note that $(\ast)$, which is the restriction of $R_{\mathrm{log-syn}}$, turns out to be the inverse of the exponential map
\begin{equation}\label{Exp}
\mathrm{Exp}:\frac{p\omega^{n-1}_{R_{m}/W_{m}(k)}}{p^{2}d\omega^{n-2}_{R_{m}/W_{m}(k)}}\rightarrow U^{1}\mathcal{K}_{\log, n}^{\mathrm{Mil}}(R_{m})
\end{equation}
induced by $pad\log b_{1}\wedge\cdots\wedge d\log b_{r-1}\mapsto\{\exp(pa),b_{1},\ldots, b_{r-1}\}$, where $R_{m}$ is a local ring on the syntomic scheme $X_{m}/W_{m}(k)$ which is flat. Indeed, the following facts cited in \cite[\S12]{BEK14} also hold for the ring $R_{m}$:
\begin{itemize}
\item[-] $\mathcal{K}^{\mathrm{Mil}}_{\log, Y,n}$ is $p$-torsion free. Indeed, $\mathcal{K}^{\mathrm{Mil}}_{\log, Y,n}$ injects into $\mathcal{K}^{\mathrm{Mil}}_{Y^{\mathrm{sm}},n}$ which is $p$-torsion free.
\item[-] $U^{1}\mathcal{K}_{\log,n}^{\mathrm{Mil}}(R_{m})$ is $p$-primary torsion of finite exponent. The proof using pointy bracket symbols for $K_{2}(R,pR)$ passes over verbatim.
\end{itemize}
The existence of the exponential map also holds more generally for rings satisfying the assumption 2.1 in \cite{Kur98}, so we have
\begin{equation*}
\mathrm{Exp}:\frac{p\omega^{n-1}_{R_{m}/W_{m}(k)}}{p^{2}d\omega^{n-2}_{R_{m}/W_{m}(k)}}\rightarrow \left(\mathcal{K}_{\log, n}^{\mathrm{Mil}}(R_{m})\right)^{\wedge}
\end{equation*}
into the $p$-adic completion. Then steps 1 and 2 in the proof of \cite[Theorem 12.3]{BEK14} carry over to show the existence of \eqref{Exp}. Since \cite[Corollary 1.3]{Kur98} holds for more general rings including $R_{m}$, $\mathrm{Exp}$ vanishes on $p^{2}d\omega^{n-2}_{R_{m}/W_{m}(k)}$. It is clear that $(\ast)$ composed with $\mathrm{Exp}$ is the identity on $p\omega^{n-1}/p^{2}d\omega^{n-2}$, so it remains to show that $\mathrm{Exp}$ is surjective. 

Define $G_{n}=p\omega_{R_{m}/W_{m}(k)}^{n-1}/pd\omega_{R_{m}/W_{m}(k)}^{n-2}$ and define a filtration $U^{\bullet}G_{n}$ by defining $U^{i}G_{n}$ to be the image of $p^{i}\omega_{R_{m}/W_{m}(k)}^{n-1}$ in $G_{n}$. Inductively define subsheaves
\begin{equation*}
0=B_{0}\subset B_{1}\subset\cdots\subset Z_{2}\subset Z_{1}\subset Z_{0}=\omega_{R_{1}/k}^{q}
\end{equation*}
using the inverse Cartier operator $C^{-1}$ by the formulae
\begin{align*}
& B_{1}=d\omega_{R/k}^{q-1} \\
& Z_{1}=\ker\left(d:\omega_{R/k}^{q}\rightarrow\omega_{R/k}^{q+1}\right) \\
& C^{-1}:B_{s}\xrightarrow{\sim}B_{s+1}/B_{1} \\
& C^{-1}:Z_{s}\xrightarrow{\sim}Z_{s+1}/B_{1}
\end{align*}
as in \cite[(1.5)]{Hyo88}. Then the analogue of \cite[I. Proposition 2.2.8]{Ill79} holds: $B_{i}$ is locally generated by sections of the form $x_{1}^{p^{r}}d\log x_{1}\wedge\cdots\wedge d\log x_{q}$, $x_{j}\in N_{Y}$, $0\leq r\leq i-1$. Define a filtration $U^{\bullet}\mathcal{K}_{\log, n}^{\mathrm{Mil}}(R_{m})$ of $\mathcal{K}_{\log, n}^{\mathrm{Mil}}(R_{m})$ by setting $U^{i}\mathcal{K}_{\log, n}^{\mathrm{Mil}}(R_{m})$ to be the subgroup generated by symbols of the form $\{1+p^{i}x_{1},x_{2},\ldots, x_{n}\}$ where $x_{1}\in R_{m}$ and $x_{2},\ldots, x_{n}\in N_{R_{m}}$. Then $U^{1}\mathcal{K}_{\log, n}^{\mathrm{Mil}}(R_{m})=\ker\left(\mathcal{K}_{\log,n}^{\mathrm{Mil}}(R_{m})\rightarrow\mathcal{K}_{\log,n}^{\mathrm{Mil}}(R_{1})\right)$. For each $i\geq 1$, the analogue of \cite[Lemma 2.3.2]{Kur88} holds: the map 
\begin{align*}
\lambda_{i}:
& \ \omega^{n-1}_{R_{1}/k}\rightarrow\mathrm{gr}^{i}\mathcal{K}_{\log,n}^{\mathrm{Mil}}(R_{m}) \\
& ad\log b_{1}\wedge\cdots\wedge d\log b_{n-1}\mapsto \{1+p^{i}\tilde{a},\tilde{b}_{1},\ldots,\tilde{b}_{n-1}\}
\end{align*}
(where $\tilde{a}$ and the $\tilde{b}_{i}$ are liftings of $a$ and the $b_{i}$ to $R_{m}$) annihilates $B_{i-1}$, hence induces a map
\begin{equation*}
\omega^{n-1}_{R_{1}/k}/B_{i-1}\rightarrow\mathrm{gr}^{i}\mathcal{K}_{\log,n}^{\mathrm{Mil}}(R_{m})\,.
\end{equation*}

By the obvious semistable analogue of \cite[I. Corollaire 2.3.14 (b)]{Ill79} (see also \cite[(2.6)]{Hyo88}) we have an isomorphism 
\begin{equation*}
\omega^{n-1}_{R_{1}/k}/B_{i-1}\simeq\mathrm{gr}^{i}G_{n}\,.
\end{equation*}
On the other hand, consider the composite map
\begin{equation*}
\omega^{n-1}_{R_{1}/k}/B_{i-1}\xrightarrow{\lambda_{i}}\mathrm{gr}^{i}\mathcal{K}_{\log,n}^{\mathrm{Mil}}(R_{m})\rightarrow\mathrm{gr}^{i}G_{n}\xrightarrow{\sim}\omega^{n-1}_{R_{1}/k}/B_{i+1}
\end{equation*}
which coincides with the inverse Cartier operator, which is injective. The second arrow is by definition surjective. Since the first map is also surjective, the second map is an isomorphism. Hence $\mathrm{Exp}$ is an isomorphism between $ \frac{p\omega^{n-1}_{R_{m}/W_{m}(k)}}{p^{2}d\omega^{n-2}_{R_{m}/W_{m}(k)}}$ and $U^{1}\mathcal{K}_{\log, n}^{\mathrm{Mil}}(R_{m})$. This completes the proof of Proposition \ref{pro-motivic k-theory}.
\end{proof}

\begin{rem}\label{weight 1 pro remark} Without the assumption that $k$ is infinite, we should replace the logarithmic Milnor $K$-theory pro-sheaf $\mathcal{K}^{\mathrm{Mil}}_{\log,X_{\bigcdot},n}$ with the improved logarithmic Milnor $K$-theory pro-sheaf $\hat{\mathcal{K}}^{\mathrm{Mil}}_{\log,X_{\bigcdot},n}$ along the lines of Remark \ref{refined remark}(ii). With this modification, Proposition \ref{pro-motivic k-theory} also holds when $k$ is finite by the same proof. Notice that Proposition \ref{pro-motivic k-theory} holds in weight $n=1$ without modification, because $\hat{\mathcal{K}}^{\mathrm{Mil}}_{\log,X_{\bigcdot},1}=\mathcal{K}^{\mathrm{Mil}}_{\log,X_{\bigcdot},1}$.
\end{rem}

\begin{prop}\label{weight one pro-complex}
The log-motivic pro-complex of weight one $\mathbb{Z}_{\log,X_{\bigcdot}}(1)$ is quasi-isomorphic to $\underline{N}_{X_{\bigcdot}}^{\mathrm{gp}}[-1]$, hence
\begin{equation*}
\mathbb{H}^{2}_{\mathrm{cont}}(Y,\mathbb{Z}_{\log,X_{\bigcdot}}(1))\cong H^{1}_{\mathrm{Zar}}(X_{\bigcdot},\underline{N}_{X_{\bigcdot}}^{\mathrm{gp}})\,.
\end{equation*}
If $X$ is proper over $\mathrm{Spec}\,W(k)$ then  we have $\mathbb{H}^{2}_{\mathrm{cont}}(Y,\mathbb{Z}_{\log,X_{\bigcdot}}(1))\cong H^{1}_{\mathrm{Zar}}(X,\underline{N}_{X}^{\mathrm{gp}})$ where $\underline{N}_{X}^{\mathrm{gp}}$ will be defined in the proof.
\end{prop}
\begin{proof}
We have a commutative diagram with exact rows
\begin{equation*}
\begin{tikzpicture}[descr/.style={fill=white,inner sep=1.5pt}]
        \matrix (m) [
            matrix of math nodes,
            row sep=2.5em,
            column sep=1.3em,
            text height=1.5ex, text depth=0.25ex
        ]
        { 0 & 1+p\mathcal{O}_{X_{\bigcdot}}[-1] &  N_{X_{\bigcdot}}^{\mathrm{gp}}[-1] & N_{Y}^{\mathrm{gp}}[-1] & 0 \\
        0 & p\mathcal{O}_{X_{\bigcdot}}[-1] & \mathbb{Z}_{\log,X_{\bigcdot}}(1) & \mathbb{Z}_{\log,Y}(1) & 0 \\ };

        \path[overlay,->, font=\scriptsize] 
        (m-1-1) edge (m-1-2)
        (m-2-1) edge (m-2-2)
        (m-1-2) edge (m-1-3)
        (m-2-2) edge (m-2-3)
        (m-2-3) edge (m-2-4) 
        (m-1-2) edge node[right]{$\log$} node[left]{$\simeq$} (m-2-2)
        (m-1-3) edge (m-2-3)
        (m-1-3) edge (m-1-4)
        (m-1-4) edge (m-1-5)
        (m-2-4) edge (m-2-5)
        (m-1-4) edge node[right]{$\simeq$}(m-2-4)
        ;
                                        
\end{tikzpicture}
\end{equation*}
where the left vertical arrow is the $p$-adic logarithm, which is an isomorphism, and the right vertical arrow is from Proposition \ref{Z(1)}. By Lemma \ref{cohomological support} we have  $\mathcal{H}^{i}(\mathbb{Z}_{\log,X_{\bigcdot}}(1))=0$ for all $i\geq 2$. The first map $(d,\frac{\varphi}{p}-1):J_{D^{\bigcdot}_{m}}\rightarrow\tilde{\omega}^{1}_{D^{\bigcdot}_{m}/W_{m}(k)}\oplus\mathcal{O}_{D^{\bigcdot}_{m}}$ in the definition of ${s_{m,X}^{\log}}'$ is injective, so ${s_{m,X}^{\log}}'(1)$ is acyclic in degrees $\neq 1$. Since $\mathcal{H}^{0}(\mathbb{Z}_{\log,Y}(1))=0$ by Corollary \ref{corollary weight one}, we conclude from the sequence in the proof of Lemma \ref{cohomological support} that $\mathcal{H}^{0}(\mathbb{Z}_{\log,X_{\bigcdot}}(1))=0$. The middle vertical arrow is induced from the canonical map (compatible with $W_{\bigcdot}\omega^{1}_{Y,\log}[-1]$) $N_{X_{\bigcdot}}^{\mathrm{gp}}[-1]\rightarrow {s_{\bigcdot,X}^{\log}}'(1)$ (defined in the same way as for $M_{X_{\bigcdot}}^{\mathrm{gp}}[-1]$) and the reduction map $N_{X_{\bigcdot}}^{\mathrm{gp}}\rightarrow N_{Y}^{\mathrm{gp}}$. This proves the first statement of the proposition.

For the second statement, we first need to define $N_{X}^{\mathrm{gp}}$. We will do this locally, so let $U\subset X$ be an open such that $U=\mathrm{Spec}\,W(k)[T_{1},\ldots, T_{n}]/(f_{1}\cdots f_{r}-p)$ and such that the log-structure $M_{X}=j_{\ast}\mathcal{O}_{X_{K}}^{\ast}$ is associated to
\begin{align*}
& \mathbb{N}^{r}\rightarrow\mathcal{O}_{X}(U) \\
&e_{i}\mapsto f_{i}\,.
\end{align*}
Define the log-structure $N_{X}$ by 
\begin{align*}
& \mathbb{N}^{r}\rightarrow\mathcal{O}_{X}(U) \\
&e_{i}\mapsto f_{i}+\prod_{\stackrel{j=1}{j\neq i}}^{r}f_{j}\,.
\end{align*}
Since $X$ is regular, it is integral, so $\mathcal{O}_{X}(U)$ is an integral domain, and $\mathcal{O}_{X}(U)\backslash\{0\}$ is a multiplicative monoid. We obtain a homomorphism of monoids
\begin{equation*}
N_{X}(U)\rightarrow\mathcal{O}_{X}(U)\backslash\{0\}
\end{equation*}
extending to a homomorphism of abelian groups
\begin{equation*}
N_{X}^{\mathrm{gp}}(U)\rightarrow(\mathcal{O}_{X}(U)\backslash\{0\})^{\mathrm{gp}}\,.
\end{equation*}
Define $\underline{N}^{\mathrm{gp}}_{X}$ as the image of $N_{X}^{\mathrm{gp}}$ inside $(\mathcal{O}_{X}\backslash\{0\})^{\mathrm{gp}}$. Then we have the canonical reduction map for each $m$
\begin{equation*}
\underline{N}_{X}^{\mathrm{gp}}\rightarrow\underline{N}_{X_{m}}^{\mathrm{gp}}\,.
\end{equation*}
Note that $\underline{N}_{X}^{\mathrm{gp}}$ is, in general, not contained in $j_{\ast}\mathcal{O}_{X_{K}}^{\ast}$, hence is very different from $M_{X}^{\mathrm{gp}}$.

Now consider the short exact sequence associated to taking continuous cohomology of pro-sheaves:
\begin{equation*}
0\rightarrow{\varprojlim_{m}}^{1}\mathbb{H}^{1}_{\mathrm{Zar}}(Y,\mathbb{Z}_{\log,X_{m}}(1))\rightarrow\mathbb{H}_{\mathrm{cont}}^{2}(Y,\mathbb{Z}_{\log,X_{\bigcdot}}(1))\rightarrow\varprojlim_{m}\mathbb{H}^{2}_{\mathrm{Zar}}(Y,\mathbb{Z}_{\log,X_{m}}(1))\rightarrow 0\,.
\end{equation*}
By the first part of the proposition, the middle entry of the sequence is $H^{1}_{\mathrm{cont}}(Y,N_{X_{\bigcdot}}^{\mathrm{gp}})$. Applying the first part of the proposition to the first and final entries in the sequence yields $\varprojlim_{m}^{1}\mathbb{H}^{1}_{\mathrm{Zar}}(Y,\mathbb{Z}_{\log,X_{m}}(1))\simeq\varprojlim_{m}^{1}H^{0}(Y,\underline{N}_{X_{m}}^{\mathrm{gp}})=0$ (because the system $\{H^{0}(Y,\underline{N}_{X_{m}}^{\mathrm{gp}})\}_{m}$ is Mittag-Leffler), and $\varprojlim_{m}\mathbb{H}^{2}_{\mathrm{Zar}}(Y,\mathbb{Z}_{\log,X_{m}}(1))\simeq\varprojlim_{m}H^{1}_{\mathrm{Zar}}(Y,\underline{N}_{X_{m}}^{\mathrm{gp}})$. In particular, we have $H^{1}_{\mathrm{cont}}(Y,\underline{N}_{X_{\bigcdot}}^{\mathrm{gp}})\simeq\varprojlim_{m}H^{1}_{\mathrm{Zar}}(Y,\underline{N}_{X_{m}}^{\mathrm{gp}})$. Now consider the following commutative diagram with exact rows
 \begin{equation*}
\begin{tikzpicture}[descr/.style={fill=white,inner sep=1.5pt}]
        \matrix (m) [
            matrix of math nodes,
            row sep=2.5em,
            column sep=1.3em,
            text height=1.5ex, text depth=0.25ex
        ]
        { H^{0}(\underline{N}_{Y}^{\mathrm{gp}}) & H^{1}(1+p\mathcal{O}_{\widehat{X}}) & \displaystyle\varprojlim_{m}H^{1}(\underline{N}_{X_{m}}^{\mathrm{gp}}) & H^{1}(\underline{N}_{Y}^{\mathrm{gp}}) & H^{2}(1+p\mathcal{O}_{\widehat{X}}) \\
        H^{0}(\underline{N}_{Y}^{\mathrm{gp}}) & H^{1}(1+p\mathcal{O}_{X}) & H^{1}(\underline{N}_{X}^{\mathrm{gp}}) & H^{1}(\underline{N}_{Y}^{\mathrm{gp}}) & H^{2}(1+p\mathcal{O}_{X}) \\ };

        \path[overlay,->, font=\scriptsize] 
        (m-1-1) edge (m-1-2)
        (m-2-1) edge (m-2-2)
        (m-1-2) edge (m-1-3)
        (m-2-2) edge (m-2-3)
        (m-2-3) edge (m-2-4) 
        (m-2-2) edge (m-1-2)
        (m-2-3) edge (m-1-3)
        (m-1-3) edge (m-1-4)
        (m-1-4) edge (m-1-5)
        (m-2-4) edge (m-2-5)
        (m-2-1) edge node[right]{$=$}(m-1-1)
        (m-2-5) edge (m-1-5)
        (m-2-4) edge node[right]{$=$}(m-1-4)
        ;
                                        
\end{tikzpicture}
\end{equation*}
where $\widehat{X}$ is the formal completion of $X$ along the special fibre. If $X$ is proper over $\mathrm{Spec}\,W(k)$ then the second and fifth vertical arrows in the diagram are isomorphisms by formal GAGA, so the middle arrow is also an isomorphism. That is,
\begin{equation*}
\mathbb{H}_{\mathrm{cont}}^{2}(Y,\mathbb{Z}_{\log,X_{\bigcdot}}(1))\cong H^{1}_{\mathrm{cont}}(Y,\underline{N}_{X_{\bigcdot}}^{\mathrm{gp}})\cong\varprojlim_{m}H^{1}(Y,\underline{N}_{X_{m}}^{\mathrm{gp}})\cong H^{1}(X,\underline{N}_{X}^{\mathrm{gp}})\,.
\end{equation*}
\end{proof}

We now have enough to obtain our main result: a generalisation to the semistable case of ``the formal part'' of the $p$-adic variational Hodge conjecture \`{a} la \cite{BEK14}. In the following we use of the continuous cohomology of pro-complexes, see \cite{Jan88} and \cite[Appendix B]{BEK14}.

\begin{thm}\label{main theorem}
Let $n<p$. Let $X$ be a proper regular flat scheme over $\mathrm{Spec}\,W(k)$ with semistable reduction. Let $z\in\mathbb{H}_{\log-\mathcal{M}}^{2n}(Y,\mathbb{Z}(n))\otimes\mathbb{Q}$. Then its log-crystalline Chern class $c_{\mathrm{HK}}(z)\in H^{n}(Y,W_{\bigcdot}\omega_{Y/k,\log}^{n})\otimes\mathbb{Q}\rightarrow H_{\mathrm{log-cris}}^{2n}(Y/W(k))_{\mathbb{Q}}\simeq H_{\mathrm{dR}}^{2n}(X/W(k))_{\mathbb{Q}}\simeq H_{\mathrm{dR}}^{2n}(X_{K}/K)$ lies in $\mathrm{Fil}^{n}H_{\mathrm{dR}}^{2n}(X_{K}/K)$ if and only if $z$ lifts to $\hat{z}\in\mathbb{H}_{\mathrm{cont}}^{2n}(Y,\mathbb{Z}_{\log,X_{\bigcdot}}(n))\otimes\mathbb{Q}$. 
\end{thm}
\begin{proof}
Using \eqref{log triangle} one derives a commutative diagram of exact triangles
\begin{equation*}
\begin{tikzpicture}[descr/.style={fill=white,inner sep=1.5pt}]
        \matrix (m) [
            matrix of math nodes,
            row sep=2.5em,
            column sep=1.3em,
            text height=1.5ex, text depth=0.25ex
        ]
        { R\Gamma(Y,\omega_{X_{\bigcdot}/W(k)}^{<n}\otimes\mathbb{Q}[-1]) & R\Gamma(Y,\mathbb{Z}_{\log,X_{\bigcdot}}\otimes\mathbb{Q}) & R\Gamma(Y,\mathbb{Z}_{\log,Y}\otimes\mathbb{Q}) & \cdots\\
        R\Gamma(Y,\omega_{X_{\bigcdot}/W(k)}^{<n}\otimes\mathbb{Q}[-1]) & R\Gamma(Y,{s_{\bigcdot,X}^{\log}}'(n)_{\mathbb{Q}}) & R\Gamma(Y,W_{\bigcdot}\tilde{\omega}_{Y/k,\log}^{n}[-n]\otimes\mathbb{Q}) & \cdots \\  };

        \path[overlay,->, font=\scriptsize] 
        (m-1-1) edge (m-1-2)
        (m-1-2) edge (m-1-3)
        (m-1-3) edge node [above]{$+1$} (m-1-4)
        (m-2-1) edge (m-2-2)
        (m-2-2) edge (m-2-3)
        (m-2-3) edge node [above]{$+1$} (m-2-4)
        (m-1-1) edge node [right]{$=$}(m-2-1)
        (m-1-2) edge (m-2-2)
        (m-1-3) edge (m-2-3)
        ;
                                        
\end{tikzpicture}
\end{equation*}
This follows from \cite[Lemma 1.4.4]{Nee01}. From this we have the top two rows of the following commutative diagram
\begin{equation*}
\begin{tikzpicture}[descr/.style={fill=white,inner sep=1.5pt}]
        \matrix (m) [
            matrix of math nodes,
            row sep=2.5em,
            column sep=1.3em,
            text height=1.5ex, text depth=0.25ex
        ]
        { \mathbb{H}_{\mathrm{cont}}^{2n}(\mathbb{Z}_{\log,X_{\bigcdot}}(n))_{\mathbb{Q}} & \mathbb{H}^{2n}(\mathbb{Z}_{\log,Y}(n))_{\mathbb{Q}} & \mathbb{H}_{\mathrm{cont}}^{2n}(\omega^{<n}_{X_{\bigcdot}/W(k)})_{\mathbb{Q}} \\
        \mathbb{H}_{\mathrm{cont}}^{2n}({s_{\bigcdot,X}^{\log}}'(n)_{\mathbb{Q}}) & H_{\mathrm{cont}}^{n}(W_{\bigcdot}\tilde{\omega}_{Y/k,\log}^{n})_{\mathbb{Q}} & \mathbb{H}_{\mathrm{cont}}^{2n}(\omega_{X_{\bigcdot}/W(k)}^{<n})_{\mathbb{Q}} \\ 
        \ & \mathbb{H}_{\mathrm{cont}}^{2n}(W_{\bigcdot}\tilde{\omega}_{Y/k}^{\bullet})_{\mathbb{Q}} & H^{2n}_{\mathrm{dR}}(X_{K}/K)/\mathrm{Fil}^{n} \\ };

        \path[overlay,->, font=\scriptsize] 
        (m-1-1) edge (m-1-2)
        (m-1-2) edge (m-1-3)
        (m-2-1) edge (m-2-2)
        (m-2-2) edge (m-2-3)
        (m-3-2) edge (m-3-3)
        (m-1-1) edge (m-2-1)
        (m-1-2) edge (m-2-2)
        (m-1-3) edge node[right]{$=$} (m-2-3)
        (m-2-2) edge (m-3-2)
        (m-2-3) edge node[right]{$\wr$} (m-3-3)
        ;
                                        
\end{tikzpicture}
\end{equation*}
The commutativity of the right hand side is proven in the same way as \cite[Theorem 6.1]{BEK14}. We see from this diagram that $z\in H^{2n}(\mathbb{Z}_{\log,Y}(n))_{\mathbb{Q}}$ lifts to $\mathbb{H}_{\mathrm{cont}}^{2n}(\mathbb{Z}_{\log,X_{\bigcdot}}(n))_{\mathbb{Q}}$ if and only if its Chern class $c_{\mathrm{HK}}(z)$ is in $\mathrm{Fil}^{n}H^{2n}(X_{K}/K)$ under the Hyodo-Kato isomorphism.
\end{proof}
\begin{rem}\label{Yamashita remark}
Although we do not reprove Yamashita's result for the logarithmic Picard group \cite[\S3]{Yam11}, we point out that the $p$-adic deformation theory of both $H^{1}(Y,\underline{N}_{Y}^{\mathrm{gp}})$ and $\mathrm{Pic}^{\log}(Y)$ coincide. We have exact sequences 
\begin{equation*}
1\rightarrow 1+p\mathcal{O}_{X_{\bigcdot}}\rightarrow M_{X_{\bigcdot}}^{\mathrm{gp}}\rightarrow M_{Y}^{\mathrm{gp}}\rightarrow 1
\end{equation*}
and 
\begin{equation*}
1\rightarrow 1+p\mathcal{O}_{X_{\bigcdot}}\rightarrow\underline{N}_{X_{\bigcdot}}^{\mathrm{gp}}\rightarrow\underline{N}_{Y}^{\mathrm{gp}}\rightarrow 1
\end{equation*}
and hence the obstruction to lifting (rational) $H^{1}$-cohomology classes from characteristic $p$ to characteristic $0$ lies in $H^{2}(X,\mathcal{O}_{X_{\bigcdot}})\otimes\mathbb{Q}$ in both cases.
\end{rem}

\end{document}